\documentclass[final,3p,times,10pt]{elsarticle}
\usepackage[centertags]{amsmath}
\usepackage{amsfonts,amssymb,amsthm,epstopdf,newlfont,graphicx}
\usepackage{fixltx2e,booktabs,bm,enumitem,setspace,empheq,cancel}
\usepackage{algorithm}
\usepackage[short]{optidef}
\usepackage{algpseudocode}
\usepackage{caption,subcaption}
\usepackage{threeparttable,multirow}
\usepackage[T1]{fontenc}
\usepackage{babel}
\setcitestyle{square,comma,numbers,sort&compress}
\usepackage{appendix}
\usepackage[colorlinks=true]{hyperref}
\newcommand\hcancel[2][black]{\setbox0=\hbox{$#2$}%
\rlap{\raisebox{.25\ht0}{\textcolor{#1}{\rule{0.7\wd0}{0.75pt}}}}#2} 
\newcommand\hcancelt[2][black]{\setbox0=\hbox{$#2$}%
\rlap{\raisebox{.25\ht0}{\textcolor{#1}{\hspace{0.3mm}\rule{0.7\wd0}{0.75pt}}}}#2} 
\biboptions{comma,round,sort&compress}
\newtheorem{thm}{Theorem}[section]

\newtheorem{cor}{Corollary}[section]

\newtheorem{rem}{Remark}[section]
\theoremstyle{definition}

\numberwithin{algorithm}{section}
\numberwithin{equation}{section}
\renewcommand{\theequation}{\thesection.\arabic{equation}}
\def\simgt{\,\hbox{\lower0.6ex\hbox{$>$}\llap{\raise0.3ex\hbox{$\sim$}}}\,}
\def\simlt{\,\hbox{\lower0.6ex\hbox{$<$}\llap{\raise0.3ex\hbox{$\sim$}}}\,}
\def\simgteq{\,\hbox{\lower0.6ex\hbox{$\ge$}\llap{\raise0.6ex\hbox{$\sim$}}}\,}
\def\simlteq{\,\hbox{\lower0.6ex\hbox{$\le$}\llap{\raise0.6ex\hbox{$\sim$}}}\,}
\def\applteq{\,\hbox{\lower0.6ex\hbox{$\le$}\llap{\raise0.8ex\hbox{$\approx$}}}\,}
\def\applt{\,\hbox{\lower0.6ex\hbox{$<$}\llap{\raise0.5ex\hbox{$\approx$}}}\,}
\DeclareMathAlphabet\mathbfcal{OMS}{cmsy}{b}{n}

\DeclareMathOperator*{\ind}{ind}
\DeclareMathOperator*{\indmax}{indmax}
\DeclareMathOperator*{\indmin}{indmin}

\let\emptyset\varnothing
\makeatletter
\def\user@resume{resume}
\def\user@intermezzo{intermezzo}
\newcounter{previousequation}
\newcounter{lastsubequation}
\newcounter{savedparentequation}
\renewenvironment{subequations}[1][]{%
      \def\user@decides{#1}%
      \setcounter{previousequation}{\value{equation}}%
      \ifx\user@decides\user@resume 
           \setcounter{equation}{\value{savedparentequation}}%
      \else  
      \ifx\user@decides\user@intermezzo
           \refstepcounter{equation}%
      \else
           \setcounter{lastsubequation}{0}%
           \refstepcounter{equation}%
      \fi\fi
      \protected@edef\theHparentequation{%
          \@ifundefined {theHequation}\theequation \theHequation}%
      \protected@edef\theparentequation{\theequation}%
      \setcounter{parentequation}{\value{equation}}%
      \ifx\user@decides\user@resume 
           \setcounter{equation}{\value{lastsubequation}}%
         \else
           \setcounter{equation}{0}%
      \fi
      \def\theequation  {\theparentequation  \alph{equation}}%
      \def\theHequation {\theHparentequation \alph{equation}}%
      \ignorespaces
}{%
  \ifx\user@decides\user@resume
       \setcounter{lastsubequation}{\value{equation}}%
       \setcounter{equation}{\value{previousequation}}%
  \else
  \ifx\user@decides\user@intermezzo
       \setcounter{equation}{\value{parentequation}}%
  \else
       \setcounter{lastsubequation}{\value{equation}}%
       \setcounter{savedparentequation}{\value{parentequation}}%
       \setcounter{equation}{\value{parentequation}}%
  \fi\fi
  \ignorespacesafterend
}
\makeatother
\newcommand{\C}[1]{\mathcal{#1}}

\newcommand{\F}[1]{\mathbf{#1}}
\newcommand{\FR}[1]{\mathfrak{#1}}
\newcommand{\MB}[1]{\mathbb{#1}}
\newcommand{\MBS}{\MB{S}}
\newcommand{\MBR}{\mathbb{R}}
\newcommand{\MBG}{\mathbb{G}}
\newcommand{\MBRzer}{\MBR_0}
\newcommand{\MBRzerP}{\MBRzer^+}
\newcommand{\MBZ}{\mathbb{Z}}
\newcommand{\MBZP}{\MBZ^+}
\newcommand{\MBZzer}{\MBZ_0}
\newcommand{\MBZzerP}{\MBZzer^+}
\newcommand{\MBZe}{\MBZ_e}
\newcommand{\MBZeP}{\MBZe^+}
\newcommand{\MBZzereP}{\MBZ_{0,e}^+}
\newcommand{\MBZOP}{\MBZ_{o}^+}
\newcommand{\MBT}{\mathbb{T}}
\newcommand{\MBJ}{\mathbb{J}}
\newcommand{\MBF}{\mathfrak{F}}
\newcommand{\MBC}{\mathfrak{C}}
\newcommand{\MBK}{\mathfrak{K}}
\newcommand{\Fthe}{\F{\Theta}}
\newcommand{\canczer}[1]{#1_{\hcancel{0}}}
\newcommand{\cancbra}[1]{\hcancel{[}#1\hcancelt{]}}
\newcommand{\sumd}{\sideset{}{'}}
\newcommand{\bs}{\bar s}
\newcommand{\bu}{\bar u}
\newcommand{\bms}{\bm{s}}
\newcommand{\bmu}{\bm{u}}
\newcommand{\Sin}{s_{\text{in}}}
\newcommand{\bmzer}{\bm{\mathit{0}}}
\newcommand{\bmone}{\bm{\mathit{1}}}

\newcommand{\bmX}{\bm{X}}

\newcommand{\foralla}{\,\forall_{\mkern-6mu a}\,}
\newcommand{\foralle}{\,\forall_{\mkern-6mu e}\,}
\newcommand{\foralls}{\,\forall_{\mkern-6mu s}\,}
\newcommand{\Def}[1]{\text{Def}\left(#1\right)}
\newcommand{\Int}[1]{\text{int}\left(#1\right)}
\usepackage{mathtools}
\mathtoolsset{showonlyrefs}
\makeatletter
\def\BState{\State\hskip-\ALG@thistlm}
\makeatother
\MakeRobust{\eqref}
\errorcontextlines\maxdimen

\makeatletter
    \newcommand*{\algrule}[1][\algorithmicindent]{\makebox[#1][l]{\hspace*{.5em}\thealgruleextra\vrule height \thealgruleheight depth \thealgruledepth}}%
\newcommand*{\thealgruleextra}{}
\newcommand*{\thealgruleheight}{.75\baselineskip}
\newcommand*{\thealgruledepth}{.25\baselineskip}

\newcount\ALG@printindent@tempcnta
\def\ALG@printindent{%
    \ifnum \theALG@nested>0% is there anything to print
        \ifx\ALG@text\ALG@x@notext% is this an end group without any text?
        \else
            \unskip
            \addvspace{-1pt}% FUDGE to make the rules line up
            \ALG@printindent@tempcnta=1
            \loop
                \algrule[\csname ALG@ind@\the\ALG@printindent@tempcnta\endcsname]%
                \advance \ALG@printindent@tempcnta 1
            \ifnum \ALG@printindent@tempcnta<\numexpr\theALG@nested+1\relax% can't do <=, so add one to RHS and use < instead
            \repeat
        \fi
    \fi
    }%
\usepackage{etoolbox}
\patchcmd{\ALG@doentity}{\noindent\hskip\ALG@tlm}{\ALG@printindent}{}{\errmessage{failed to patch}}
\makeatother

\newbox\statebox
\newcommand{\myState}[1]{%
    \setbox\statebox=\vbox{#1}%
    \edef\thealgruleheight{\dimexpr \the\ht\statebox+1pt\relax}%
    \edef\thealgruledepth{\dimexpr \the\dp\statebox+1pt\relax}%
    \ifdim\thealgruleheight<.75\baselineskip
        \def\thealgruleheight{\dimexpr .75\baselineskip+1pt\relax}%
    \fi
    \ifdim\thealgruledepth<.25\baselineskip
        \def\thealgruledepth{\dimexpr .25\baselineskip+1pt\relax}%
    \fi
    \State #1%
    \def\thealgruleheight{\dimexpr .75\baselineskip+1pt\relax}%
    \def\thealgruledepth{\dimexpr .25\baselineskip+1pt\relax}%
}
\begin{document}
\begin{frontmatter}
\title{New Optimal Periodic Control Policy for the Optimal Periodic Performance of a Chemostat Using a Fourier-Gegenbauer-Based Predictor-Corrector Method}
%A Filtered Predictor-Corrector Fourier Pseudospectral Integration Method for the Optimal Periodic Control of the Chemostat with Contois Growth Function
%\author[assiut,fahd,IRC]{Kareem T. Elgindy\corref{cor1}}
\author[assiut]{Kareem T. Elgindy\corref{cor1}}
\ead{kareem.elgindy@gmail.com}
\cortext[cor1]{Corresponding author}
\address[assiut]{Mathematics Department, Faculty of Science, Assiut University, Assiut 71516, Egypt}

\begin{abstract}
In its simplest form, a chemostat consists of microorganisms or cells that grow continually in a specific phase of growth while competing for a single limiting nutrient. Under certain conditions of the cell growth rate, substrate concentration, and dilution rate, the theory predicts and numerical experiments confirm that a periodically operated chemostat exhibits an ``overyielding'' state in which the performance becomes higher than that at steady-state operation. In this paper, we show that an optimal periodic control policy for maximizing chemostat performance can be accurately and efficiently derived numerically using a novel class of integral pseudospectral (IPS) methods and adaptive $h$-IPS methods composed through a predictor-corrector algorithm. New formulas for the construction of Fourier pseudospectral (PS) integration matrices and barycentric-shifted Gegenbauer (SG) quadratures are derived. A rigorous study of the errors and convergence rates of SG quadratures, as well as the truncated Fourier series, interpolation operators, and integration operators for nonsmooth and generally $T$-periodic functions,  is presented. We also introduce a novel adaptive scheme for detecting jump discontinuities and reconstructing a piecewise analytic function from PS data. An extensive set of numerical simulations is presented to support the derived theoretical foundations.
\end{abstract}
\begin{keyword}
Adaptive method \sep Chemostat model \sep Fourier interpolation \sep Gegenbauer polynomials \sep $h$-integral pseudospectral \sep Integration matrix \sep Optimal control \sep Predictor-Corrector \sep Pseudospectral method.
\end{keyword}
\end{frontmatter}
\section{Introduction}
\label{Int} 
A chemostat is a laboratory bioreactor in which a microbial culture in a well-stirred culture medium is continuously supplied with nutrients at a fixed/variable rate, while an equal flow of the culture liquid containing microorganisms and nutrients is continuously drained from the culture vessel, so that the vessel retains a constant volume of culture at all times. In this manner, a chemostat enables the experimental control of cell growth rates within limits in a well-defined and controlled environment, and the microbial cell external environment remains constant. The control of the growth rate in this manner allows to (i) optimize the production of specific microbial products like ethanol and antibiotics, (ii) optimize the production of L-leucine used for blood sugar and energy levels regulations, the growth and repairment of bones and muscles and healing of wounds, (iii) facilitate the study of nutrient limitation and the microorganisms growth in natural ecological environments such as rivers and lakes, (iv) analyzing the complex interactions between distinct biological populations, (v) waste-water treatment by decomposing harmful substances in waste-water to improve the water quality, etc. \cite{ziv2013use,MAIER201537,xu2022bifurcation,
eliasson2000anaerobic,gray1980production,raatz2018one}. For these reasons, chemostats are of great theoretical and practical value in applied science and industry. 

One of the main challenges in operating a chemostat is the maintenance of the required growth conditions of microorganisms. There are three variables that can be manipulated in a chemostat: the dilution rate (i.e., the feeding rate) and the input concentrations of the substrate and biomass. Assuming that the latter two are given, the flow rate can be controlled and a constant substrate concentration can be maintained. However, it is well known, since the discovery of \citet{douglas1966unsteady}, that the performance of an unsteady-state (dynamically or periodically) operated biochemical reactor is sometimes superior to that obtained under conventional steady-state\footnote{At a steady state, the dilution rate is equal to the specific growth rate of the cell population; thus, the experimenter can force the cells to grow at a desired rate.} operation in the sense that the conversion can be increased by cycling one or more inputs, stirring many research works in that direction for more than half a century; see \cite{butler1985mathematical,abulesz1987periodic,kumar1993periodic,petkovska2010fast,wolkowicz1998n,peng2000global,
wang2016periodic,bayen2018optimal,bayen2020improvement}, and the references therein. An important case of periodic operation is when the steady-state condition is chosen arbitrarily, and the time-averaged substrate concentration is equal to the substrate concentration value during steady-state operation; see \cite{bailey1974periodic,renken1984unsteady}. Another perhaps more interesting class of periodic operations is when the time-averaged substrate concentration is less than that at the steady-state operation; such systems were termed ``overyielding'' systems \cite{caraballo2015dynamics}.

The success of optimizing the performance of a chemostat depends on the availability of a representative mathematical model of the process and the choice of an appropriate numerical and optimization routine. For periodically operated chemostat systems, the process can be described using a finite horizon optimal control (OC) problem, and chemostat performance can be maximized by determining the optimal periodic control under certain conditions \cite{cogan2016optimal,bayen2018optimal}. When the dilution rate is used as the control variable, \citet{bayen2018optimal} showed through Pontryagin Maximum Principle that the best periodic control for a chemostat exhibiting an overyielding state is ``bang-bang with even switching times.'' However, the derived mathematical model was solved numerically by using BOCOP\footnote{BOCOP is an open source toolbox for OC; see \url{https://www.bocop.org/}.}, which approximates the OC problem (OCP) by a finite dimensional nonlinear programming problem (NLP) using a time discretization. The NLP problem is then solved by IPOPT\footnote{IPOPT is an open source software package for large-scale nonlinear optimization; see \url{https://coin-or.github.io/Ipopt/}.} software. It is important to note that IPOPT is generally not suitable for solving the OCP under study for at least two reasons: (i) IPOPT is a smooth optimization solver that is designed to exploit the first- and second-derivative information if provided and approximates them using quasi-Newton methods, specifically using a BFGS update, if this derivative information is not provided. For OCPs where some/all of the state or control variables are discontinuous, the objective function of the reduced NLP becomes generally discontinuous, and the gradient information obtained by the solver becomes generally null or highly inaccurate in the neighborhoods of the discontinuities, causing poor approximations to the optimal solutions of the OCP. In particular, the imprecise gradient information gathered during the numerical optimization procedure does not provide useful information to the IPOPT solver, or any smooth optimization solver in general, which renders its implementation less rewarding and inefficient. Therefore, the use of IPOPT and BOCOP is generally limited for building OC policies, where OCPs often include discontinuous state or control variables. It is interesting to note that BOCOP was later deemed less accurate than other traditional optimization solvers by the same authors \cite{bayen2020improvement}. In addition, some recent research works have further acknowledged the inconsistency of the IPOPT solver for solving OCPs exhibiting discontinuous solutions; cf. \cite{bouchet2021derivative,mork2022nonlinear}. (ii) For IPOPT to be effective compared with other small/medium-scale optimization solvers, one needs to discretize the OCP at a large time mesh grid because IPOPT was mainly written for large-scale problems with up to million variables and constraints; for such large problems, it is assumed that the derivative matrices are sparse. However, the OCP under study can be converted into a small/medium-scale NLP, which can be solved much faster using standard small/medium-scale NLP solvers, and the gathered control data after this stage can be quickly analyzed using a novel smart algorithm that can rapidly recover the control variables with excellent accuracy. The corresponding state variables can be constructed accurately later with rapid convergence by solving the discrete dynamical system equations at a certain set of collocation points, as will be presented later in Section \ref{sec:SRMay221}.

The above arguments motivated us to explore new optimal periodic control policies to improve chemostat performance in light of a more accurate, robust, and efficient numerical method. In our work, we consider the chemostat of a single-reaction model in which the growth rate\footnote{A clear exposition of a broad class of growth rate functions exhibiting a wide range of outlooks on the subject can be found in the book of \citet{moser2012bioprocess}.} of the microorganisms is given by a Contois expression, which was introduced by \citet{contois1959kinetics} to model the growth of Aerobacter aerogenes and was often used later to model the growth of biomass in wastewater containing biodegradable organic materials \cite{alqahtani2011analysis,bayen2020improvement}. We sought to numerically determine the optimal periodic dilution rate associated with a periodically varying substrate concentration that can optimize the performance of an overyielding chemostat in terms of the time-averaged substrate concentration over a given finite horizon. To this end, we explored the possibility of applying two attractive classes of methods in a unified composite approach. The first class is Fourier integral pseudospectral (FIPS) methods\footnote{FIPS methods are aka ``Fourier pseudospectral (FPS) integration methods'' and can also be termed ``nodal Fourier integration methods.''} where the periodic solutions are represented in terms of grid point values using interpolants. Integral pseudospectral (IPS) methods are robust variants of the popular pseudospectral (PS) methods, in which an initial step of reformulating the dynamical system equations in their integral form is required before the collocation phase starts. Integral reformulation can be performed by either direct integration of the dynamical system equations if they have constant coefficients, or by approximating the solution's highest-order derivative involved in the problem by a nodal finite series in terms of its grid point values, and then solving for those grid point values before successively integrating back in a stable manner to obtain the desired solution grid point values \cite{ElgindyHareth2022a}. Some of the advantages of FIPS methods inherited from FPS methods include: (i) their ability to furnish exponential convergence rates when the problem exhibits sufficiently smooth solutions, (ii) the nodal representation of the solution is extremely useful because its values are immediately available at the collocation points once the full discretization is implemented, whereas Fourier series integration (FSI) methods\footnote{FSI methods can also be termed ``modal Fourier integration methods.''} require a further step of computing the modal approximation after calculating the Fourier coefficients, and (iii) although FPS integration (FPSI) methods often introduce an aliasing error that does not exist in FSI methods, we prefer the former methods because of the Discrete Fourier Transform (DFT) pair, which allows us to rapidly transform from the set of function values at equally spaced points to the set of interpolation coefficients using Fast Fourier Transform (FFT), and vice versa, instead of computing $N$ integrals to determine the Fourier series coefficients. A bonus advantage of FIPS methods over usual FPS methods is manifested in the integral reformulation strategy imposed by the former, which avoids the degradation of precision often caused by numerical differentiation processes \cite{elgindy2020high,dahy2021high}. 

For problems with non-smooth solutions, the IPS methods lose their exponential convergence virtue, and the class of $h$-IPS methods comes into play as a better choice because of their ability to recover the discontinuous/non-smooth solutions with high accuracy via the decomposition of the solution interval into smaller mesh intervals or elements ($h$-refinement), and approximating the restricted solution on each element with a finite, nodal expansion series in terms of the solution grid point values by means of interpolation \cite{Elgindy20172}. For discontinuous solutions with unknown discontinuities, adaptive strategies are desirable to determine which elements need to be refined in advance before the collocation process occurs. This prevents excessive and blind divisions of the solution domain using composite-grid discretizations, which are computationally expensive, time consuming, and often produce poor levels of accuracy compared to the former adaptive strategy \cite{elgindy2020high}. Because the controller of the problem under study is discontinuous with unknown jump discontinuities, as proven in \cite{bayen2018optimal}, an FIPS approximation of the controller suffers from the Gibbs phenomenon, which appears in the form of over- and undershoots around the jump discontinuities. While Gibbs phenomenon is generally considered a demon that needs to be cast out, we shall demonstrate later that it is rather ``a blessing,'' in view of the current work, that can be constructively used to set up a robust adaptive algorithm. In particular, the over- and undershoots developed near a discontinuity in the event of a Gibbs phenomenon provide an excellent means of detecting one. This adaptive scheme, together with the $h$-IPS method, can be combined with the FIPS method to solve the OCP accurately and efficiently.

In light of the above arguments, we propose a novel composite class of IPS methods and adaptive $h$-IPS methods composed through a predictor-corrector algorithm. In the prediction step, the composite method applies a direct IPS method, in which the OCP in integral form is initially collocated in the Fourier physical space. In the correction step, the composite method carries out an adaptive $h$-IPS method in which the integrated dynamical system equation is collocated in the shifted Gegenbauer (SG) physical space after splitting the time domain into smaller elements while allowing the SG interpolant degree to increase on each element, as desired. For these reasons we coin the proposed method with the name ``Fourier-Gegenbauer-based predictor-corrector composite IPS and adaptive $h$-IPS method,'' which we prefer to abbreviate simply by the ``Fourier-Gegenbauer-based predictor-corrector (FG-PC) method.'' To the best of our knowledge, this study introduces the first unified procedure combining the IPS and adaptive $h$-IPS methods to derive optimal periodic control policies for periodically operated biochemical reactors. The proposed FG-PC method is not only unique in its kind, but it can also produce OC policies that are much more efficient than those obtained by traditional methods. For example, we show later in Section \ref{sec:SRMay221} that the FG-PC method can cleverly set up an optimal periodic control policy that can optimize the performance of an overyielding chemostat by reducing the required time-averaged substrate concentration over a given finite horizon by approximately 57\% of the amount recorded in \cite{bayen2018optimal} over the same time period and under the same parameter settings. Furthermore, the FG-PC method can convert the OCP into an NLP of small/medium scale in the initial stage, which can be treated rapidly using standard small/medium scale NLP solvers, and the noisy data gathered from this stage are filtered out using a smart algorithm that is both accurate and efficient, thus alleviating the main drawbacks of BOCOP software.

The remainder of this paper is organized as follows. In the next section, we provide some preliminary notations to be used in this study. In Section \ref{sec:PS1}, the mathematical model under study is presented. In Section \ref{sec:NDOnew1}, we present the numerical discretization operators used to discretize the mathematical model. In particular, Section \ref{sec:FPIMIRF1} presents new formulas for the construction of FPSI matrices that are superior to those obtained earlier by \citet{elgindy2019high} in terms of accuracy, speed, and computational complexity. Moreover, in Section \ref{subsec:BSGQMay20221}, we discuss how to accurately evaluate definite integrals of reconstructed piecewise analytic functions from FPS data and derive new formulas for constructing barycentric SG quadratures. The proposed FG-PC method is described in Section \ref{sec:FPCFPI1}. The simulation results are presented in Section \ref{sec:SRMay221} followed by conclusions, remarks, and future work in Section \ref{sec:conc}. In \ref{subsec:CRFNSPF1}, we study the errors and convergence rates of the truncated Fourier series, interpolation operators, and integration operators for non-smooth and generally $T$-periodic functions. In \ref{sec:BOFIAAJD1}, we study the behavior of the Fourier interpolants at jump discontinuities. \ref{subsubsub:ECABSGQ1} presents a study on the barycentric SG quadrature errors and convergence. \ref{sec:DTD} and \ref{sec:DTD2} present a discussion and practical prescription of a novel edge detection strategy for detecting jump discontinuities and reconstructing a piecewise analytic function with high accuracy from FPS data. Two computational algorithms for the fast, accurate, and economic construction of FPSI matrices and the reconstruction of an approximate piecewise analytic function from the FPS data are described in \ref{app:Alg1}. The computational complexity and speed of constructing the FPSI matrices are investigated in \ref{subsec:CCAA1}. The efficient and stable computation of the SG matrices necessary for constructing SG quadratures is discussed in \ref{sss:ECSGM1}. 

\section{Preliminary Notations}
\label{sec:PN}
The following notations are used throughout this paper; most of them are new to the mathematical community, but we hope the reader will become more familiar with them after carefully reading this section.\\

\noindent\textbf{Logical Symbols.} $\forall, \foralla, \foralle$, and $\foralls$ stand for the phrases ``for all'', ``for any'', ``for each'',  and ``for some'', respectively. The notations $f \in \Def{\F{\Omega}}$ and $f \in C^{k}(\F{\Omega})$ mean $f$ is defined on the set $\F{\Omega}$ and $f$ has $k$ continuous derivatives on the set $\F{\Omega}\,\foralla$ function $f$, in respective order.\\[0.5em]
\textbf{Set and List Notations.}  $\Int{\F{\Omega}}$ stands for the interior of a set $\F{\Omega}$. The symbols $\MBC, \MBF, \MBZ, \canczer{\MBZ}, \MBZP, \MBZzerP, \MBZeP, \MBZzereP, \MBZOP, \MBR, \canczer{\MBR}$, and $\MBRzerP$ denote the sets of all complex-valued functions, all real-valued functions, integers, non-zero integers, positive integers, non-negative integers, positive even integers, non-negative even integers, positive odd integers, real numbers, non-zero real numbers, and non-negative real numbers, respectively. The notations $i:j:k$ or $i(j)k$ indicate a list of numbers from $i$ to $k$ with increment $j$ between numbers, unless the increment equals one where we use the simplified notation $i:k$. For example, $0:0.5:2$ simply means the list of numbers $0, 0.5, 1, 1.5$, and $2$, while $0:2$ means $0, 1$, and $2$. The set of any numbers $y_1, y_2, \ldots, y_n$ is represented by $\{y_{1:n}\}$. The list of any sets $\F{\Omega}_1, \F{\Omega}_2\ldots, \F{\Omega}_n$ is represented by $\F{\Omega}_{1:n}\,\foralla n \in \MBZP$. We define $\MBJ_n = \{0:n-1\}$ and $\MBJ_n^+ = \MBJ_n \cup \{n\}\,\foralla n \in \MBZP$; moreover, $\MB{K}_N = \{-N/2:N/2\}, \MB{K}'_N = \MB{K}\backslash\{N/2\}$, and $\MB{K}''_N = \MB{K}\backslash\{\pm N/2\} \foralla N \in \MBZeP$. $\MBT_T$ is the space of $T$-periodic, univariate functions $\foralla T \in \MBR^+$. Also, $\MBS_n = \left\{x_{n,0:n-1}\right\}$ and $\MBS_n^+ = \MBS_n \cup \{x_{n,n}\}$ are the sets of $n$- and $(n+1)$- equally-spaced points such that $x_{n,j} = T j/n\, \forall j \in \MBJ_n$ and $j \in \MBJ_n^+$, respectively.\\[0.5em]
\textbf{Function Notations.} For convenience, we shall denote $g(x_{N,n})$ by $g_n \foralla g \in \MBC$, unless stated otherwise. Moreover, if a set $\F{\Omega}$ is partitioned into a number of subsets $\F{\Gamma}_{1:n}\,\foralls n \in \MBZP$, the notation ${}_kg$ indicates the restriction of $g$ to ${\F{\Gamma}_k}$.\\[0.5em]
\textbf{Vector Notations.} We shall use the shorthand notation $[z]_n$ to denote a row vector containing $n$ copies of $z\,\foralla z \in \MBC, n \in \MBZP$. Moreover, $c^{0:N-1}, {\bm{x}_N}\;(\text{or }x_{N,0:N-1}^t), {\bm{x}_N^+}\; (\text{or }x_{N,0:N}^t)$ stand for the $N$th-dimensional row vector $[c^0, c^1, \ldots, c^{N-1}]\,\foralla c \in \canczer{\MBR}$ and the column vectors $[x_{N,0}, x_{N,1}, \ldots, x_{N,N-1}]^t$ and $[x_{N,0}, x_{N,1}, \ldots, x_{N,N}]^t$, respectively. $g_{0:N-1}$, $g^t_{0:N-1}$, $g^{(0:n)}$, and $f(g_{0:N-1})$ denote the column vector $[g_0, g_1, \ldots, g_{N-1}]^t$, the transpose vector $[g_0, g_1, \ldots, g_{N-1}]$, the column vector of derivatives $[g, g', \ldots, g^{(n)}]^t\,\forall n \in \MBZzerP$, and the column vector of composite function values $[f(g_0), f(g_1), \ldots$, $f(g_{N-1})]^t \foralla$ $f \in \MBC$ in respective order. Moreover, $\ind g_{0:N-1}, \indmax g_{0:N-1}$, and $\indmin g_{0:N-1}$ denote the indices vector of nonzero values and maximum- and minimum-values of $g_{0:N-1}$, respectively.\\[0.5em] 
\textbf{Interval Notations.} The shorthand notation $[c, y_i]_{i = 0:n-1}$ stands for the collection of intervals $[c, y_0], \ldots, [c, y_{n-1}]\,\foralla c$, $\{y_{0:n-1}\} \subset \MBR: y_j > c\,\forall j \in \MBJ_n$. The specific interval $[0, c]$ is denoted by $\F{\Omega}_c\,\forall c > 0$. For example, $[0, x_{N,n}]$ is denoted by ${\F{\Omega}_{x_{N,n}}}$; moreover, ${\F{\Omega}_{x_{N,0:N-1}}}$ stands for the list of intervals ${\F{\Omega}_{x_{N,0}}}, {\F{\Omega}_{x_{N,1}}}, \ldots, {\F{\Omega}_{x_{N,N-1}}}$. $|\F{\Omega}|$ gives the length of an interval $\F{\Omega}$.\\[0.5em] 
\textbf{Integral Notations.} By closely following the convention for writing definite integrals introduced in \cite{elgindy2019high}, we denote $\int_0^{{x_{N,l}}} {h(x)\,dx}$ by $\C{I}_{{x_{N,l}}}^{(x)}h \foralla$ integrable $h \in \MBC$; moreover, by $\C{I}_{a,b}^{(x)}f$ we mean $\int_a^{b} {f(x)\,dx} \foralla a, b \in \MBR$ and an integrable function $f$. If the integrand functions $h$ and $f$ are to be evaluated at any other expression of $x$, say $u(x)$, we express $\int_0^{{x_{N,l}}} {h(u(x))\,dx}$ and $\int_a^{b} {f(u(x))\,dx}$ with a stroke through the square brackets as $\C{I}_{{x_{N,l}}}^{(x)}h\cancbra{u(x)}$ and $\C{I}_{a,b}^{(x)}f\cancbra{u(x)}$, respectively. We adopt the notation $\C{I}_{{\bm{x}_{N}}}^{(x)}h$ to denote the $N$th-dimensional column vector $\left[ {\C{I}_{{x_{N,0}}}^{(x)}h,\C{I}_{{x_{N,1}}}^{(x)}h, \ldots ,\C{I}_{{x_{N,N - 1}}}^{(x)}h} \right]^t$. The notation $\C{I}_{\F{\Omega}}^{(x)} h$ simply means the definite integral $\C{I}_{a,b}^{(x)} h\,\foralla \F{\Omega} = [a, b]$. Moreover, $\C{I}_{\F{\Omega}_{1:n}}^{(x)} h$ stands for $[\C{I}_{\F{\Omega}_1}^{(x)} h, \ldots, \C{I}_{\F{\Omega}_n}^{(x)} h]^t\,\foralla$ collection of intervals $\F{\Omega}_{1:n}$.\\[0.5em] 
\textbf{Matrix Notations.} $\F{O}_n, \F{1}_n$, and $\F{I}_n$ stand for the zero, all ones, and the identity matrices of size $n$. By $[\F{A}\,;\F{B}]$ we mean the usual vertical matrix concatenation of $\F{A}$ and $\F{B}\,\foralla$ two matrices $\F{A}$ and $\F{B}$ having the same number of columns. For a two-dimensional matrix $\F{C}$, the notation $\F{C}_{\hcancel{0}}$ stands for the matrix obtained by deleting the zeroth-row of $\F{C}$. Moreover, $\F{C}_n$ denotes a row vector whose elements are the $n$th-row elements of $\F{C}$, except when $\F{C}_n = \F{O}_n, \F{1}_n$, or $\F{I}_n$, where it denotes the size of the matrix. $\F{C}_{n,m}$ indicates that $\F{C}$ is a rectangular matrix of size $n \times m$. For convenience, a vector is represented in print by a bold italicized symbol while a two-dimensional matrix is represented by a bold symbol, except for a row vector whose elements form a certain row of a matrix where we represent it in bold symbol as stated earlier. For example, $\bmone_n$ and $\bmzer_n$ denote the $n$-dimensional all ones- and zeros- column vectors, while $\F{1}_n$ and $\F{O}_n$ denote the all ones- and zeros- matrices of size $n$, respectively.\\[0.5em] 
\textbf{Algorithmic Notations.} For algorithmic purpose, we adopt the notation ``$==$'' such that $\F{A}==\F{B}$ gives a logical array with elements set to logical $1$ where arrays $\F{A}$ and $\F{B}$ are equal; otherwise, the element is logical $0\,\foralla$ arrays $\F{A}$ and $\F{B}$ of the same size.

\section{Problem Statement}
\label{sec:PS1}
Consider the classical biochemical reaction kinetics model
\begin{subequations}
\begin{align}
\dot s &=  - \mu x + ({s_{{\text{in}}}} - s)u,\label{eq:BRKM1}\\
\dot x &= (\mu  - u)x,\label{eq:BRKM2}
\end{align}
governed by the Contois growth model 
\begin{equation}\label{eq:BRKM3}
\mu  = \frac{{{\mu _{\max }}s}}{{{k_s}x + s}},
\end{equation}
\end{subequations}
where $s(t), x(t)$, and $u(t)$ are the substrate concentration, the microorganism concentration, and the dilution rate (aka the feeding rate) at any time $t \in \F{\Omega}_T\,\foralls T > 0$, respectively, $\mu(s,x)$ is the specific growth rate of the microorganisms, $\mu_{\max}$ is the maximum specific growth rate, and $k_s$ is the Contois saturation constant. Assume that (i) $\Sin > 0$ is the input substrate concentration, (ii) $x(0) > 0$, (iii) $k_s > 1$, (iv) $u \in [u_{\min}, u_{\max}]: u_{\min}, u_{\max} \ge 0$, where $u_{\min}$ and $u_{\max}$ are the minimum and maximum dilution rates allowed, respectively, and (v) the substrate quantity $\bs \in [0, \Sin]$ brought by a $T$-periodic dilution rate is equal to the quantity brought by some constant dilution rate $\bu \in [u_{\min}, u_{\max}]: \bu < \mu_{\max}$. Under these assumptions, the goal is to find the optimal $T$-periodic waveforms, $x^*, s^*$, and $u^*$, which satisfy the dynamical system equations \eqref{eq:BRKM1} and \eqref{eq:BRKM2} of the chemostat model governed by the Contois growth model \eqref{eq:BRKM3} and minimize the time-averaged substrate concentration. Such a problem can be described by the following finite horizon OCP: For a given time period $T > 0$, find the control variable $u$ on the time interval $[0, T]$ that minimizes the performance index
\begin{subequations}
\begin{equation}\label{eq:OC1}
J(u) = \frac{1}{T} \C{I}_{T}^{(t)} {s}
\end{equation}
subject to Eqs. \eqref{eq:BRKM1}-\eqref{eq:BRKM3}, 
\begin{gather}
0 \le s \le {s_{{\text{in}}}},\quad u \in \C{U},\quad \frac{1}{T} \C{I}_{T}^{(t)} {u}  = \bu,\label{eq:EC1}\\
s(0) = s(T) = \bs,\quad \text{and }\label{eq:EC2}\\
x(0) = x(T),
\end{gather}
\end{subequations}
and under Assumptions (i)-(v), where $s$ and $x$ are the state variables, and 
\[\C{U} = \{u: \MBRzerP \to [u_{\min}, u_{\max}]\text{ s.t. }u\text{ is measurable and }T\text{-periodic}\}.\]
To reduce the dimensionality of the OCP, \citet{bayen2018optimal} argued that all trajectories of Eqs. \eqref{eq:BRKM1} and \eqref{eq:BRKM2} converge asymptotically to the invariant set $s + x =\Sin$, so we can reduce the coupled system of dynamic equations \eqref{eq:BRKM1} and \eqref{eq:BRKM2} into the single differential equation
\begin{equation}\label{eq:SDE1}
\dot s = \psi,
\end{equation}
where $\psi = (u - \nu )({s_{{\text{in}}}} - s)$ is the state derivative variable and $\nu  = \displaystyle{\frac{{{\mu _{\max }}s}}{{{k_s}({s_{{\text{in}}}} - s) + s}}}$. Moreover, the non-trivial equilibrium solution $\bs$ of the differential equation ${\left. {\dot s} \right|_{u = \bu}} = 0$ exists and is given by 
\begin{equation}\label{eq:bsbu1}
\bs = \displaystyle{\frac{{\bu{k_s}{s_{{\text{in}}}}}}{{\bu({k_s} - 1) + {\mu _{\max }}}}}.
\end{equation}
The system under Assumptions (i)--(v) exhibits an overyielding state in the sense that there exists $u: J(u) < J(\bu) = \bs$. The goal now is to find the optimal $T$-periodic waveforms, $s^*$ and $u^*$, which minimize the averaged substrate concentration \eqref{eq:OC1} and satisfy the differential Eq. \eqref{eq:SDE1} together with Conditions \eqref{eq:EC1} and \eqref{eq:EC2}. We refer to this problem by Problem $\C{P}$. If we integrate both sides of Eq. \eqref{eq:SDE1} over the time interval $\F{\Omega}_t \foralls t \in {\F{\Omega}_T}\backslash\{0\}$, we transform the OCP into its integral form where the same performance index $J$ is minimized subject to the integral equation
\begin{equation}\label{eq:IDS1}
s(t) = \bs + \C{I}_t^{(x)} {\psi},
\end{equation}
and Conditions \eqref{eq:EC1} and \eqref{eq:EC2}. We refer to this integral form of Problem $\C{P}$ by Problem $\C{IP}$. Although the solutions of Problems $\C{P}$ and $\C{IP}$ are mathematically equivalent, they are not necessarily numerically equivalent in floating-point arithmetic. In particular, while the numerical discretization of Problem $\C{P}$ entails the use of numerical differentiation operators known to be ill-conditioned as they could potentially lead to serious round-off errors, the numerical discretization of Problem $\C{IP}$ admits the use of numerical integration operators widely popular for being `well-conditioned operators,' and `their well-conditioning is essentially unaffected for increasing number of points'; see \cite{Elgindy20171,Elgindy2020distributed} and the references therein. 

%\section{Numerical Discretization Operators}
\section{FPSI Matrices and Barycentric SG Quadratures}
\label{sec:NDOnew1}
In this section, we present some novel numerical tools required to discretize Problem $\C{IP}$. In particular, the first set of numerical tools is used to construct novel FPSI matrices that can produce more accuracy and speed in approximating the integrations of periodic functions while reducing the computational complexity required compared to the recent formulas of \citet{elgindy2019high}. We also analyze the practical difficulty with the Fourier collocation of Problem $\C{IP}$ at the mesh points set $\MBS_N$, which is the well recovery of the $N$ definite integrals of the nonlinear state derivative variable $\psi$ over the intervals $\F{\Omega}_{x_{N,n}} \forall n \in \MBJ_N$. A similar difficulty is encountered when evaluating the definite integral of the approximate substrate concentration $\tilde s$ over ${\F{\Omega}_T}$. Note that $\psi$ is a $T$-periodic function, because it is a composition of $T$-periodic functions $s$ and $u$. Thus, it is possible to estimate the required integrals using the FPSQs. However, $\psi$ is generally a discontinuous function due to the presence of the bang–bang controller $u$, so the FPSQ error Euclidean-norm of $\psi$ decays like $O\left(N^{-1/2}\right)$ as we shall discuss later in \ref{subsec:CRFNSPF1}, assuming that both $u$ and $s$ are computed using exact arithmetic. In practice, the expected poor convergence rate of FPSQ in computing the required integrals of $\psi$ after recovering the approximate discontinuous controller through Algorithm \ref{alg:2}, which is presented later in \ref{app:Alg1}, motivates us to seek an alternative strategy to evaluate the sought integrals efficiently and with higher accuracy. In particular, to refine the required integral approximations and improve the convergence rate, an alternative approach in lieu of using FPSQs, after utilizing them to determine the approximate discontinuous controller, is to numerically piecewise integrate $\psi$ using Gauss-type quadratures. This is where the second set of our numerical tools kicks in. In particular, because the optimal controller switches between two predefined states at two unspecified time instances in $\F{\Omega}_T$, the idea is to estimate these two time instances and then partition the time interval $\F{\Omega}_T$ into three subintervals determined by the estimated time instances over which the restricted $\psi$ on each subinterval is smooth. The application of the Gauss quadrature over each subinterval is optimal in this case because it requires the smallest number of points to calculate the exact integration of the highest possible order polynomial. In fact, an $n$-point Gaussian quadrature rule is exact for polynomials of degree at most $2n - 1$. 

\subsection{FPSI Matrices in Reduced Form}
\label{sec:FPIMIRF1}
Using the Fourier quadrature rule 
\[{Q_F}(f) = \frac{T}{N}\sum\limits_{j = 0}^{N - 1} {{f_j}},\quad \forall f \in \MBT_T,\]
we can define the following discrete inner product
\[{(u,v)_N} = \frac{T}{N}\sum\limits_{j = 0}^{N - 1} {{u_j}v_j^*},\]
$\foralla u, v \in \MBC$, where $v_j^*$ is the complex conjugate of $v_j$. Now, let ${I_N}f$ be the $N/2$-degree, $T$-periodic Fourier interpolant that matches $f$ at the set of nodes $\MBS_N$ so that
\begin{equation}\label{eq:FI1}
{I_N}f(x) = \sum\limits_{k =  - N/2}^{N/2} {\frac{{{{\tilde f}_k}}}{{{c_k}}}{e^{i{\omega _k}x}}},
\end{equation}
where ${\omega _{\alpha}} = \displaystyle{\frac{{2\pi \alpha}}{T}}\,\forall \alpha \in \MBR$,
\begin{empheq}[left={c_k = }\empheqbiglbrace]{align*}
  1, &\quad k \in \MB{K}''_N,\\
  2, &\quad k =  \pm \frac{N}{2},
\end{empheq}
and ${\tilde f_k}$ is the discrete Fourier interpolation coefficient given by
\[{\tilde f_k} = \frac{1}{T}{\left( {f,{e^{i \omega_k x}}} \right)_N} = \frac{1}{N}\sum\limits_{j = 0}^{N - 1} {{f_j}{e^{ - i \omega_k {x_{N,j}}}}},\quad \forall k \in \MB{K}_N.\]
Since
\begin{equation}\label{eq:FCI1}
{\tilde f_{k \pm N}} = \frac{1}{N}\sum\limits_{j = 0}^{N - 1} {{f_j}{e^{ - i{\omega _{k \pm N}}{x_{N,j}}}}}  = \frac{1}{N}\sum\limits_{j = 0}^{N - 1} {{f_j}{e^{ - i{\omega _k}{x_{N,j}}}}{e^{ \mp i{\omega _N}{x_{N,j}}}}}  = \frac{1}{N}\sum\limits_{j = 0}^{N - 1} {{f_j}{e^{ - i{\omega _k}{x_{N,j}}}}{e^{ \mp 2\pi ij}}}  = \frac{1}{N}\sum\limits_{j = 0}^{N - 1} {{f_j}{e^{ - i{\omega _k}{x_{N,j}}}}}  = {\tilde f_k},
\end{equation}
then $\tilde f_{N/2} = \tilde f_{-N/2}$. Therefore, we can rewrite Eq. \eqref{eq:FI1} in the following reduced form
\begin{equation}\label{eq:FI1nn1}
{I_N}f(x) = \sumd\sum\limits_{\left| k \right| \le N/2} {{\tilde f_k} {e^{i{\omega _k}x}}},
\end{equation}
where the primed sigma denotes a summation in which the last term is omitted. We can now define the DFT pair by
\begin{subequations}
\begin{empheq}[left=\empheqbiglbrace]{alignat=2}
  {{\tilde f}_k} &= \frac{1}{N}\sum\limits_{j = 0}^{N - 1} {{f_j}{e^{ - i{\omega _k}{x_{N,j}}}}}  = \frac{1}{N}\sum\limits_{j = 0}^{N - 1} {{f_j}{e^{ - i{{\hat \omega }_{jk}}}}}, &&\quad k \in \MB{K}'_N,\label{DFP1}\\
  {f_j} &= \sumd\sum\limits_{\left| k \right| \le N/2} {{{\tilde f}_k}{e^{i{\omega _k}{x_{N,j}}}}}  = \sumd\sum\limits_{\left| k \right| \le N/2} {{{\tilde f}_k}{e^{i{{\hat \omega }_{jk}}}}}, &&\quad \forall j \in \MBJ_N,\label{DFP2}
\end{empheq}
\end{subequations}
where ${{\hat \omega }_k} = 2\pi k/N\;\forall k$. Substituting Eq. \eqref{DFP1} into Eq. \eqref{eq:FI1nn1}, and then swapping the order of the summations, express the interpolant in the equivalent Lagrange form
\begin{equation}\label{eq:eqLF1}
{I_N}f(x) = \sum\limits_{j = 0}^{N - 1} {{f_j}{\C{F}_j}(x)},
\end{equation}
where ${\C{F}_j}(x)$ is the trigonometric Lagrange interpolating polynomial given by
\[{\C{F}_j}(x) = \frac{1}{N}\sumd\sum\limits_{\left| k \right| \le N/2} {{e^{i{\omega _k}(x - {x_{N,j}})}} = {\left[ {\frac{1}{N}\sin \left( {\frac{{\pi N}}{T}\left( {x - {x_{N,j}}} \right)} \right)\cot \left( {\frac{\pi }{T}\left( {x - {x_{N,j}}} \right)} \right)} \right]_{x \ne {x_{N,j}}}},\quad j \in \MBJ_N};\]
see \cite{elgindy2019high}. Since $\C{F}_j(x_{N,l}) = \delta_{j,l}\,\forall j,l \in \MBJ_N$, where $\delta_{j,l}$ is the kronecker delta function of variables $j$ and $l$, one can easily write the vector of interpolant values at the grid points set $\MBS_N$ using the writing conventions introduced in Section \ref{sec:PN} as $\left( {{I_N}f} \right)_{0:N - 1} = f_{0:N - 1}$. We can also integrate ${I_N}f$ over the interval $\F{\Omega}_{x_{N,l}}$ through the formula
\begin{equation}\label{eq:AppRedFIM1}
\C{I}_{{x_{N,l}}}^{(x)}({I_N}f) = \sum\limits_{j = 0}^{N - 1} {{f_j}\C{I}_{{x_{N,l}}}^{(x)}{\C{F}_j}}  = \sum\limits_{j = 0}^{N - 1} {{\theta _{l,j}}{f_j}},\quad \forall l \in \MBJ_N,
\end{equation}
where 
\begin{equation}\label{eq:RedFIM1}
{\theta _{l,j}} = \C{I}_{{x_{N,l}}}^{(x)}{\C{F}_j} = \frac{1}{N}\left[ {{x_{N,l}} + \frac{{Ti}}{{2\pi }} \sumd\sum\limits_{\scriptstyle \left| k \right| \le N/2\atop
\scriptstyle k \ne 0} {\frac{1}{k}{e^{ - i{\omega _k}{x_{N,j}}}}\left( {1 - {e^{i{\omega _k}{x_{N,l}}}}} \right)}} \right],\quad l,j \in \MBJ_N,
\end{equation}
are the entries of the first-order square Fourier integration matrix (FIM), $\Fthe$, of size $N$. \citet{elgindy2019high} pointed out further that when the calculation of $\C{I}_{{y_{M,l}}}^{(x)}({I_N}f)$ is needed, $\foralls M$-random set of points $\{ {{y_{M,0:M-1}}}\} \subset (0,T]:{y_{M,l}} \notin {\MBS_N}\forall M \in {\MBZ^ + },l \in \MBJ_M$, one can derive the elements formulas of the associated rectangular FIM, $\F{\hat{\Theta}} = \left(\hat \theta_{l,j}\right), l \in \MBJ_M, j \in \MBJ_N$, by performing the replacement $x_{N,l} \leftarrow y_{M,l}$ in Formulas \eqref{eq:RedFIM1}:
\begin{equation}\label{eq:RedMFIM1}
{\hat{\theta}_{l,j}} = \C{I}_{{y_{M,l}}}^{(x)}{\C{F}_j} = \frac{1}{N}\left[ {{y_{M,l}} + \frac{{Ti}}{{2\pi }}\sumd\sum\limits_{\scriptstyle\left| k \right| \le N/2\atop
\scriptstyle k \ne 0} {\frac{1}{k}{e^{ - i{\omega _k}{x_{N,j}}}}\left( {1 - {e^{i{\omega _k}{y_{M,l}}}}} \right)}} \right],\quad l \in \MBJ_M, j \in \MBJ_N.
\end{equation}
We can rewrite Formulas \eqref{eq:AppRedFIM1} and its variants
\begin{equation}\label{eq:AppRedFIM2}
\C{I}_{{y_{M,l}}}^{(x)}({I_N}f) = \sum\limits_{j = 0}^{N - 1} {{f_j}\C{I}_{{y_{M,l}}}^{(x)}{\C{F}_j}}  = \sum\limits_{j = 0}^{N - 1} {{\hat \theta _{l,j}}{f_j}},\quad l \in \MBJ_M,
\end{equation}
which are consistent with the set $\{ {{y_{M,0:M-1}}}\}$ in matrix notation as 
\begin{subequations}
\begin{align}
{{\C{I}_{{{\bm{x}_N}}}^{(x)}(I_Nf)}} &= \Fthe f_{0:N - 1},\label{eq:Toto13Feb1}\\
{{\C{I}_{{{\bm{y}_M}}}^{(x)}(I_Nf)}} &= \F{\hat{\Theta}} f_{0:N - 1},
\end{align}
\end{subequations}
respectively. In the special case when $y_{M,l} = T$, Formula \eqref{eq:RedMFIM1} reduces to ${\hat{\theta}_{l,j}} = T/N\, \forall j \in \MBJ_N$. For convenience, we denote ${\hat{\theta}_{l,j}}$ by $\theta_{N,j}$ in this particular case and define $\Fthe_N = \frac{T}{N} \bmone_N^t$ so that 
\begin{equation}\label{eq:PCase13Feb221}
{{\C{I}_{{{T}}}^{(x)}(I_Nf)}} = \Fthe_N f_{0:N - 1} = \frac{T}{N} \left(\bmone_N^t f_{0:N - 1}\right).
\end{equation}
The zeroth-rows of $\Fthe$ and $\F{\hat{\Theta}}$ are zeros rows, and the matrix $\Fthe_{\hcancel{0}}$ is ``a row-wise element-twins matrix'' in the sense that each element in each row has exactly one twin element in the same row such that
\begin{align*}
{\theta _{l,j}} &= {\theta _{l,l - j}},\quad \forall l = 1, \ldots ,N - 1,\quad j = 0, \ldots ,\left\lfloor {\frac{{l - 1}}{2}} \right\rfloor ,\\
{\theta _{l,N - j}} &= {\theta _{l,l + j}},\quad \forall l = 1, \ldots ,N - 2,\quad j = 1, \ldots ,\left\lfloor {\frac{{N - l - 1}}{2}} \right\rfloor;
\end{align*}
see (\cite[pp. 379-380]{elgindy2019high}). This distinguished characteristic of $\Fthe_{\hcancel{0}}$ can be exploited to efficiently accelerate the construction of $\Fthe$ through Algorithm \ref{alg:1} in \ref{app:Alg1}; the operations $\odot$ and $\oslash$ in the algorithm refer to the Hadamard product and division, respectively. It is noteworthy to mention that the rectangular matrix $\Fthe_{\hcancel{0}}$ has a $T$-invariant, 2-norm condition number $\C{K}_2 = \left\| {\Fthe_{\hcancel{0}}} \right\|\left\| {\Fthe_{\hcancel{0}}^+} \right\| = O\left(N^{1.5}\right)$, for relatively large values of $N$, as clearly seen in Figure \ref{fig:Fig12_CondK2}, where ${\Fthe_{\hcancel{0}}^+}$ is the Moore-Penrose pseudoinverse of $\Fthe_{\hcancel{0}}$.

\begin{figure}[ht]
     \centering
     \begin{subfigure}[b]{0.5\textwidth}
         \centering
         \includegraphics[width=1\textwidth]{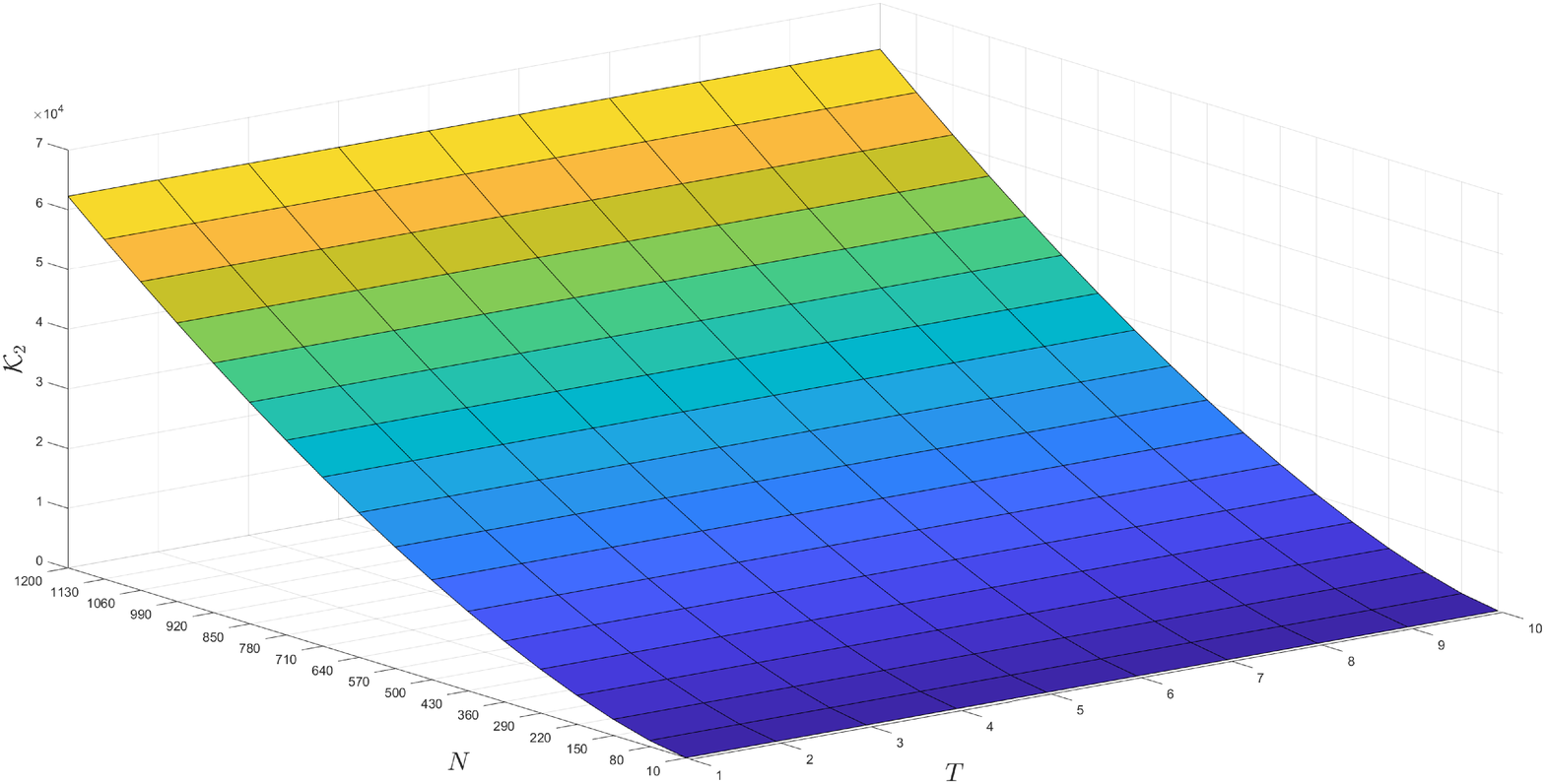}
         \caption{}
         \label{fig:Fig1_CondK2}
     \end{subfigure}
     \begin{subfigure}[b]{0.4\textwidth}
         \centering
         \includegraphics[width=\textwidth]{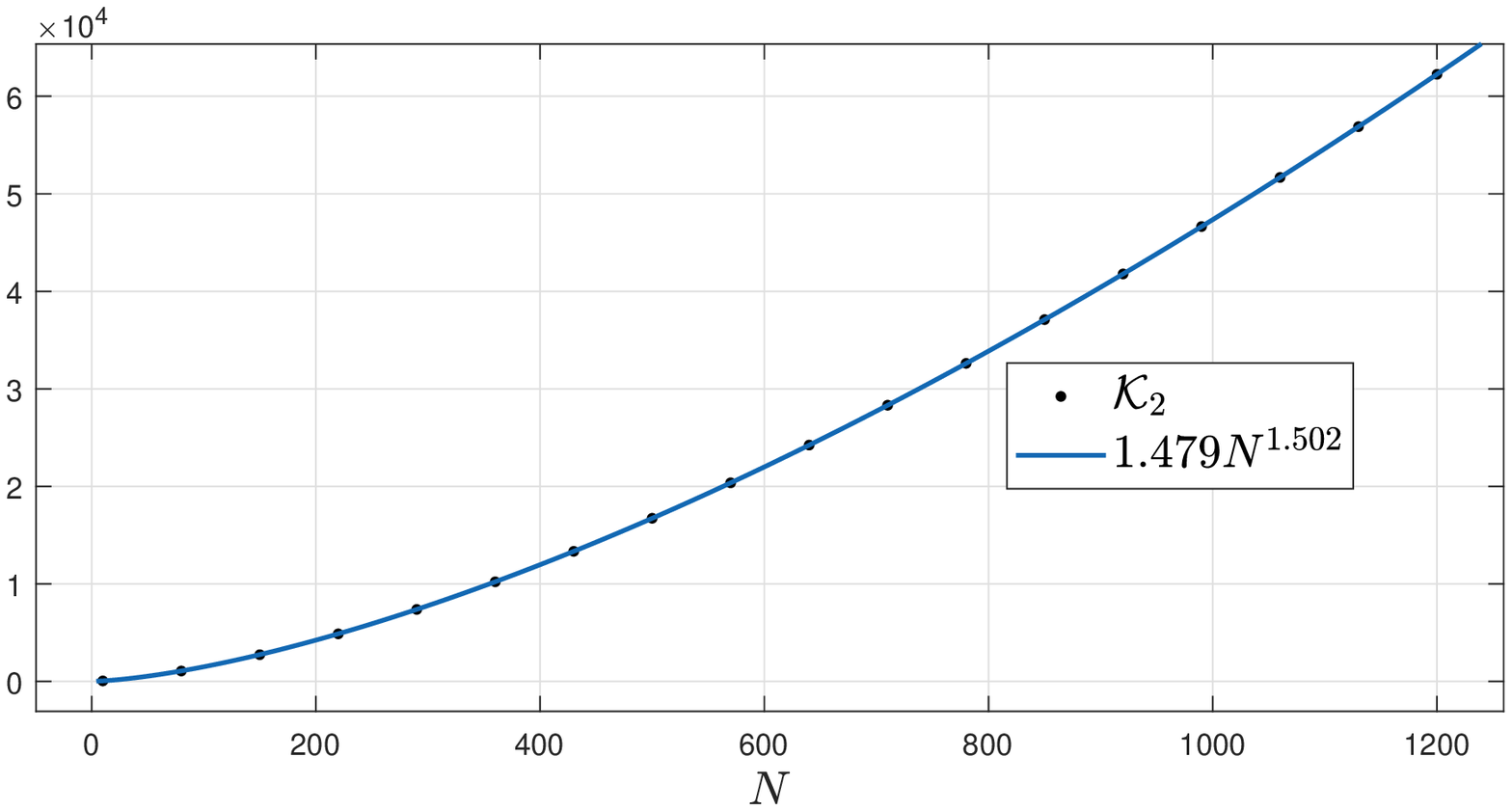}
		 \caption{}
         \label{fig:Fig2_CondK2}
     \end{subfigure}
     \caption{Figure (a): Surface plot of $\C{K}_2$, formed by joining adjacent point values of $\C{K}_2$ with straight lines, above a rectangular grid in the $TN$-plane generated using $T = 1:10$ and $N = 10(70)1200$. Figure (b): Plot of $\C{K}_2$ and its power function least-squares fit of the form $a N^b$ against $N = 10(70)1200$. The fitted coefficients have the 95\% confidence bounds $1.476 \le a \le 1.483$ and $1.501 \le b \le 1.502$.}
     \label{fig:Fig12_CondK2}
\end{figure}

%\section{Accurate Computation of Definite Integrals of Reconstructed Discontinuous Functions from the FPS Data}
%\label{sec:DIORDFFTSD}

\subsection{Barycentric SG Quadratures}
\label{subsec:BSGQMay20221}
Suppose that we collocate Problem $\C{IP}$ using Fourier collocation, construct the Fourier interpolant, estimate the jump discontinuity points $\xi_1$ and $\xi_2$ of the OC $u$ by $\tilde \xi_1$ and $\tilde \xi_2$, and finally establish a piecewise function whose pieces are defined over the three subintervals $\F{\Omega}_{\tilde \xi_1}, [\tilde \xi_1, \tilde \xi_2]$, and $[\tilde \xi_2, T]$. Motivated by the need to successively integrate the composite function $\psi$, we use the barycentric SG-Gauss (SGG) quadratures prompted by the barycentric Gegenbauer-Gauss (GG) quadratures derived by \citet{Elgindy20171}. These numerical operators were constructed using the stable barycentric representation of Lagrange interpolating polynomials and explicit barycentric weights for the GG points, and they are well known for their stability and superior accuracy. To derive the SG quadratures in barycentric form suited to carry out integrals over any partition of $\F{\Omega}_T$ and study their errors, we closely follow the notations and derivations presented in \cite{Elgindy20172}. Let $\MBK_K = \{1, \ldots, K\}$ and consider the partitioning of the time interval $\F{\Omega}_T$ into $K$ mesh intervals $\bm{\Gamma}_{1:K}$ using $K+1$ mesh points $\tau_{0:K}$ distributed along the interval $\F{\Omega}_T$ such that $\F{\Omega}_T = \bigcup\limits_{k = 1}^K {{\mkern 1mu} {\bm{\Gamma}_k}} ,\quad {\bm{\Gamma}_k} = [{\tau _{k - 1}},{\tau _k}],\quad 0 = {\tau _0} < {\tau _1} <  \ldots  < {\tau _K} = T$. Let $\tau_k^+ = (\tau_{k}+\tau_{k-1})/2$ and $\tau_k^- = |\F{\Gamma}_k|/2$ in respective order, and denote the restricted variable $t$ whose values are confined to ${\bm{\Gamma}_k}$ by $t^{(k)}$ such that $t^{(k)} = t: \tau_{k-1} \le t \le \tau_k$. Define $\hat G_{k,j}^{(\alpha )}\left( {{t^{(k)}}} \right) = G_{j}^{(\alpha )}\left( {\left( {t^{(k)}} - \tau_k^+ \right)/\tau_k^-} \right)$ to be the $j$th-degree SG polynomial defined on the partition $\bm{\Gamma}_k \foralle j \in \MBZzerP, k \in \MBK_K$-- henceforth referred to by the $j$th-degree, $k$th element SG polynomial (or simply the $(j,k)$-SG polynomial), where $G_{j}^{(\alpha )}(t)$ is the classical $j$th-degree Gegenbauer polynomial associated with the real parameter $\alpha > -1/2$ and standardized by \citet{Doha199075}; see also \cite[Formula (A.1)]{elgindy2013fast}. The $k$th element SG polynomials $\hat G_{k,0:N_k}^{(\alpha )}\left( {{t^{(k)}}} \right)$ form a complete $L_{w_k^{(\alpha)}}^2\left(\bm{\Gamma}_k\right)$-orthogonal system with respect to the weight function $w_k^{(\alpha )}\left({t^{(k)}}\right) = {\left( {{\tau _k} - {t^{(k)}}} \right)^{\alpha  - 1/2}}{\left( {{t^{(k)}} - {\tau _{k - 1}}} \right)^{\alpha  - 1/2}} \foralla N_k \in \MBZzerP$. An important and convenient property of these polynomials is that they are orthogonal with respect to the weighted inner product
\begin{equation}
	{\left( {\hat G_{k,m}^{(\alpha )},\hat G_{k,n}^{(\alpha )}} \right)_{w_k^{(\alpha )}}} = 
\C{I}_{{\tau _{k - 1}},{{\tau _k}}}^{\left( {{t^{(k)}}} \right)} {\hat G_{k,m}^{(\alpha )} \hat G_{k,n}^{(\alpha )} w_k^{(\alpha )}}  = \left\| {\hat G_{k,n}^{(\alpha )}} \right\|_{w_k^{(\alpha )}}^2{\delta _{m,n}} = \lambda _{k,n}^{(\alpha )}{\delta _{m,n}},\quad \forall m, n \in \MBZzerP,
\end{equation}
where $\delta _{m,n}$ is the Kronecker delta function, $\lambda _{k,n}^{(\alpha )} = {\left(\tau_k^- \right)^{2\alpha }}\lambda _n^{(\alpha )}$ is the normalization factor, and $\lambda _n^{(\alpha )}$ is as defined by (\cite[Eq. (2.6)]{Elgindy20161}); see also \cite{Elgindy20172}. Now, $\foralle k \in \MBK_K$, let $\MBG_{N_k}^{(\alpha),k} = \left\{\hat t_{N_k,l}^{(k),\alpha}\,\forall l \in \MBJ_{N_k}^+\right\}$ be the set of the zeroes of the $\left(N_k+1,k\right)$-SG polynomial, $\hat G_{k,{N_k} + 1}^{(\alpha )}\left( {{t^{(k)}}} \right) \foralls N_k \in \MBZzerP$. If we denote by $\mathbb{P}{_{n}}$, the space of all polynomials of degree at most $n \in \MBZP$, then 
\begin{equation}\label{eq:hihi1}
\C{I}_{\tau_{k-1}, \tau_k}^{\left(t^{(k)}\right)} {\phi \,w_k^{(\alpha )}} = (\tau_k^-)^{2 \alpha} \C{I}_{ - 1, 1}^{(t)} {\phi \cancbra{\tau_k^- t + \tau_k^+}\,{w^{(\alpha )}}} = {\left( \tau_k^- \right)^{2\alpha }}\sum\limits_{j = 0}^{N_k} {\varpi _j^{(\alpha )}\,\phi \left( {\tau_k^- t_{N_k,j}^{(\alpha )} + \tau_k^+} \right)} = \sum\limits_{j = 0}^{N_k} {\varpi _{k,j}^{(\alpha )}\,\phi \left(\hat t_{N_k,j}^{(k),\alpha}\right)},\quad \forall \phi  \in \mathbb{P}{_{2n + 1}},
\end{equation}
using the standard GG quadrature, where $t_{N_k,0:N_k}^{(\alpha )}$ are the zeroes of the classical $(N_k+1)$th-degree Gegenbauer polynomial $G_{{N_k} + 1}^{(\alpha )}(t), \varpi _{0:n_k}^{(\alpha )}$ are the corresponding Christoffel numbers as given by (\cite[Eq. (2.6)]{Elgindy201382}), and $\varpi _{k,0:N_k}^{(\alpha )}$ are the Christoffel numbers corresponding to the SGG set $\MBG_{N_k}^{(\alpha),k}$ and defined by
\begin{equation}
\varpi _{k,l}^{(\alpha )} = (\tau_k^-)^{2 \alpha} \varpi_l^{(\alpha)} = \frac{1}{{\sum\limits_{j = 0}^{N_k} {{{\left( {\lambda _{k,l}^{(\alpha )}} \right)}^{ - 1}}{\mkern 1mu} {{\left( {\hat G_{k,j}^{(\alpha )}\left( {\hat t_{{N_k},l}^{(k), \alpha}} \right)} \right)}^2}} }},\quad \forall l \in \MBJ_{N_k}^+.
\end{equation}
The SG Quadrature Rule \eqref{eq:hihi1} allows us to define the discrete inner product $(\cdot, \cdot)_{k,N_k}$ associated with the SGG interpolation points as follows:
\begin{equation}
	{(u,v)_{k,N_k}} = \sum\limits_{j = 0}^{N_k} {\varpi _{k,j}^{(\alpha )}\,u\left( {\hat t_{N_k,j}^{(k),\alpha}} \right)\,v\left( {\hat t_{N_k,j}^{(k),\alpha}} \right)},\quad \foralla u, v \in \Def{\F{\Gamma}_k}.
\end{equation} 
With this mathematical setting, we can write the SGG interpolant of a restricted, real function $f$ on $\F{\Gamma}_k$ obtained through interpolation at the set $\MBG_{N_k}^{(\alpha),k}$ as 
\begin{equation}\label{eq:Sel1}
{P_{{N_k}}}f\left( {{t^{(k)}}} \right) = \sum\limits_{j = 0}^{{N_k}} {a_j^{(k)}\,\hat G_{k,j}^{(\alpha )}\left( {{t^{(k)}}} \right)},
\end{equation}
where $a_{0:N_k}^{(k)}$ are the associated discrete interpolation coefficients, $\forall k \in \MBK_K$, defined by
\begin{equation}\label{sec:ort:eq:sgt}
a_j^{(k)} = \frac{{{{\left( {{P_{{N_k}}}f,\hat G_{k,j}^{(\alpha )}} \right)}_{k,{N_k}}}}}{{\left\| {\hat G_{k,j}^{(\alpha )}} \right\|_{w_k^{(\alpha )}}^2}} = \frac{{{{\left( {f,\hat G_{k,j}^{(\alpha )}} \right)}_{k,{N_k}}}}}{{\left\| {\hat G_{k,j}^{(\alpha )}} \right\|_{w_k^{(\alpha )}}^2}} = \frac{1}{{\lambda _{k,j}^{(\alpha )}}}\sum\limits_{l = 0}^{N_k} {\varpi _{k,l}^{(\alpha )}\,f\left(\hat t_{N_k,l}^{(k),\alpha}\right)\,\hat G_{k,j}^{(\alpha )}\left(\hat t_{N_k,l}^{(k),\alpha}\right)},\quad \forall j \in \MBJ_{N_k}^+.
\end{equation}
Equation \eqref{sec:ort:eq:sgt} gives the discrete SG transform on $\F{\Gamma}_k$. Substituting Eq. \eqref{sec:ort:eq:sgt} into Eq. \eqref{eq:Sel1} yields the SGG interpolant of $f$ in the following Lagrange form 
\begin{equation}\label{sec:ort:eq:Lagint1}
	{P_{N_k}}f\left(t^{(k)}\right) = \sum\limits_{l = 0}^{N_k} {f\left(\hat t_{N_k,l}^{(k),\alpha}\right) \,\C{L}_{k,l}^{(\alpha )}\left(t^{(k)}\right)},
\end{equation}
where $\C{L}_{k,0:N_k}^{(\alpha )}\left(t^{(k)}\right)$ are the shifted Lagrange interpolating polynomials in basis-form defined on $\F{\Gamma}_k$ by
\begin{equation}\label{sec:ort:eq:Lag1}
	\C{L}_{k,l}^{(\alpha )}\left(t^{(k)}\right) = \varpi _{k,l}^{(\alpha )}\sum\limits_{j = 0}^n {{{\left( {\lambda _{k,j}^{(\alpha )}} \right)}^{ - 1}}\,\hat G_{k,j}^{(\alpha )}\left(\hat t_{N_k,l}^{(k),\alpha}\right)\,\hat G_{k,j}^{(\alpha )}\left(t^{(k)}\right)},\quad \forall l \in \MBJ_{N_k}^+.
\end{equation}
A faster and more numerically stable way to evaluate ${P_{N_k}}f\left(t^{(k)}\right)$ can be achieved by calculating Lagrange polynomials through the ``true'' barycentric formula
\begin{equation}
\C{L}_{k,l}^{(\alpha )}\left( {{t^{(k)}}} \right) = \frac{{\xi _{k,l}^{(\alpha )}}}{{{t^{(k)}} - \hat t_{{N_k},l}^{(k),\alpha }}}/\sum\limits_{j = 0}^{{N_k}} {\frac{{\xi _{k,j}^{(\alpha )}}}{{{t^{(k)}} - \hat t_{{N_k},j}^{(k),\alpha }}}},\quad \forall l \in \MBJ_{N_k}^+,
\label{eq:new1}
\end{equation}
which brings into play the barycentric weights ${\xi _{k,0:N_k}^{(\alpha )}}$ that depend on the interpolation points. An interpolation in Lagrange form with Lagrange polynomials defined by Formula  \eqref{eq:new1} is often referred to by ``a barycentric rational interpolation.'' The barycentric weights associated with the SGG points in $\bm{\Gamma}_1$ can be expressed explicitly in terms of the corresponding Christoffel numbers in algebraic form by
\begin{equation}
\xi _{1,l}^{(\alpha )} = \left(\frac{2}{\tau_1}\right)^{\alpha+1}{( - 1)^l}\sqrt {\left( {{\tau _1} - \hat t_{{N_1},l}^{(1),\alpha }} \right)\hat t_{{N_1},l}^{(1),\alpha }\varpi _{1,l}^{(\alpha )}},\quad \forall l \in \MBJ_{N_1}^+,
\end{equation}
or in trigonometric form through
\begin{equation}
\xi _{1,l}^{(\alpha )} = {\left( {\frac{2}{{{\tau _1}}}} \right)^\alpha }{( - 1)^l}\sin \left( {{{\cos }^{ - 1}}\left( {\frac{{2\hat t_{{N_1},l}^{(1),\alpha }}}{{{\tau _1}}} - 1} \right)} \right)\sqrt {\varpi _{1,l}^{(\alpha )}},\quad \forall l \in \MBJ_{N_1}^+,
\end{equation}
see (\cite[Eqs. (21) and (22)]{elgindy2018optimal}). Through the change of variables $x = \left({t^{(k)}} - \tau_k^+ \right)/\tau_k^-$, it is easy to show that the barycentric weights associated with the SSG in any partition $\bm{\Gamma}_k$ can be defined in algebraic form by
\begin{equation}\label{eq:newbaryform14May20221}
\xi _{k,l}^{(\alpha )} = {(\tau _k^ - )^{ - (\alpha  + 1)}}{( - 1)^l}\sqrt {\left( {{\tau _k} - \hat t_{{N_k},l}^{(k),\alpha }} \right)\left( {\hat t_{{N_k},l}^{(k),\alpha } - {\tau _{k - 1}}} \right)\varpi _{k,l}^{(\alpha )}},\quad \forall l \in \MBJ_{N_k}^+,
\end{equation}
or in trigonometric form through
\begin{equation}\label{eq:newbaryform14May20222}
\xi _{k,l}^{(\alpha )} = {(\tau _k^ - )^{ - \alpha }}{( - 1)^l}\sin \left( {{{\cos }^{ - 1}}\left( {\frac{{\hat t_{{N_k},l}^{(k),\alpha } - \tau _k^ + }}{{\tau _k^ - }}} \right)} \right)\sqrt {\varpi _{k,l}^{(\alpha )}},\quad \forall l \in \MBJ_{N_k}^+.
\end{equation}
Formula \eqref{eq:newbaryform14May20222} avoids the cancellation error in calculating Formula \eqref{eq:newbaryform14May20221} using floating point arithmetic due to the clustering of the SGG points near the endpoints of each partition as the mesh size increase. The successive integrations of the SGG interpolant \eqref{sec:ort:eq:Lagint1} on the intervals $\left[\tau_{k-1}, \hat t_{N_k,i}^{(k),\alpha}\right] \forall i \in \MBJ_{N_k}^+$, give rise to the first-order, $k$th element, square SG integration matrix (SGIM) in barycentric form, ${}_k\F{P} = {\left( {{}_k{p_{i,j}}} \right)_{{\mkern 1mu} 0 \leqslant i,j \leqslant N_k}}$, whose elements are defined by
\begin{equation}\label{eq:catmice1}
{{}_k{p_{i,j}}} = \varpi _{k,j}^{(\alpha )}\sum\limits_{l = 0}^n {{{\left( {\lambda _{k,l}^{(\alpha )}} \right)}^{ - 1}}\,\hat G_{k,l}^{(\alpha )}\left(\hat t_{N_k,j}^{(k),\alpha}\right)\,\C{I}_{{\tau _{k - 1}}, {\hat t_{N_k,i}^{(k),\alpha}}}^{\left(t^{(k)}\right)} \hat G_{k,l}^{(\alpha )}},\quad i,j \in \MBJ_{N_k}^+,\; k \in \MBK_K,
\end{equation}
in basis-form, or
\begin{equation}\label{eq:catmicerat1}
{{}_k{p_{i,j}}} = \C{I}_{{\tau _{k - 1}}, {\hat t_{N_k,i}^{(k),\alpha}}}^{\left(t^{(k)}\right)} \C{L}_{k,j}^{(\alpha )},\quad i,j \in \MBJ_{N_k}^+,\; k \in \MBK_K,
\end{equation}
in rational-form, where $\C{L}_{k,j}^{(\alpha )}$ is as defined by Eq. \eqref{eq:new1}. When the computations of $\C{I}_{{y_{M_k,i}}}^{\left(t^{(k)}\right)}({P_{N_k}}f)$ are needed $\foralls M_k$-random set of points $\left\{ {{y_{M_k,0:M_k}}} \right\} \subset (\tau_{k-1},\tau_k]:{y_{M_k,i}} \notin \MBG_{N_k}^{(\alpha),k}\,\forall M_k \in {\MBZ^ + },i \in \MBJ_{M_k}^+$, one can derive the elements formulas of the associated rectangular SGIM, ${}_k\hat{\F{P}} = \left({}_k{\hat p_{i,j}}\right), i \in \MBJ_{M_k}^+, j \in \MBJ_{N_k}^+$, by performing the replacement ${\hat t_{N_k,i}^{(k),\alpha}} \leftarrow y_{M_k,i}$ in Formulas \eqref{eq:catmice1} and \eqref{eq:catmicerat1}. If an element $y_{M_k,i} = \tau_k\,\foralla k \in \MBK_K$, we shall conveniently denote it by ${\hat t_{N_k,N_k+1}^{(k),\alpha}}$, and replace its associated matrix elements ${}_k{\hat p_{i,j}}$ in this particular case by ${{}_k{p_{N_k+1,j}}}$ such that ${}_k\F{P}_{N_k+1} = {{}_k{p_{N_k+1,0:N_k}}}$. For a comprehensive review on Gegenbauer polynomials and quadratures and their relevant theory, the reader may consult \cite{Abramowitz1964handbook,el1969chebyshev,Doha199075,Elgindy201382,
Elgindy20161,Elgindy20171,abd2014new,hafez2022shifted,taghian2021shifted}, and the references therein. For additional information on why we prefer to use Gegenbauer polynomials and their shifted variants in numerical discretization, the reader may consult \cite{ElgindyHareth2022a,elgindy2018high} and the references therein.

In the following section, we present the main numerical method used in this study. Interested readers to learn about the errors and convergence rates associated with FPS approximations and barycentric SG quadratures may refer to Section \ref{sec:ECAnKK1}. The reader may also consult Section \ref{sec:DTD} for a detailed prescription of how to reconstruct piecewise analytic functions directly from the FPS data with high accuracy up to the points of jump discontinuities. 

\section{The FG-PC Method}
\label{sec:FPCFPI1}
We initiate our proposed method of solution by collocating Problem $\C{IP}$ in the Fourier physical space at the set of mesh points $\MBS_N$ with the aid of Formulas \eqref{eq:Toto13Feb1} and \eqref{eq:PCase13Feb221} to obtain the discrete OCP

\begin{mini}
  {}{J_N = \frac{1}{N} \left(\bmone_N^t [\bs; s_{1:N-1}]\right)}{}{}
  {\label{prob:DOC1}}{}
  \addConstraint{s_{1:N-1}}{= \bs \bmone_{N-1} + \Fthe_{\hcancel{0}}\,\psi_{0:N-1}}{}
  \addConstraint{\bmone_N^t u_{0:N-1}}{= N \bu}{}
  \addConstraint{{\bmzer_{N-1}}}{\le s_{1:N-1} \le \Sin \bmone_{N-1}}{}
  \addConstraint{u_{\min} \bmone_N}{\le u_{0:N-1} \le u_{\max} \bmone_N.}{}
\end{mini}
Let $\bmX = [s_{1:N-1}; u_{0:N-1}], \bar{\bmX} = [\bs; \bmX_{1:N-1}], \bm{A} = [\bmzer_{N-1}^t, \bmone_N^t], \bmX_{lb} = [\bmzer_{N-1}; u_{\min} \bmone_N], \bmX_{ub} = [\Sin \bmone_{N-1}; u_{\max} \bmone_N]$, and $\hat\bmX = \Sin \bmone_N-\bar{\bmX}$. Then 
\[\psi_{0:N-1} = \left[\bmX_{N:2 N-1}-\mu_{\max} \bar{\bmX} \oslash \left(k_s \hat\bmX +\bar{\bmX}\right)\right] \odot \hat\bmX,\]
and the following scaled optimization problem 
\begin{mini}
  {}{\bar J_N = \bmone_N^t \bar{\bmX}}{}{}
  {\label{prob:DOC2}}{}
  \addConstraint{\bmX_{1:N-1}}{= \bs \bmone_{N-1} + \Fthe_{\hcancel{0}}\,\psi_{0:N-1}}{}
  \addConstraint{\bm{A} \bmX}{= N \bu}{}
  \addConstraint{\bmX_{lb}}{\le \bmX \le \bmX_{ub}.}{}
\end{mini}
is equivalent to the foregoing constrained NLP in the sense that an optimal solution $\bmX^* = [s_{1:N-1}^*; u_{0:N-1}^*]$ to the latter problem is also an optimal solution to the former problem; moreover, the optimal objective function value $J_N^* = \bar J_N^*/N$. We denote the predicted optimal state- and control-variables obtained at this stage by $s^p(t)$ and $u^p(t)$, respectively; their associated predicted optimal objective function value is denoted by $J_N^p$. Moreover, we denote $s_{1:N-1}^p$ and $u_{0:N-1}^p$ by $\bms_N^p$ and $\bmu_N^p$, and refer to them together with $\bmX^p = [\bms_N^p; \bmu_N^p]$ by the predicted state- control, and solution-vectors, in respective order. 

To improve the obtained approximations we construct the $T$-periodic Fourier interpolants ${I_N}s^p$ and ${I_N}u^p$ from $\bms_N^p$ and $\bmu_N^p$ through Formula \eqref{eq:eqLF1} as follows:
\begin{subequations}
\begin{align}
{I_N}s^p(t) &= \boldsymbol{\C{F}}(t)\,\bar{\bms}_N^p,\label{eq:May0920221}\\
{I_N}u^p(t) &= \boldsymbol{\C{F}}(t)\,\bmu_N^p,
\end{align}
\end{subequations}
where $\bar{\bms}_N^p = [\bs; \bms_N^p]$ and $\boldsymbol{\C{F}} = [\C{F}_{0}, \ldots, \C{F}_{N-1}]$. We then estimate the jump discontinuity points ${\tilde \xi }_{1:2}$ of the predicted controller and reconstruct the approximate piecewise analytic controller $\breve u_{N,M}$ from the PS data using Algorithm \ref{alg:2}. Since the controller is a bang–bang controller, then $u \in \{u_{min}, u_{\max}\}$, and the approximation $\breve u_{N,M}$ can be further corrected by the following formula:
\begin{subequations}
\begin{equation}\label{eq:CLPR1}
u_{N,M}^c(t) = \left\{ \begin{array}{l}
u_{\max},\quad 0 \le t < {{\tilde \xi }_1} \vee {{\tilde \xi }_2} \le t \le T,\\
u_{\min},\quad {{\tilde \xi }_1} \le t < {{\tilde \xi }_2},
\end{array} \right.
\end{equation}
if $\left|\breve u_{N,M}(\tilde \xi_1) - u_{\min}\right| < \left|u_{\max} - \breve u_{N,M}(\tilde \xi_1)\right|$, or by
\begin{equation}\label{eq:CLPR2}
u_{N,M}^c(t) = \left\{ \begin{array}{l}
u_{\min},\quad 0 \le t < {{\tilde \xi }_1} \vee {{\tilde \xi }_2} \le t \le T,\\
u_{\max},\quad {{\tilde \xi }_1} \le t < {{\tilde \xi }_2},
\end{array} \right.
\end{equation}
\end{subequations}
otherwise. Formulas \eqref{eq:CLPR1} and \eqref{eq:CLPR2} provide accurate approximations to the exact controller for relatively large values of $N$ and $M$ due to the close proximity of ${\tilde \xi }_{1:2}$ from the true jump discontinuity points ${\xi }_{1:2}$ as we shall demonstrate later in \ref{sec:DTD}. Now, let $\tau_1 = \tilde \xi_1, \tau_2 = \tilde \xi_2$, and $K = 3$. To obtain the corresponding corrected values of the state variable we can solve the nonlinear equality constraints 
\begin{equation}\label{eq:almthere1}
s_{1:N-1} = \bs \bmone_{N-1} + \C{I}_{\F{\Omega}_{x_{N,1: N-1}}}^{(t)} \psi,
\end{equation}
for $s_{1:N-1}$ starting from some initial approximations. However, to approximate $\C{I}_{\F{\Omega}_{x_{N,1: N-1}}}^{(t)} \psi$ using the highly accurate SG quadratures, the grid point values of $\psi$ at the SGG points in each partition $\F{\Gamma}_k$ are required $\forall k \in \MBK_3$ as illustrated earlier in Section \ref{subsec:BSGQMay20221}. While the corrected control values at the SGG points in each partition, $\bmu_{N_k}^c = u_{0:N_k}^{c,k} := u_{N,M}^c\left(\hat t_{N_k,0:N_k}^{(k),\alpha}\right)$, can be easily calculated through Formulas \eqref{eq:CLPR1} or \eqref{eq:CLPR2}, a difficulty arises in accurately computing the corresponding corrected state values at the SGG points, $\bms_{N_k}^c = s_{0:N_k}^{c,k}: = s^{c,k}\left(\hat t_{N_k,0:N_k}^{(k),\alpha}\right)\,\forall k \in \MBK_{3}$, where $s^{c,k}$ is the corrected state function on $\F{\Gamma}_k\,\forall k \in \MBK_3$. In particular, attempting to calculate the interpolated, predicted state values at the SGG points, $ \bms_k^p = s_{0:N_k}^{p,k} := {I_N}s^p\left(\hat t_{N_k,0:N_k}^{(k),\alpha}\right)\,\forall k \in \MBK_3$, through Eq. \eqref{eq:May0920221} to recover the necessary values of $\psi$ at the SGG points would drive the iterative method employed to solve the nonlinear system \eqref{eq:almthere1} to generate a sequence of spurious approximations to the state values at the equispaced nodes, $s_{1:N-1}$, using crude input data induced by poor approximations to the state derivative values at the SGG points, $\psi_{0:N_k}^G := \psi\left(\hat t_{N_k,0:N_k}^{(k),\alpha}\right)\,\forall k \in \MBK_{3}$, which are inherited from the noisy data $\bar{\bms}_N^p$. A viable alternative to overcome this problem is to resample the collocation points set and carry out the collocation of Eq. \eqref{eq:IDS1} at the SGG points $\hat t_{N_k,0:N_k}^{(k),\alpha}\,\forall k \in \MBK_{3}$ in lieu of the equispaced nodes $x_{N,1: N-1}$ to obtain the following nonlinear systems of equations:
\begin{equation}\label{eq:almthere2}
\bms_{N_k}^c = \bs \bmone_{N_k+1} + \C{I}_{\F{\Omega}_{\hat t_{N_k,0:N_k}^{(k),\alpha}}}^{(t)} \psi,\quad \forall k \in \MBK.
\end{equation}
To put it another way, the initial ``prediction'' phase of our proposed method predicts the optimal state- and control-values through collocation of Problem $\C{IP}$ in the Fourier physical space at the set of mesh points $\MBS_N$. The next ``correction'' phase refines the predicted values of the solutions through three steps: (i) Estimating ${\tilde \xi }_{1:2}$ of the predicted controller and reconstructing $\breve u_{N,M}$ from the FPS data using Algorithm \ref{alg:2}, (ii) correcting $\breve u_{N,M}$ through Formula \eqref{eq:CLPR1} or \eqref{eq:CLPR2} to obtain the corrected OC $u_{N,M}^c$, and (iii) correcting the predicted optimal state values $\bms_N^p$ at the equispaced nodes through collocation of Eq. \eqref{eq:IDS1} at the SGG points to obtain the corrected optimal state values $\bms_{N_k}^c$. The profile of the corrected optimal state variable $s^{c,k}$ on $\F{\Omega}_T$ can be generated via a piecewise combination of the SG interpolants defined by Eqs. \eqref{sec:ort:eq:Lagint1}, \eqref{eq:new1}, and \eqref{eq:newbaryform14May20222}. We denote the corrected optimal objective function value associated  with $s^{c,k}$ and $u_{N,M}^c$ by $J_N^c$.
%To accurately resolve the profile of the corrected optimal state variable, it is rather more judicious to evaluate the interpolant at the approximate points of discontinuities $\tilde \xi_{1:2}$ in addition to any random set of points to capture the abrupt changes at these positions.

%\begin{rem}
%The current method initially collocates Problem $\C{IP}$ in the Fourier physical space, so the computational method carried out in the prediction phase belongs to the class of direct IPS methods. On the other hand, the current method collocates the integrated dynamical system equation in the SG physical space after splitting the time domain $\F{\Omega}_T$ in the correction phase into generally three elements $\F{\Gamma}_{1:3}$ ($h$-adaptivity) while allowing to increase the SG interpolant degree on each element as desired ($p$-adaptivity), the computational method carried out in the correction phase belongs to the class of adaptive $h$-IPS numerical methods. For these reasons the current method is coined the name ``Fourier-Gegenbauer-based predictor-corrector composite IPS and adaptive $h$-IPS method'', which we prefer to abbreviate by the ``FG-PC method.''
%%``FGPC-IPS-$hp$-APS method.''
%\end{rem}

\begin{rem}
We prefer to solve the nonlinear system \eqref{eq:almthere2} using the MATLAB fsolve solver carried out using the efficient Trust-Region-Dogleg (TRD) algorithm, which is specially designed to solve nonlinear equations. To initiate the iterative method, we can evaluate $\bms_k^p$ using Eq. \eqref{eq:May0920221} and use it as an initial guess for the corrected state vector together with $\bmu_{N_k}^c$ to set up the initial guesses for $\psi_{0:N_k}^G\,\forall k \in \MBK$. Therefore, as the TRD algorithm progresses through iterations, both approximations to the state and its derivative are improved, and ultimately the sequence of solution approximations converges rapidly to $\bms_{N_k}^c$.
\end{rem}

\section{Simulation Results}
\label{sec:SRMay221}
This section shows the approximate optimal solutions and objective function values obtained through the proposed method using the following two experimental data sets $\FR{D}_1 = \{\Sin = 3, \bs = 2.9, u_{\min} = 0, u_{\max} = 2, \mu_{\max} = 1, k_s = 2.5\}$ and $\FR{D}_2 = \{\Sin = 3, \bs = 1.8377, u_{\min} = 0, u_{\max} = 2, \mu_{\max} = 1, k_s = 2.5\}$. The set $\FR{D}_1$ is a new set of parameter values that has not yet been investigated for optimizing chemostat performance, whereas this subject has been studied before for $\FR{D}_2$ in \cite{bayen2018optimal}. The numerical experiments for both data sets were carried out using MATLAB R2022a software installed on a personal laptop equipped with a 2.9 GHz AMD Ryzen 7 4800H CPU and 16 GB memory running on a 64-bit Windows 11 operating system. The constrained optimization problem \eqref{prob:DOC2} was solved using MATLAB fmincon solver with the interior-point algorithm. The fmincon solver was carried out using the stopping criteria TolFun $=$ TolX $=$ 1E-14 and the initial guesses $s(t) = \bs$ and $u(t) = \bu = 58/63$ and $36754/94869\,\forall t \in \F{\Omega}_T$ for $\FR{D}_1$ and $\FR{D}_2$, respectively, which can be calculated through Eq. \eqref{eq:bsbu1}. The nonlinear system \eqref{eq:almthere2} was solved using MATLAB fsolve solver carried out using the TRD algorithm with the stopping criteria StepTolerance $=$ TolFun $=$ 1E-15.

First, we consider the problem of optimizing the periodic performance of a chemostat for the experimental dataset $\FR{D}_1$. Figure \ref{fig:figures1} shows the approximate optimal solution plots obtained at various stages on $\F{\Omega}_{10}$ for the parameter values $N = M = 100, N_{1:3} = [16, 16, 4]$, and $\alpha = -0.1$. Both the fmincon and fsolve solvers were terminated successfully in $53$ and $4$ iterations, respectively. The corresponding median of the measured wall-clock time for the FG-PC method was approximately $1.19$ s. Notice here how the noisy data obtained in the prediction phase owing to the Gibbs effect have been successfully smoothed out in the correction phase. The distinctive feature in the profile of the corrected optimal state variable arose in the formation of an interior layer followed by a boundary layer at which a steep fall in the substrate concentration was noted near the end of the time period. In particular, $s^{c,k}$ exhibits two thin transition layers, where it varies rapidly but varies regularly and slowly in the remaining part of the domain. The corresponding corrected optimal controller is a bang-bang controller that contains only two switches, with the second switch being in close proximity to $t = T$. Figure \ref{Sim2} shows the plots of both $J_N^p$ and $J_N^c$ for increasing $N$ values, where $J_N^c$ converges to $2.407$ rounded to three decimal digits. Assuming that the time period is measured in hours (h), the approximate time switch values $\tilde \xi_i, i = 1, 2$ obtained using the current method were approximately $5.39$ and $9.95$ h, respectively. Figure \ref{fig:TRes_Sim3_2} shows the plots of the corrected optimal performance index value $J_N^c$ against the cycling time $T$ for $T = 0:40, N = 200, M = 500, N_{1:3} = [20]_3$, and $\alpha = 1/2$, in which we observe the monotonic decline of the $J_N^c$ curve as the cycling time $T$ increases before it nearly flattens as $T$ increases.

\begin{figure}
\vspace{-0.5cm}
\centering
\includegraphics[width=\textwidth]{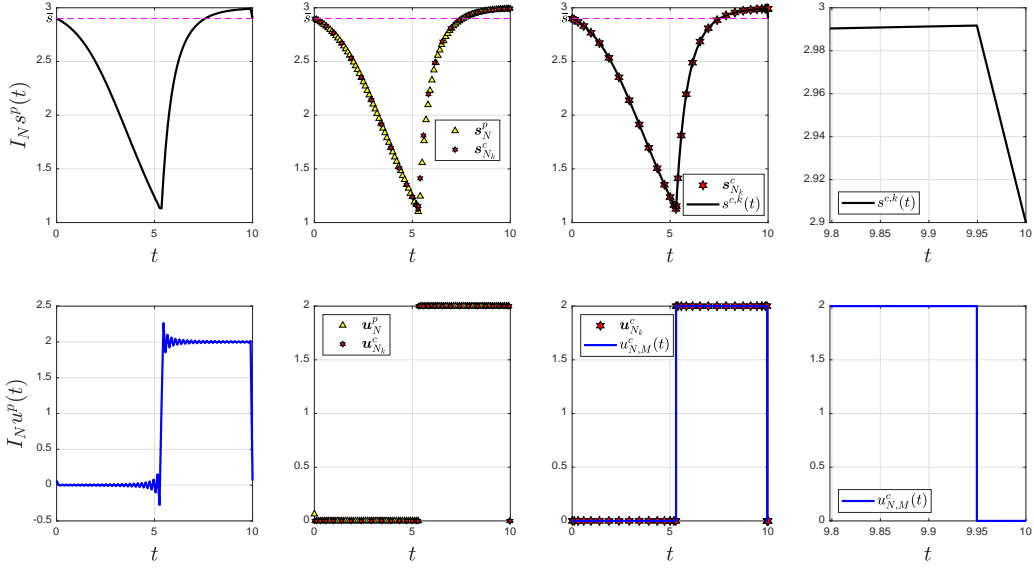}
\caption{Approximate solutions obtained by the FG-PC method using the parameter values $N = M = 100, N_{1:3} = [16, 16, 4]$, and $\alpha = -0.1$. The first column shows the plots of the predicted optimal state- and control-variables on $\F{\Omega}_{10}$ using $M+1$ equispaced points from $0$ to $10$ connected by line segments. The second column shows the predicted state- and control-variables at the equispaced collocation points set $\MBS_N$ together with their corresponding corrected values at the SGG points sets $\MBG_{N_{1:3}}^{(\alpha)}$. Column 3 shows the corrected data at the collection of SGG points sets $\MBG_{N_{1:3}}^{(\alpha)}$ and the estimated discontinuity points $\tilde \xi_i, i = 1, 2$ sorted in ascending order in addition to the interpolated, corrected optimal state- and control-variables, $s^{c,k}$ and $u_{N,M}^c$, respectively. Column 4 shows a zoom in of the corrected optimal state- and -control-variables near $t = T$. The level of the substrate concentration $\bs$ is shown in horizontal dashed line.}
\label{fig:figures1}
\end{figure}

\begin{figure}[ht]
\centering
\hspace{-2cm}\includegraphics[scale=0.45]{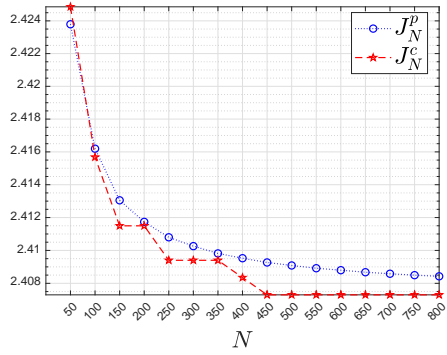}
\caption{Plots of $J_N^p$ and $J_N^c$ against $N = 50:50:800$ using the FG-PC method with $M = 1000, N_{1:3} = [30]_3$, and $\alpha = 1/2$.}
\label{Sim2}
\end{figure}

\begin{figure}[ht]
\centering
\includegraphics[scale=0.45]{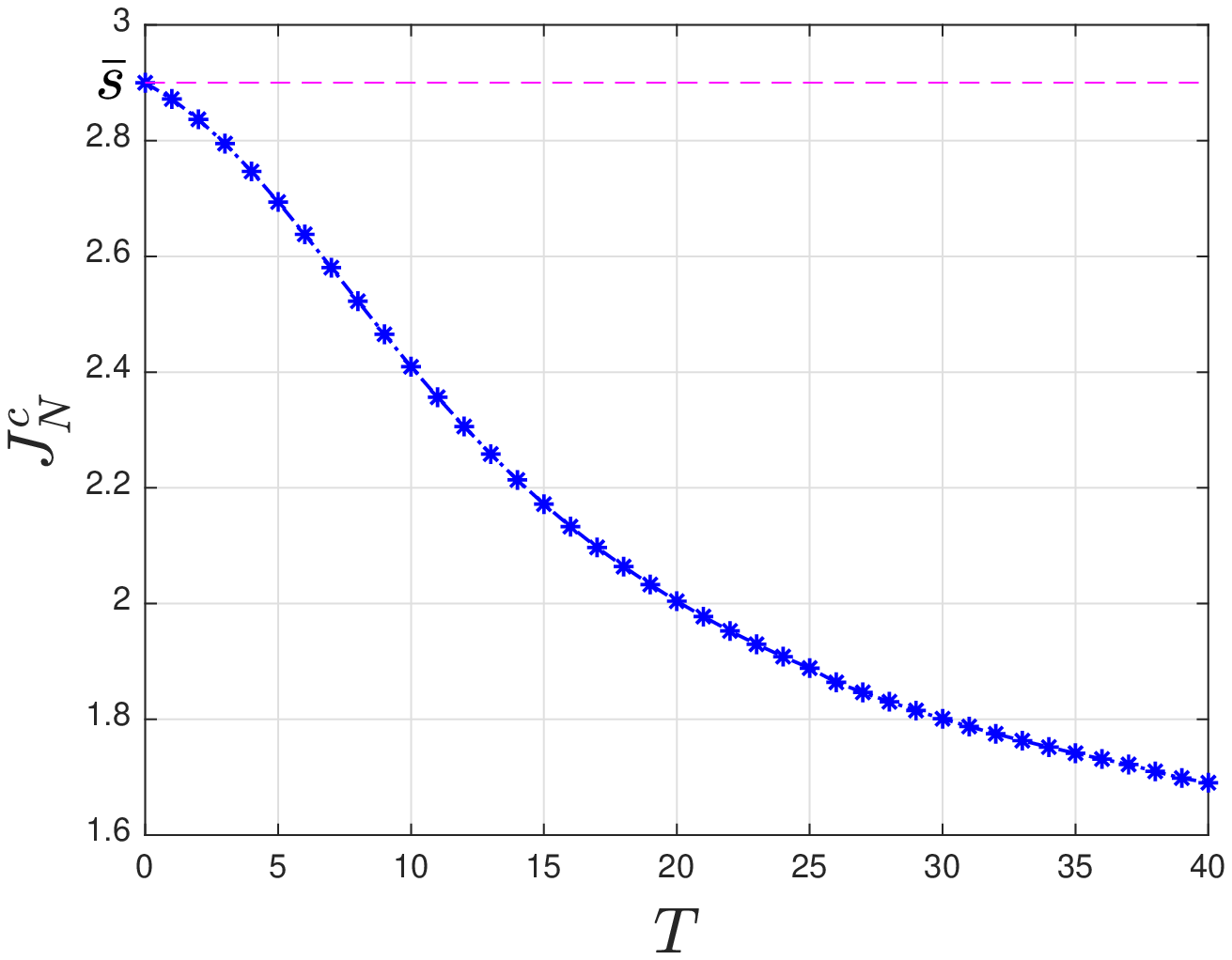}
\caption{Plots of the corrected optimal performance index $J_N^c$ against the cycling time $T$ for $T = 0:40, N = 200, M = 500, N_{1:3} = [20]_3$, and $\alpha = 1/2$.}
\label{fig:TRes_Sim3_2}
\end{figure}

Figure \ref{fig:figures1} shows the numerical simulations for experimental dataset $\FR{D}_2$. In particular, the figure shows the approximate optimal solution plots obtained using the proposed FG-PC method at various stages on $\F{\Omega}_{10}$ for the parameter values $N = 50, M = 200, N_{1:3} = [20, 20, 2]$, and $\alpha = 0.5$. Both the fmincon and fsolve solvers stopped successfully in $62$ and $10$ iterations, respectively. The corresponding median of the measured wall-clock time for the FG-PC method was approximately $0.66$ s. Note again how the correction phase successfully smoothens the noisy data of the substrate and control profiles acquired in the prediction phase. However, in contrast to the observations made for the experimental dataset $\FR{D}_1$, we observe the formation of two boundary layers in the profile of the corrected optimal state variable, at which two steep falls in the substrate concentration occur near the boundaries of each time cycle of $10$ h. The corresponding corrected optimal controller is also a bang-bang controller that contains only two switches in close proximity to $t = 0$ h and $t = 10$ h. The approximate time switch values $\tilde \xi_i, i = 1, 2$ obtained using the current method were approximately $0.025$ h and $9.4$ h, respectively. It is interesting to note that our calculated corrected optimal performance index $J_{50}^c$ value is approximately $0.678$ rounded to three decimal digits, a reduction of approximately 57\% of the estimated optimal cost function value $1.5817$ obtained earlier in \cite{bayen2018optimal}. 

\begin{figure}
\vspace{-0.5cm}
\centering
\includegraphics[width=\textwidth]{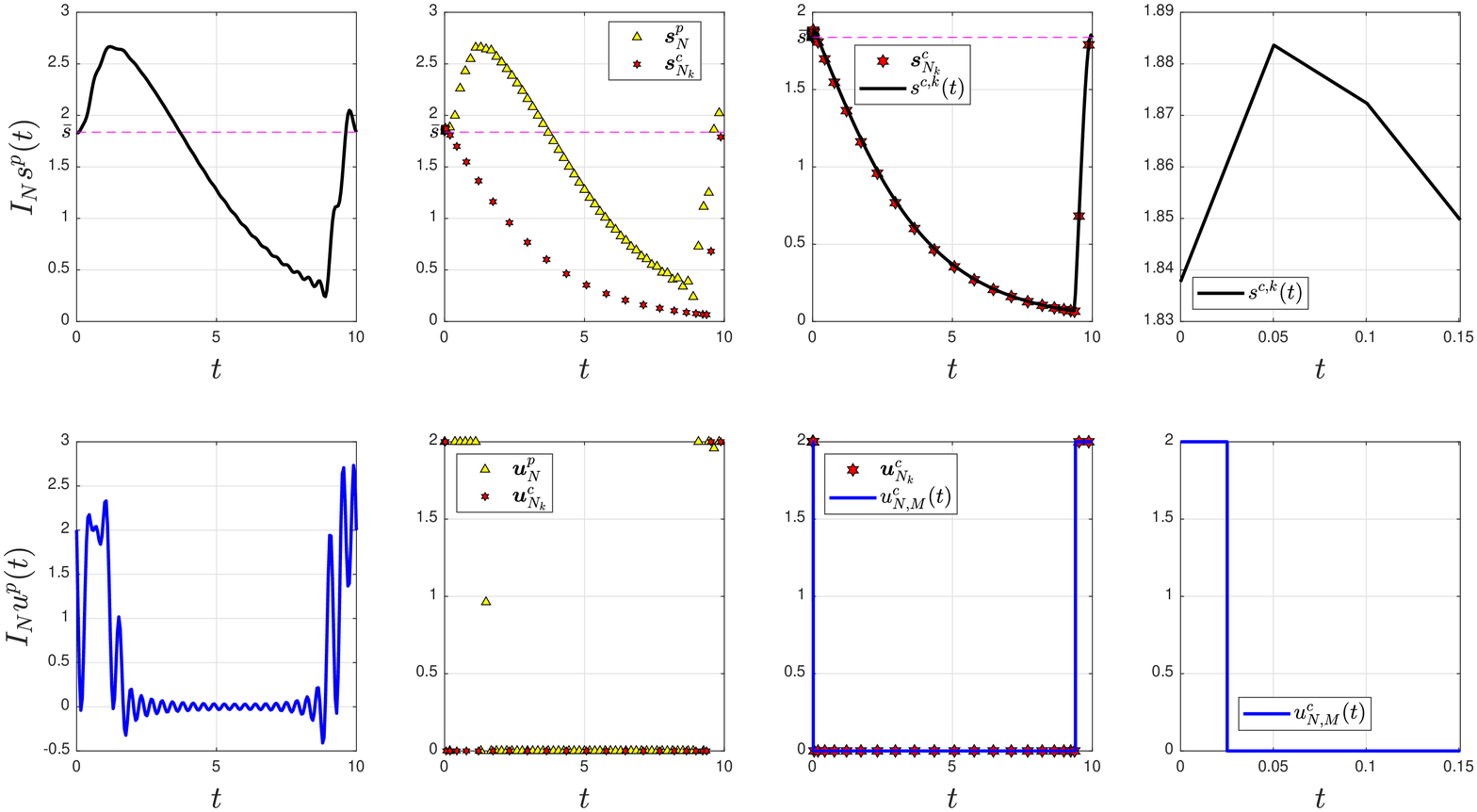}
\caption{Approximate solutions obtained by the FG-PC method using the parameter values $N = 50, M = 200, N_{1:3} = [20, 20, 2]$, and $\alpha = 0.5$. The first column shows the plots of the predicted optimal state- and control-variables on $\F{\Omega}_{10}$ using $M+1$ equispaced points from $0$ to $10$ connected by line segments. The second column shows the predicted state- and control-variables at the equispaced collocation points set $\MBS_N$ together with their corresponding corrected values at the SGG points sets $\MBG_{N_{1:3}}^{(\alpha)}$. Column 3 shows the corrected data at the collection of SGG points sets $\MBG_{N_{1:3}}^{(\alpha)}$ and the estimated discontinuity points $\tilde \xi_i, i = 1, 2$ sorted in ascending order in addition to the interpolated, corrected optimal state- and control-variables, $s^{c,k}$ and $u_{N,M}^c$, respectively. Column 4 shows a zoom in of the corrected optimal state- and -control-variables near $t = 0$. The level of the substrate concentration $\bs$ is shown in horizontal dashed line.}
\label{fig:Res1NewComp1}
\end{figure}

\section{Conclusions, Discussion and Future Work}
\label{sec:conc}
Numerical simulations demonstrate that the performance of the chemostat can be upgraded with minimal cost in terms of the time-averaged substrate concentration by adopting the OC strategy obtained through the proposed FG-PC method. The numerical simulations also manifest the decay of $J_N^c$ at a slower rate as $T$ increases, before it converges  asymptotically to a certain limit for a large cyclic time $T$. Therefore, the performance of the chemical reactor can be increased significantly as $T$ increases up to a certain $T$ limit. Part of the success of the FG-PC method owes to the accurate and efficient construction of the newly developed FPSI matrices in reduced form through Algorithm \ref{alg:1}, which can quickly generate the required FPS data in the prediction phase. Another indispensable feature of the proposed method is its ability to determine sufficiently close estimates to the jump discontinuities of the bang-bang controller and reconstruct an accurate control model through the novel Algorithm \ref{alg:2}. The derived barycentric SG quadratures proved to be highly accurate and feasible in the correction phase of the FG-PC method, and the numerical simulations demonstrated their excellent capacity together with the TRD algorithm in computing the necessary definite integrals of the reconstructed state derivative discontinuous function from the FPS data. In the absence of a priori knowledge of the OC extreme values $u_{\max}$ and $u_{\min}$, the FG-PC method generates a sequence of approximations that converge algebraically to the OC for increasing $N$ and $M$ values, as shown in Tables \ref{tab:OFICAAPOD2} and \ref{tab:OFICAAPOD3}. However, given the $u_{\max}$ and $u_{\min}$ values, as often the case with chemostat processes, the built-in adaptivity of the FG-PC method enables it to produce exact OC policies up to the vicinity of a jump discontinuity through Eqs. \eqref{eq:CLPR1} and \eqref{eq:CLPR2}.

The current study asserts that the optimal substrate concentration corresponding to the optimal dilution rate of the experimental dataset $\FR{D}_1$ exhibits two thin transition layers: one interior layer and one boundary layer near the end of the time period, where it varies rapidly, but varies regularly and slowly in the remaining part of the domain. The associated optimal periodic control has exactly two switching times, and the optimal controller should be defined by Formula \eqref{eq:CLPR2} such that the optimal dilution rate remains zero for approximately half the time period until the occurrence of the first switching time. The policy of the FG-PC method recommends that the chemostat should initially pass through a starvation phase, in which no nutrient flow to the chemostat culture is allowed for approximately $5.39$ h. Because $\mu_{\max} > \bu$, the utilization of the substrate exceeds the initial supply of the substrate in the absence of washout\footnote{The washout term means the mass flow rate of cells that leave with the outgoing stream.}, causing the microorganism population to monotonically increase at a rapid pace, whereas the substrate concentration is consumed until it nearly runs out by the end of the same period. The current OC policy indicates that the dilution rate should shift abruptly to $2$ h$^{-1}$ and remain constant at this level for approximately $4.56$ h. The supply of substrate added with the inflowing fresh medium during this phase exceeds the demand of the substrate in the presence of an increasing washout rate, causing the substrate concentration to increase rapidly, while the microbes can no longer reproduce fast enough to maintain a population as they are continuously being washed out of the vessel until their population reaches nearly an extinction level at the end of this period. For the next $3$ minutes, the optimal dilution rate shifted back to zero, and the media feed stopped. During this last phase, the substrate concentration is initially at its peak, while the number of cells is at its minimum, so substrate depletion is substantial, lowering the concentration of the limiting substrate rapidly while the cell population grows swiftly until they both return to their initial states by the end of the time period.

The problem of optimizing the chemostat performance using the experimental dataset $\FR{D}_2$ was solved numerically in \cite{bayen2018optimal} using BOCOP with different initial guesses, where the optimal periodic control of the problem varied abruptly at only two switching times. In particular, the calculated optimal substrate was allowed to increase monotonically for nearly $1.5$ h until it reached its peak at nearly $2.8$. The optimal controller then shifts abruptly from $2$ h$^{-1}$ to $0$ h$^{-1}$, and remains constant at this level for approximately $8$ h. During this period, the substrate concentration profile continued to decrease slowly until it reached its base value of approximately $0.3$. Finally, the optimal controller switches back to its former state, and the substrate concentration grows rapidly owing to the constant flow of fresh media and nutrients into the system until it reaches its initial state. The optimal performance index value owing to this control policy was estimated to be $1.5817$. However, the developed FG-PC method reveals that the optimal dilution rate should remain at its peak for nearly $1.5$ min. until the first switching time, at which the media feed stops. The chemostat should then pass through a starvation phase in which no nutrient flow to the chemostat culture is allowed for approximately $9.37$ h. During this second phase, the substrate concentration is initially at its peak, whereas the number of cells is at its minimum. Therefore, the microorganism population monotonically increase at a rapid pace, while the substrate concentration is consumed until it nearly runs out by the end of the same period. The proposed OC policy shows that the dilution rate should shift abruptly back to $2$ h$^{-1}$ and remain constant at this level for approximately $36.18$ min, and the supply of substrate added with the inflowing fresh medium during this phase causes the substrate concentration to increase rapidly, while the microbes can no longer reproduce fast enough to maintain a population as they are continuously being washed out of the vessel until their population reaches nearly an extinction level at the end of this period. The profile of the optimal substrate concentration corresponding to the optimal dilution rate in this case exhibited two thin boundary layers, where it varied rapidly but varied regularly and slowly in the remaining part of the domain. It is noteworthy to mention here the reported corrected optimal performance index value $J_{50}^c \approx 0.678$ in Section \ref{sec:SRMay221}, which was obtained using the FG-PC method. Compared with the estimated optimal cost function value $1.5817$ obtained earlier in \cite{bayen2018optimal}, our recommended OC policy gives rise to a new OC policy for the optimal periodic performance of a chemostat in which the optimal cost function value can be reduced by approximately 57\%.

Another important result from a numerical viewpoint that could be added to the major contributions of this work is that the numerical results derived in \cite{bayen2018optimal} were obtained using a local optimization solver with random initial guesses, whereas the current FG-PC method employs automatically calculated and sufficiently close initial guesses at the outset of the correction phase, which adds more support and credibility to the new study results presented in this paper. A major contribution of this study lies in the introduction of a novel predictor-corrector approach combining IPS and adaptive $h$-IPS methods for the derivation of OC policies of periodically operated biochemical reactors. One potential direction in the future is to investigate the possibility of adding the time period $T$ as another optimization variable, and then probe how that would affect the proposed method. 

\section*{Limitation}
Although the OC policies derived in this work can maximize the chemostat performance overall, caution must be exercised when applying these strategies because slight changes in the process path may cause severe process instabilities. For example, a shift to a slightly higher dilution rate in the second phase using experimental data $\FR{D}_1$ may lead to a complete washout of the cells. 

\section*{Declarations}
\subsection*{Competing Interests}
The author declares there is no conflict of interests.

\subsection*{Availability of Supporting Data}
The author declares that the data supporting the findings of this study are available within the article.

\subsection*{Ethical Approval and Consent to Participate}
Not Applicable.

\subsection*{Human and Animal Ethics}
Not Applicable.

\subsection*{Consent for Publication}
Not Applicable.

\subsection*{Funding}
The author received no financial support for the research, authorship, and/or publication of this article.

\subsection*{Authors' Contributions}
The author confirms sole responsibility for the following: study conception and design, data collection, analysis and interpretation of results, and manuscript preparation.

\subsection*{Acknowledgments}
Not applicable.

\appendix
\section{Error and Convergence Analyses}
\label{sec:ECAnKK1}
In this section, we derive rigorous formulas for the errors and convergence rates associated with the two sets of numerical discretization tools introduced in Section \ref{sec:NDOnew1}.

\subsection{Errors and Convergence Rates of Fourier Interpolation and Quadrature for Nonsmooth and Generally $T$-Periodic Functions} 
\label{subsec:CRFNSPF1}
In this section, we study the errors and convergence rates of Fourier interpolation and integration operators for nonsmooth and generally $T$-periodic functions. To this end, let $\bm{\beta} = [-\beta, \beta]\,\forall \beta > 0, {\F{C}_{T,\beta} } = \left\{ {x + iy:x \in {\F{\Omega}_T},y \in \bm{\beta}} \right\}\forall \beta  > 0, C^n({\F{\Omega}_T})$ be the space of $n$ times continuously differentiable functions on ${\F{\Omega}_T}\,\forall n \in \MBZzerP$, $L^p({\F{\Omega}_T})$ is the Banach space of measurable functions $u$ defined on ${\F{\Omega}_T}$ such that ${\left\| u \right\|_{{L^p}}} = {\left( {{\C{I}_{\F{\Omega}_T}}{{\left| u \right|}^p}} \right)^{1/p}} < \infty$, and 
\[\displaystyle{{H^s}({\F{\Omega}_T}) = \left\{ {u \in {L_{loc}}({\F{\Omega}_T}),\;{D^\alpha }u \in {L^2}({\F{\Omega}_T}),\left| \alpha  \right| \le s\;} \right\}}\,\forall s \in \MBZzerP,\]
is the inner product space with the inner product $\displaystyle{{(u,v)_s} = \sum\nolimits_{\left| \alpha  \right| \le s} {\C{I}_{{\F{\Omega}_T}} ^{(x)}\left( {{D^\alpha }u\,{D^\alpha }v} \right)}}$, where ${{L_{loc}}({\F{\Omega}_T} )}$ is the space of locally integrable functions on ${\F{\Omega}_T}$ and ${{D^\alpha }u}$ denotes any derivative of $u$ with multi-index $\alpha$. Let also 
\[\C{H}_T^s = \left\{ {u \in {H^s}({\F{\Omega}_T}),\;{u^{(s)}} \in {BV},\;{u^{(0:s - 1)}}(0) = {u^{(0:s - 1)}}(T)} \right\},\]
where $\displaystyle{{BV} = \left\{ {u \in {L^1}({\F{\Omega}_T}):{{\left\| u \right\|}_{BV}} < \infty } \right\}}$ with the norm $\displaystyle{{{\left\| u \right\|}_{BV}} = \sup \left\{ {\C{I}_T^{(x)}(u\phi '),\;\phi  \in \C{D}({\F{\Omega}_T}),\;{{\left\| \phi  \right\|}_{{L^\infty }}} \le 1} \right\}}$ such that\\ $\C{D}({\F{\Omega}_T}) = \left\{ {u \in {C^\infty }({\F{\Omega}_T}):{\text{supp}}(u){\text{ is a compact subset of }}{\F{\Omega}_T}} \right\}$. For convenience of writing, we shall denote ${\left\|  \cdot  \right\|_{{L^2}({\F{\Omega}_T})}}$ and $e^{i \omega_k x}$ by $\left\|  \cdot  \right\|$ and $\phi_k(x)\,\forall k$, respectively. We shall first derive the decay rate of Fourier series coefficients for functions in ${{\C{H}_T^s}}\,\forall s \in \MBZzerP$. Note that functions in this space are $T$-periodic and continuous (but nonsmooth) for $s \ge 1$. However, functions in ${{\C{H}_T^0}}$ may neither be $T$-periodic nor continuous.

\begin{thm}[Decay of Fourier Series Coefficients for nonsmooth and generally $T$-periodic functions]\label{thm:00}
Suppose that $f \in {{\C{H}_T^s}} \foralls s \in \MBZzerP$ is approximated by the $N/2$-degree, $T$-periodic truncated Fourier series 
\begin{equation}\label{eq:FTS1}
{\Pi _N}f(x) = \sum\limits_{\left| k \right| \le N/2} {{{\hat f}_k}{\phi_k(x)}},\quad \foralls N \in \MBZeP,
\end{equation}
where $\hat f_{-N/2:N/2}$ is the Fourier series coefficients vector of $f$, then 
\begin{equation}\label{eq:fhatkMAR312021new10}
\left| {{{\hat f}_k}} \right| = O\left( {{{\left| {{k}} \right|}^{ - s - 1}}} \right),\quad \text{ as }\left|k\right| \to \infty.
\end{equation}
\end{thm}
\begin{proof}
Notice first that the set of complex exponentials $\displaystyle{\left\{ {{\phi _{- N/2:N/2}}} \right\}}$ is orthogonal on ${\F{\Omega}_T}$ with respect to the weight function $w(x) = 1\,\forall x \in {\F{\Omega}_T}$ such that $\left( {{\phi _n},{\phi _m}} \right) = \C{I}_T^{(x)}\left( {{\phi _n}\,\phi _m^*} \right) = T{\delta _{n,m}}$, where $\delta _{n,m}$ is the Kronecker delta function defined by 
\[{\delta _{n,m}} = \left\{ \begin{array}{l}
1,\quad n = m,\\
0,\quad n \ne m.
\end{array} \right.\]
Therefore, $\left( {{\phi _n},{\phi _n}} \right) = \C{I}_T^{(x)}\left( {{\phi _n}\,\phi _n^*} \right) = \C{I}_T^{(x)}\left( {{{\left| {{\phi _n}} \right|}^2}} \right) = {\left\| {{\phi _n}} \right\|^2} = T$. Fourier coefficients, $\hat f_k$, of $f$ can thus be determined via the
orthogonal projection $(f,\phi_k)$, which produces
\begin{equation}\label{eq:19APr2021_1}
{\hat f_k} = \frac{1}{T}(f,{\phi _k}) = \frac{1}{T}\C{I}_T^{(x)}\left( {f{\mkern 1mu} {\phi _{ - k}}} \right) = \left\{ \begin{array}{l}
\frac{1}{T}\C{I}_T^{(x)}{\left(f\cancbra{x - i\beta}{\mkern 1mu} {\phi _{ - k}}\cancbra{x - i\beta}\right)}\;\forall k \ge 0,\\
\frac{1}{T}\C{I}_T^{(x)}\left({f\cancbra{x + i\beta}{\mkern 1mu} {\phi _{ - k}}\cancbra{x + i\beta}} \right)\;\forall k < 0
\end{array} \right. = \left\{ \begin{array}{l}
\frac{{{e^{ - {\omega _{k\beta }}}}}}{T}\C{I}_T^{(x)}\left({f\cancbra{x - i\beta}{\mkern 1mu} {\phi _{ - k}}} \right)\;\forall k \ge 0,\\
\frac{{{e^{ - {\omega _{ - k\beta }}}}}}{T}\C{I}_T^{(x)}\left({f\cancbra{x + i\beta}{\mkern 1mu} {\phi _{ - k}}} \right)\;\forall k < 0.
\end{array} \right.
\end{equation}
Through Eq. \eqref{eq:19APr2021_1} and integration by parts, we have
\begin{equation}\label{eq:fhatkMAR312021new100}
{{\hat f}_k} = \frac{1}{T}\C{I}_T^{(x)}\left( {f{\phi _{ - k}}} \right) = \frac{1}{{T{{(i{\omega _k})}^s}}}\C{I}_T^{(x)}\left( {{f^{(s)}}{\phi _{ - k}}} \right) = \frac{1}{{T{{(i{\omega _k})}^{s + 1}}}}\C{I}_T^{(x)}\left( {{f^{(s)}}{\phi'_{ - k}}} \right)\;\forall k \in \MB{K}_N\backslash\{0\} \Rightarrow \left| {{{\hat f}_k}} \right| \le \frac{{{{\left\| {{f^{(s)}}} \right\|}_{BV}}}}{{T{\omega_{\left|k\right|}^{s + 1}}}},\quad \forall k \in \MB{K}_N\backslash\{0\},
\end{equation}
from which the Asymptotic Formula \eqref{eq:fhatkMAR312021new10} immediately follows. Further, $\displaystyle{\hat f_0 = \frac{1}{T} \C{I}_T^{(x)} {f}}$, i.e., $\hat f_0$ is the average value of the function.
\end{proof}
We refer to the upper bound \eqref{eq:fhatkMAR312021new100} by the ``Fourier coefficients upper bounds for nonsmooth and $T$-periodic functions.'' 

\begin{thm}[Fourier truncation error for nonsmooth and generally $T$-periodic functions]\label{thm:11}
Suppose that $f \in {{\C{H}_T^s}} \foralls s \in \MBZzerP$ is approximated by the $N/2$-degree, $T$-periodic truncated Fourier series \eqref{eq:FTS1}, then 
\begin{equation}\label{eq:Thm11}
{\left\| {f - {\Pi _N}f} \right\|} = O\left(N^{-s-1/2}\right),\quad \text{ as }N \to \infty.
\end{equation}
\end{thm}
\begin{proof}
Observe first that
\begin{equation}\label{eq:thm1ineqfor01Apr20212}
{\left\| {f - {\Pi _N}f} \right\|^2} = \C{I}_T^{(x)}\left( {\sum\limits_{\left| k \right| > N/2} {{{\hat f}_k}{\phi _k}} \sum\limits_{\left| l \right| > N/2} {\hat f_l^*{\phi _{ - l}}} } \right) = \sum\limits_{\left| k \right| > N/2} {\sum\limits_{\left| l \right| > N/2} {{{\hat f}_k}\hat f_l^*} } \C{I}_T^{(x)}{{\phi _{k - l}}} = \sum\limits_{\left| k \right| > N/2} {\sum\limits_{\left| l \right| > N/2} {{{\hat f}_k}\hat f_l^*} ({\phi _k},{\phi _l})} = T\sum\limits_{\left| k \right| > N/2} {{{\left| {{{\hat f}_k}} \right|}^2}}.
\end{equation}
Through Eq. \eqref{eq:fhatkMAR312021new100} and Ineq. \eqref{eq:thm1ineqfor01Apr20212}, we have
\begin{align}
{\left\| {f - {\Pi _N}f} \right\|^2} &= T\sum\limits_{\left| k \right| > N/2} {{{\left| {{{\hat f}_k}} \right|}^2}} \le \frac{1}{T}\left\| {{f^{(s)}}} \right\|_{BV}^2\sum\limits_{\left| k \right| > N/2} {\omega _k^{ - 2s - 2}}  \le \frac{2}{T}\left\| {{f^{(s)}}} \right\|_{BV}^2  \C{I}_{N/2,\infty }^{(x)}\omega _x^{ - 2s - 2}  = \frac{{\left\| {{f^{(s)}}} \right\|_{BV}^2}}{{(2s + 1)\pi \omega _{N/2}^{2s + 1}}}\quad \forall N \in Z_e^ + .\nonumber\\
\Rightarrow \left\| {f - {\Pi _N}f} \right\| &\le \frac{{{{\left\| {{f^{(s)}}} \right\|}_{BV}}}}{{\sqrt {(2s + 1)\pi } \omega _{N/2}^{s + 1/2}}}\quad \forall N \in Z_e^ +.\label{eq:hello1}
\end{align}
\end{proof}

\begin{thm}[Fourier aliasing error for nonsmooth and $T$-periodic functions]\label{thm:22} Suppose that $f \in {{\C{H}_T^s}} \foralls s \in \MBZ^+$ is approximated by the $T$-periodic Fourier interpolant $I_Nf\,\foralls N \in \MBZeP$, then 
\begin{equation}\label{eq:FTS1AE1new1}
\left\| {{E_Nf}} \right\| = O\left( {{N^{-s-1/2}}} \right),\quad \text{as }N \to \infty.
\end{equation}
\end{thm}
\begin{proof}
Replacing $f$ in \eqref{DFP1} by its Fourier series yields
\begin{equation}\label{eq:ftildekMAR312021}
{\tilde f_k} = \frac{1}{N}\sum\limits_{j = 0}^{N - 1} {\left[ {\sum\limits_{l \in \MBZ} {{{\hat f}_l}{\phi _l}({x_j})} } \right]{\phi _{ - k}}({x_j})}  = \sum\limits_{l \in \MBZ} {{{\hat f}_l}\left[ {\frac{1}{N}\sum\limits_{j = 0}^{N - 1} {{\phi _{l - k}}({x_j})} } \right]}  = {\left[ {\sum\limits_{l \in \MBZ} {{{\hat f}_l}{\delta _{l - k,pN}}} } \right]_{\left| p \right| \in \MBZzer^ + }} = {\hat f_k} + \sum\limits_{p \in \MBZ_{\hcancel{0}}} {{{\hat f}_{k + pN}}} ,\quad \forall k \in {\MB{K}'_N}.
\end{equation}
Formula \eqref{eq:ftildekMAR312021} and the Triangle Difference Ineq. imply that
\begin{align}
&{\left| {\left\| {{E_Nf}} \right\| - \left\| {{{\hat f}_{N/2}}{\phi _{N/2}}} \right\|} \right|^2} \le {\left\| {{E_Nf} - {{\hat f}_{N/2}}{\phi _{N/2}}} \right\|^2} = \C{I}_T^{(x)}\left( {\sumd\sum\limits_{\left| k \right| \le N/2} {\sum\limits_{p \in \MBZ_{\hcancel{0}}} {{{\hat f}_{k + pN}}{\phi _k}} }  \cdot \sumd\sum\limits_{\left| l \right| \le N/2} {\sum\limits_{p \in \MBZ_{\hcancel{0}}} {\hat f_{l + pN}^*{\phi _{ - l}}} } } \right)\nonumber\\
&= \sumd\sum\limits_{\left| k \right| \le N/2} {\sumd\sum\limits_{\left| l \right| \le N/2} {\sum\limits_{p \in \MBZ_{\hcancel{0}}} {{{\hat f}_{k + pN}}\sum\limits_{p \in \MBZ_{\hcancel{0}}} {\hat f_{l + pN}^*} } } } \C{I}_T^{(x)}{\phi _{k - l}} = \sumd\sum\limits_{\left| k \right| \le N/2} {{{\left| {\sum\limits_{p \in \MBZ_{\hcancel{0}}} {{{\hat f}_{k + pN}}} } \right|}^2}{{\left\| {{\phi _k}} \right\|}^2}}  = T\sumd\sum\limits_{\left| k \right| \le N/2} {{{\left| {\sum\limits_{p \in \MBZ_{\hcancel{0}}} {{{\hat f}_{k + pN}}} } \right|}^2}}\nonumber\\
&\le T\sumd\sum\limits_{\left| k \right| \le N/2} {\sum\limits_{p \in \MBZ_{\hcancel{0}}} {{{\left| {{{\hat f}_{k + pN}}} \right|}^2}} }.\label{eq:proof1}
\end{align}
From Ineqs. \eqref{eq:fhatkMAR312021new100} and \eqref{eq:proof1}, we find that
\begin{align*}
&{\left| {\left\| {{E_Nf}} \right\| - \left\| {{{\hat f}_{N/2}}{\phi _{N/2}}} \right\|} \right|^2} \le T\sumd\sum\limits_{\left| k \right| \le N/2} {\sum\limits_{p \in \MBZ_{\hcancel{0}}} {{{\left| {{{\hat f}_{k + pN}}} \right|}^2}} } \le \frac{\left\| f^{(s)} \right\|_{BV}^2}{T}\sumd\sum\limits_{\left| k \right| \le N/2} {\sum\limits_{p \in \MBZ_{\hcancel{0}}} {{{\omega _{\left|k + pN\right|}}}^{ - 2s - 2}} }\\
&= \frac{\left\| f^{(s)} \right\|_{BV}^2}{T}\left[ 2 {\sum\limits_{k = 0}^{N/2} {{\sum\limits_{p \in \MBZ_{\hcancel{0}}} {{\omega _{\left|k + pN\right|}^{ - 2s - 2}}} }}  - \sum\limits_{p \in \MBZ_{\hcancel{0}}} \left(\omega _{\left| {pN} \right|}^{ - 2s - 2} + \omega _{\left| {N/2 + pN} \right|}^{ - 2s - 2}\right) } \right]\\
&\le \frac{\left\| f^{(s)} \right\|_{BV}^2}{T}\left[2\,{\omega_1^{-2s - 2}}\sum\limits_{k = 0}^{N/2} {\frac{1}{{{N^{2s + 2}}}}\sum\limits_{p \ge 1} {\left( {\frac{1}{{{{(p - 1/2)}^{2s + 2}}}}+\frac{1}{{{p^{2s + 2}}}}} \right)} }\right.\\
&\left. - \omega _N^{ - 2s - 2}\left( {\sum\limits_{p \ge 1} {\left( {\frac{1}{{{{(p - 1/2)}^{2s + 2}}}} + \frac{2}{{{p^{2s + 2}}}} + \frac{1}{{{{(p + 1/2)}^{2s + 2}}}}} \right)} } \right) \right]
\end{align*}
\begin{align*}
&= \frac{\left\| f^{(s)} \right\|_{BV}^2}{T}\left[(N + 2)\zeta (2s + 2)\omega _{N/2}^{ - 2s - 2} - \left( {2\zeta (2s + 2) - 1} \right)\omega _{N/2}^{ - 2s - 2} \right] = \frac{1}{\pi }\left( {\zeta (2s + 2) + \frac{1}{N}} \right) \left\| {{f^{(s)}}} \right\|_{BV}^2 \omega _{N/2}^{ - 2s - 1}\quad \forall N \in \MBZeP,
\end{align*}
where $\zeta$ is the Riemann zeta function. Therefore,
\[
{\left| {\left\| {{E_Nf}} \right\| - \left\| {{{\hat f}_{N/2}}{\phi _{N/2}}} \right\|} \right|} \le \sqrt {\frac{1}{\pi }\left( {\zeta (2s + 2) + \frac{1}{N}} \right)}\,{\left\| {{f^{(s)}}} \right\|_{BV}}\,\omega _{N/2}^{ - s - 1/2}\quad \forall N \in \MBZeP.
\]
Since $\displaystyle{\left\| {{{\hat f}_{N/2}}{\phi _{N/2}}} \right\| = \sqrt T \left| {{{\hat f}_{N/2}}} \right| \le \frac{{{\mkern 1mu} {{\left\| {{f^{(s)}}} \right\|}_{BV}}}}{{\sqrt T {\kern 1pt} \omega _{N/2}^{s+1}}}}$ by Ineq. \eqref{eq:fhatkMAR312021new100}, then 
\begin{equation}\label{eq:thm2ineqfor26Apr20211}
\left\| {{E_N}f} \right\| \le \left\| {{{\hat f}_{N/2}}{\phi _{N/2}}} \right\| + \left\| {{E_N}f - {{\hat f}_{N/2}}{\phi _{N/2}}} \right\| \le \frac{1}{{\sqrt \pi  }} \left( {\sqrt {\zeta (2s + 2) + \frac{1}{N}}  + \frac{1}{{\sqrt N }}} \right) {\left\| {{f^{(s)}}} \right\|_{BV}} \omega _{N/2}^{ - s - 1/2},
\end{equation}
whence the Asymptotic Formula \eqref{eq:FTS1AE1new1} is obtained.
\end{proof}
Since $\zeta(s) < \displaystyle{\frac{1}{{1 - {2^{1 - s}}}}}\,\forall s > 2$ \cite{Batir2008a}, then the aliasing error is roughly bounded by
\begin{equation}\label{eq:thm2ineqfor26Apr20222}
\left\| {{E_N}f} \right\| < \frac{1}{{\sqrt \pi  }}\left( {\sqrt {1 + \frac{1}{{{2^{2s + 1}} - 1}} + \frac{1}{N}}  + \frac{1}{{\sqrt N }}} \right){\left\| {{f^{(s)}}} \right\|_{BV}}\omega _{N/2}^{ - s - 1/2} \sim \sqrt {\frac{{{2^{2s + 1}}}}{{\pi \left( {{2^{2s + 1}} - 1} \right)}}} {\left\| {{f^{(s)}}} \right\|_{BV}}\omega _{N/2}^{ - s - 1/2},\quad \text{as } N \to \infty.
\end{equation}
Theorem \ref{thm:22} demonstrates that the aliasing error is comparable to the Fourier series truncation error. The next corollary shows that the Fourier interpolation error is comparable to the Fourier series truncation error.

\begin{cor}[Fourier interpolation error for nonsmooth and $T$-periodic functions]\label{cor:11}
Suppose that $f \in {{\C{H}_T^s}} \foralls s \in \MBZ^+$ is approximated by the $T$-periodic Fourier interpolant $I_Nf\,\foralls N \in \MBZeP$, then 
\begin{equation}\label{eq:FTS1AE1hi126Apr2021}
\left\| f - {{I_Nf}} \right\| = O\left( {{N^{-s-1/2}}} \right),\quad \text{as }N \to \infty.
\end{equation}
\end{cor}
\begin{proof}
Ineqs. \eqref{eq:hello1} and \eqref{eq:thm2ineqfor26Apr20211} yield
\[{\left\| {f - {I_N}f} \right\|^2} = {\left\| {f - {\Pi _N}f} \right\|^2} + {\left\| {{E_N}f} \right\|^2} \le \frac{{\left\| {{f^{(s)}}} \right\|_{BV}^2}}{{(2s + 1)\pi \omega _{N/2}^{2s + 1}}} + \frac{1}{\pi }{\left( {\sqrt {\zeta (2s + 2) + \frac{1}{N}}  + \frac{1}{{\sqrt N }}} \right)^2}\left\| {{f^{(s)}}} \right\|_{BV}^2\omega _{N/2}^{ - 2s - 1}\]
\vspace{-5mm}
\begin{align}
\Rightarrow \left\| {f - {I_N}f} \right\| &\le \nu_{1,s,N} {\left\| {{f^{(s)}}} \right\|_{BV}}\omega _{N/2}^{ - s - 1/2}\label{eq:Sams1}\\
&< \nu_{2,s,N} {\left\| {{f^{(s)}}} \right\|_{BV}}\omega _{N/2}^{ - s - 1/2} \sim \nu_{3,s} {\left\| {{f^{(s)}}} \right\|_{BV}}\omega _{N/2}^{ - s - 1/2},\quad \text{as }N \to \infty,
\end{align}
where
\begin{align*}
\nu_{1,s,N} &= \sqrt {\frac{1}{\pi }\left[ {\frac{1}{{2s + 1}} + {{\left( {\sqrt {\zeta (2s + 2) + \frac{1}{N}}  + \frac{1}{{\sqrt N }}} \right)}^2}} \right]},\\
\nu_{2,s,N} &= \sqrt {\frac{1}{\pi }\left[ {\frac{1}{{2s + 1}} + {{\left( {\sqrt {1 + \frac{1}{{{2^{2s + 1}-1}}} + \frac{1}{N}}  + \frac{1}{{\sqrt N }}} \right)}^2}} \right]},\quad \text{and}\\
\nu_{3,s} &= \sqrt {\frac{1}{\pi }\left( {1 + \frac{1}{{2s + 1}} + \frac{1}{{{2^{2s+1}-1}}}} \right)}.
\end{align*}
\end{proof}

\begin{cor}[FPSQ error for nonsmooth and $T$-periodic functions]\label{cor:22}
Suppose that $f \in {{\C{H}_T^s}} \foralls s \in \MBZ^+$ is approximated by the $T$-periodic Fourier interpolant $I_Nf\,\foralls N \in \MBZeP$, then 
\begin{equation}\label{eq:FPSQEFATPF126Apr2021Hi1}
\left| {\C{I}_{{{\bm{x}_N}}}^{(x)}f - \Fthe f_{0:N - 1}} \right| = O\left( {{N^{-s-1/2}}} \bmone_N \right),\quad \text{as }N \to \infty.
\end{equation}
\end{cor}
\begin{proof}
The Triangle Ineq. implies
\begin{align}
\left| {\C{I}_{{{\bm{x}_N}}}^{(x)}f - \Fthe f_{0:N - 1}} \right| &= \left| {\C{I}_{{{\bm{x}_N}}}^{(x)}f - \C{I}_{{{\bm{x}_N}}}^{(x)}(I_Nf) + \C{I}_{{{\bm{x}_N}}}^{(x)}(I_Nf) - \Fthe f_{0:N - 1}} \right| \le \left| {\C{I}_{{{\bm{x}_N}}}^{(x)}f - \C{I}_{{{\bm{x}_N}}}^{(x)}(I_Nf)}\right| + \left|{\C{I}_{{{\bm{x}_N}}}^{(x)}(I_Nf) - \Fthe f_{0:N - 1}} \right|.\quad\label{eq:pointhere1}
\end{align}
The proof is established through Ineqs. \eqref{eq:Sams1} and \eqref{eq:pointhere1} together with Schwarz's Ineq. by realizing that
\begin{align*}
&\left| {\C{I}_{{{\bm{x}_N}}}^{(x)}f - \Fthe f_{0:N - 1}} \right| \le \left| {\C{I}_{{{\bm{x}_N}}}^{(x)}f - \C{I}_{{{\bm{x}_N}}}^{(x)}(I_Nf)}\right| + \left|{\C{I}_{{{\bm{x}_N}}}^{(x)}(I_Nf) - \Fthe f_{0:N - 1}} \right| \le \left\| f - {{I_Nf}} \right\| \sqrt{\bm{x}_N}\\
&\le \nu_{1,s,N} {\left\| {{f^{(s)}}} \right\|_{BV}}\omega _{N/2}^{ - s - 1/2} \sqrt{\bm{x}_N} < \nu_{2,s,N} {\left\| {{f^{(s)}}} \right\|_{BV}}\omega _{N/2}^{ - s - 1/2} \sqrt{\bm{x}_N} \sim \nu_{3,s} {\left\| {{f^{(s)}}} \right\|_{BV}}\omega _{N/2}^{ - s - 1/2} \sqrt{\bm{x}_N},\quad \text{as }N \to \infty . 
\end{align*}
\end{proof}
\noindent Corollary \ref{cor:22} shows that
\begin{subequations}
\begin{align}
\left\| {\C{I}_{{{\bm{x}_N}}}^{(x)}f - \Fthe f_{0:N - 1}} \right\|_2 &\le \nu_{1,s,N} {\left\| {{f^{(s)}}} \right\|_{BV}}\omega _{N/2}^{ - s - 1/2}\,\left\|\sqrt{\bm{x}_N}\right\|_2\label{eq:UPPa}\\
&< \nu_{2,s,N} {\left\| {{f^{(s)}}} \right\|_{BV}}\omega _{N/2}^{ - s - 1/2}\,\left\|\sqrt{\bm{x}_N}\right\|_2\label{eq:UPPb}\\
&\sim \nu_{3,s} {\left\| {{f^{(s)}}} \right\|_{BV}}\omega _{N/2}^{ - s - 1/2}\,\left\|\sqrt{\bm{x}_N}\right\|_2,\quad \text{as }N \to \infty.\label{eq:UPPc} 
\end{align}
\end{subequations}
We refer to the upper bounds \eqref{eq:UPPa}, \eqref{eq:UPPb}, and \eqref{eq:UPPc} by the ``FPSQ-NSTP error norm upper bound,'' ``relaxed FPSQ-NSTP error norm upper bound,'' and ``asymptotic FPSQ-NSTP error norm upper bound,'' respectively, or collectively by the ``FPSQ error norm upper bounds for nonsmooth and $T$-periodic functions.'' 
Note that each of the error factors $\nu_{1,s,N}, \nu_{2,s,N}$, and $\nu_{3,s}$ is a monotonically decreasing function for increasing values of $s$, indicating that the smoother the function, the faster the error convergence rate. All quadrature error upper bounds demonstrate that FPSQ approximation for nonsmooth and $T$-periodic functions has ``a polynomial order accuracy.''

To sense how the smoothness of the function can influence the FPSQ error, consider the FPSQ approximations of the periodic extensions of the five polynomial functions $\displaystyle{f_s \in \C{H}_1^s}$, $s = 1,2,\ldots, 5$, defined by 
\begin{gather*}
{f_1}(x) = x (1-x),\quad {f_2}(x) = \frac{1}{3}{x^3} - \frac{1}{2}{x^2} + \frac{1}{6}x + 1,\quad {f_3}(x) = {x^4} - 2{x^3} + {x^2},\quad {f_4}(x) = - {x^5} + \frac{{15}}{6}{x^4} - \frac{5}{3}{x^3} + \frac{1}{6}x + 2,\\
\text{and }{f_5}(x) = - {x^6} + 3{x^5} - \frac{5}{2}{x^4} + \frac{1}{2}{x^2} - 1.
\end{gather*}
The periodic extension of each function $f_s$ exhibits a jump discontinuity in the $s$th derivative for $s = 1:5$. Therefore, Corollary \ref{cor:22} anticipates the decay rate of the FPSQ error of each function $f_s$ to be $O\left( {{N^{-s-1/2}}} \right)$ as $N \to \infty\,\forall s = 1:5$. Figure \ref{fig:FAccuracy2} shows the infinity- and Euclidean-error norms in log-lin scale of the FPSQ approximations of the five periodic functions. We can clearly observe that the calculated FPSQ errors remain below the estimated upper bounds in all cases with faster decay rates for smoother functions, which is consistent with Corollary \ref{cor:22}. The figure also shows the fitted curves of the FPSQ Euclidean-error norms in the power function model form $y = b x^m$. Taking the natural logarithm of both sides of the equation gives the equivalent log-log regression model $\ln y = m \ln x + \ln b$, which has the form of the linear regression model $Y = m X + B$ using the change of variables $X = \ln x$ and $Y = \ln y$ and the parameter substitution $B = \ln b$. The parameters $m$ and $b$ of the latter model were obtained using MATLAB polynomial curve fitting function ``polyfit'' with the observed data $\left\{ {({X_i},{Y_i})} \right\}_{i = 1}^{10} = \left\{\ln N_i,\ln \left\| {\C{I}_{{{\bm{x}_{N_i}}}}^{(x)}f - \Fthe f_{0:{N_i} - 1}} \right\|\right\}_{i=0}^{10}\,\forall\,N_i = 10 (10 + i), i = 1:10$. The estimated $m$ values were $-1.5, -2.5, -3.5, -4.5$, and $-5.57$ in close agreement with Corollary \ref{cor:22}.

\begin{figure}[ht]
\centering
\includegraphics[scale=0.34]{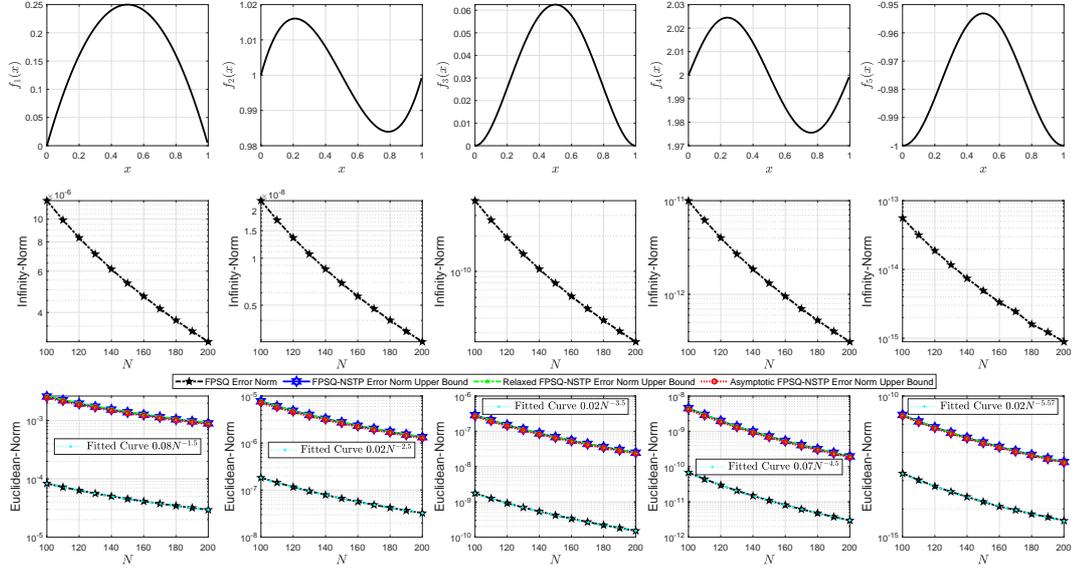}
\caption{The first row shows the five polynomial functions $\displaystyle{{f_1}(x) = x (1-x), {f_2}(x) = \frac{1}{3}{x^3} - \frac{1}{2}{x^2} + \frac{1}{6}x + 1, {f_3}(x) = {x^4} - 2{x^3} + {x^2}}$, $\displaystyle{{f_4}(x) = - {x^5} + \frac{{15}}{6}{x^4} - \frac{5}{3}{x^3} + \frac{1}{6}x + 2}$, and $\displaystyle{{f_5}(x) = - {x^6} + 3{x^5} - \frac{5}{2}{x^4} + \frac{1}{2}{x^2} - 1}$ on the interval $[0, 1]$. The infinity- and Euclidean-error norms in log-lin scale of the FPSQ approximations of each function are shown in the second- and -third rows of the same column, respectively, against $N = 100(10)200$. All FPSQ error norm upper bounds are shown in the last row of plots as well as the fitted curve of the FPSQ Euclidean-error norms obtained using MATLAB.}
\label{fig:FAccuracy2}
\end{figure}

To sense how the lack of continuity of the function can influence the FPSQ error, consider the FPSQ approximations of the the square wave function $f_6 \in \C{H}_{1}^{1}$ defined in one period by
\begin{empheq}[left={{f_6}(x) =} \empheqbiglbrace]{align*}
  1, &\quad 0 \le x < \frac{1}{2} \vee x = 1,\\
  0, &\quad \frac{1}{2} \le x < 1.
\end{empheq}
The function and its $N/2$-degree, $1$-periodic Fourier interpolant for $N = 10, 20, 40$, and $80$ are depicted in Figure \ref{fig:FAccuracy3}. We can observe spurious oscillations throughout most of the interval with large peaks near the jump discontinuities, giving rise to the well-known Gibbs phenomenon, which often occurs in Fourier series expansions and interpolations of discontinuous data. Adding more terms to the Fourier interpolant slowly decreases the error away from the discontinuities; however, the over- and undershoots remain visible to the naked eye and ultimately compress into a single vertical line at the points of discontinuity as $N \to \infty$. The largest amount of over- or undershoot in the Fourier interpolant $I_N(x)\,\forall\,N = 10(10)80$ when evaluated at $10,000$ equally spaced nodes between $0$ and $1$ are approximately $12.3\%, 14.3\%, 13.9\%, 14.2\%, 14.0\%, 14.1\%, 14.1\%$, and $14.1\%$ of the jump size, respectively. These over- and undershoot factors are slightly larger than the classical asymptotic overshoot factor in the Fourier expansion series, which approaches $8.95\%$ of the jump size. The contrast between the over- and undershoot factors associated with Fourier interpolation and that of the classical Gibbs phenomenon is not surprising because Fourier interpolation computations in finite-precision arithmetic are subject to aliasing and round-off errors. In fact, it was discovered more than a quarter century ago that the behavior of the Fourier interpolant near an isolated jump discontinuity point $\xi$ of a function depends on the position of $\xi$ with respect to the interpolation nodes considered \cite{HELMBERG199441}. In addition, Figure \ref{fig:FAccuracy3} shows the magnitude of the symmetric discrete Fourier interpolation coefficients against their indices, where the coefficients with even indices vanish. The fitted curves of the discrete data $\left\{\left(1:2:\frac{N}{2}-1,\left|\tilde f_{1:2:\frac{N}{2}-1}\right|\right)\right\}$ using the power function model $y = b x^m$ indicate that the coefficients decay like $O\left(k^{-1}\right)$ as indicated by Theorem \ref{thm:00}. The FPSQ error Euclidean-norm decays as $O\left(N^{-1/2}\right)$ as revealed by another curve fit obtained using MATLAB. This result agrees with the Asymptotic Formula \eqref{eq:FPSQEFATPF126Apr2021Hi1}, although the theoretical proof of Corollary \ref{cor:22} does not apply to this case.

\begin{figure}[ht]
\centering
\includegraphics[scale=0.35]{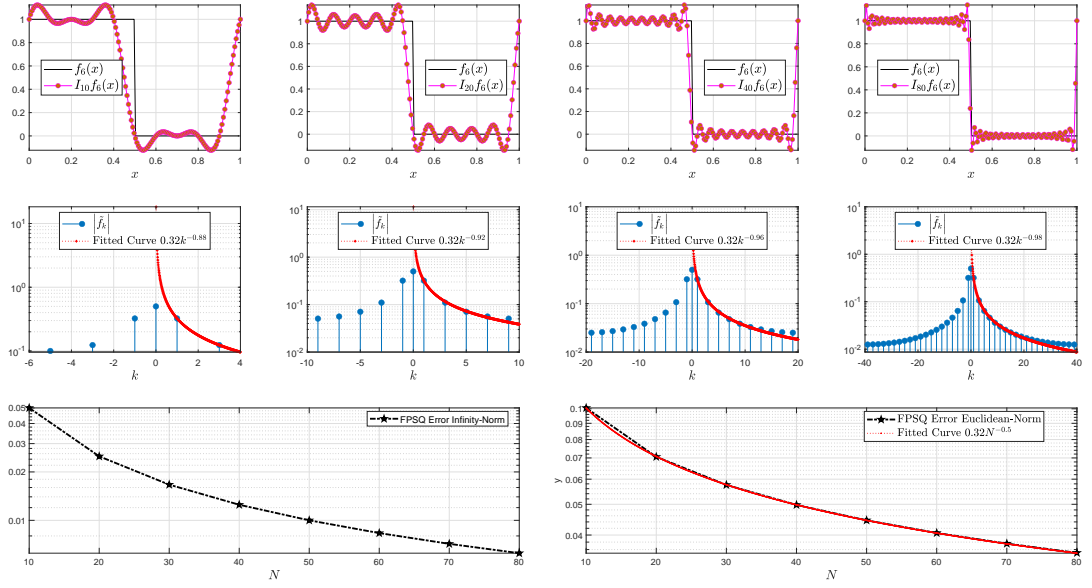}
\caption{The square wave function $f_6$ and its $N/2$-degree, $T$-periodic Fourier interpolant $I_N(x)\,\forall\,N = 10, 20, 40$, and $80$. The plots of $I_N(x)$ were
produced using 150 linearly spaced nodes between $0$ and $1$. The second row shows stem plots of the Fourier interpolation coefficients magnitudes against their indices in log-lin scale together with the fitted curves of the odd-indexed right half discrete data obtained through MATLAB. The last row displays the FPSQ error infinity- and -Euclidean norms in log-lin scale in addition to the fitted curve obtained from the latter experimental error data.}
\label{fig:FAccuracy3}
\end{figure}
The variance between the asymptotic overshoot factors of the Fourier interpolants and those of the Fourier expansion series is not the only distinctive difference in their behaviors. In fact, there is another remarkable difference that seemingly did not attract much attention in the literature: their behavior at a jump discontinuity. In the next section we shed some light on this subject.

\subsubsection{Behavior of Fourier Interpolants at Jump Discontinuities}
\label{sec:BOFIAAJD1}
Although it is well known that the classical Fourier expansion series converges to the average value of the left- and right-hand limits of the function at a jump discontinuity point $\xi$, to the best of our knowledge, the behavior of the Fourier interpolant at $\xi$ has not been investigated clearly in the literature. We confined our study to $T$-periodic piecewise constant functions with two jump discontinuity points in ${\F{\Omega}_T}$ and studied the behavior of Fourier interpolants of their periodic extensions on $\MBRzerP$ at the jump discontinuities. This study was motivated by the fact that the OC of Problem $\C{P}$ is a bang–bang controller that switches abruptly between two states at two switching times \cite{bayen2018optimal}. We refer to the jump discontinuity points in $(0, T)$ and at the boundary point $x = T$ as the interior and boundary jump discontinuity points, respectively. Now, consider the $1$-periodic, piecewise constant functions $f_j \in \C{H}_1^0\,\forall\,j = 6:12$, where $f_7, \ldots, f_{12}$ are given by
\[\begin{array}{l}
{f_7}(x) = \left\{ \begin{array}{l}
1,\quad 0 \le x < \frac{1}{3},\\
0,\quad \frac{1}{3} \le x < 1,
\end{array} \right.\;{f_8}(x) = \left\{ \begin{array}{l}
4,\quad 0 \le x < \frac{2}{3},\\
0,\quad \frac{2}{3} \le x < 1,
\end{array} \right.\;{f_{9}}(x) = \left\{ \begin{array}{l}
-5,\quad 0 \le x < 0.8183,\\
0,\quad 0.8183 \le x < 1,
\end{array} \right.\\\\
{f_{10}}(x) = \left\{ \begin{array}{l}
2,\quad 0 \le x < \frac{\pi }{5},\\
1,\quad \frac{\pi}{5} \le x < \frac{\pi }{4},\\
2,\quad \frac{\pi }{4} \le x < 1,
\end{array} \right.\;{f_{11}}(x) = \left\{ \begin{array}{l}
10,\quad 0 \le x < \frac{e}{5},\\
-2,\quad \frac{e}{5} \le x < \frac{e}{3},\\
10,\quad \frac{e}{3} \le x < 1.
\end{array} \right.\;\text{and }{f_{12}}(x) = \left\{ \begin{array}{l}
-3,\quad 0 \le x < \ln 1.5,\\
2,\quad \ln 1.5 \le x < \ln 2,\\
-3,\quad \ln 2 \le x < 1.
\end{array} \right.
\end{array}\]
The set of functions $\{f_{7:9}\}$ is created by various translations of the interior jump discontinuity point or different scaling of the square wave function $f_6$; thus, each function has exactly one interior jump discontinuity point and one boundary jump discontinuity point in ${\F{\Omega}_1}$. The remaining functions have both jump discontinuity points of the interior type. Table \ref{tab:OFICAAPOD1} lists the observed Fourier interpolant approximations at the interior jump discontinuity points of each function for several increasing values of $N$. For $f_6$, the interpolant converges to zero at $\xi = 1/2\,\forall N$ because $\xi$ always belongs to the set of interpolation nodes. However, each of the interpolant values of $f_7$ and $f_8$ appear to swing back and forth among three different limiting values, including the zero value when $\xi$ conforms to one of the interpolation nodes. Interestingly, the sum of each of the three interpolant values at the jump discontinuity point corresponding to any three consecutive $N$ values having a fixed step size of power of $10$ (nearly) equals the jump size at $\xi \foralle$ function! For the remaining functions, the interpolants did not show any signs of convergence to a single or even to a finite set of multiple values. In all cases where the jump discontinuity point $\xi$ does not coincide with any interpolation node, the value of the Fourier interpolant falls within the open interval $\left(f(\xi^-),f(\xi^+)\right)$, where $\xi^-$ and $\xi^+$ are points infinitesimally to the left and right of $\xi$, respectively. These results indicate that the Fourier interpolant diverges at $\xi$, except when the jump discontinuity point $\xi$ coincides with an interpolation node, where the Fourier interpolant matches the value of the discontinuous function according to the interpolation condition. 

\begin{table}[ht]
\caption{Observed Fourier interpolant values at a discontinuity point rounded to 4 decimal digits.} % title of Table
\centering % used for centering table
\resizebox{1\columnwidth}{!}{%
\begin{tabular}{*{11}{c}} % centered columns (7 columns)
\hline\hline %inserts double horizontal lines
$N$ & $I_Nf_6(1/2)$ & $I_Nf_7(1/3)$ & $I_Nf_8(2/3)$ & $I_Nf_9(0.8183)$ & $I_Nf_{10}(\pi/5)$ & $I_Nf_{10}(\pi/4)$ & $I_Nf_{11}(e/5)$ & $I_Nf_{11}(e/3)$ & $I_Nf_{12}(\ln 1.5)$ & $I_Nf_{12}(\ln 2)$\\ [0.5ex] % inserts table
%heading
\hline % inserts single horizontal line
100 & 0 & 0.6887 & 1.2253 & -0.7210 & 1.1430 & 1.5554 & 5.8852 & 5.5461 & -0.2099 & 0.5742\\
200 & 0 & 0.3101 & 2.7694 & -1.5989 & 1.3162 & 1.0614 & 0.8580 & 0.2428 & -2.6403 & -1.2379\\
300 & 0 & 0.0000 & 0.0000 & -2.5714 & 1.5021 & 1.6352 & 9.1071 & 8.3030 & 0.2979 & -2.7950\\
400 & 0 & 0.6905 & 1.2332 & -3.5372 & 1.6956 & 1.1289 & 4.5168 & 3.1217 & -2.2202 & 0.8602\\
500 & 0 & 0.3094 & 2.7663 & -4.3962 & 1.8709 & 1.7281 & -0.2953 & -1.5870 & 0.8087 & -0.9315\\
600 & 0 & 0.0000 & 0.0000 & -0.0705 & 1.0062 & 1.2070 & 8.0419 & 6.1515 & -1.7527 & -2.5647\\
700 & 0 & 0.6907 & 1.2343 & -0.7931 & 1.1475 & 1.8107 & 3.1736 & 0.8337 & 1.2771 & 1.1385\\
800 & 0 & 0.3092 & 2.7656 & -1.6890 & 1.3231 & 1.2936 & -1.3221 & 8.8158 & -1.2407 & -0.6059\\
900 & 0 & 0.0000 & 0.0000 & -2.6705 & 1.5145 & 1.8857 & 6.8532 & 3.7869 & 1.6919 & -2.3100\\
1000 & 0 & 0.6908 & 1.2347 & -3.6347 & 1.7066 & 1.3811 & 1.8420 & -1.1351 & -0.7053 & 1.4047\\
1100 & 0 & 0.3092 & 2.7652 & -4.4817 & 1.8786 & 1.9541 & 9.8124 & 6.7762 & -2.9592 & -0.2798\\
1200 & 0 & 0.0000 & 0.0000 & -0.1455 & 1.0125 & 1.4733 & 5.5581 & 1.4333 & -0.1601 & -2.0377\\
1300 & 0 & 0.6909 & 1.2349 & -0.8961 & 1.1549 & 1.0124 & 0.5913 & 9.2934 & -2.5944 & 1.6498\\
1400 & 0 & 0.3091 & 2.7650 & -1.8090 & 1.3326 & 1.5675 & 8.8764 & 4.4389 & 0.3703 & 0.0462\\
1500 & 0 & 0.0000 & 0.0000 & -2.7937 & 1.5265 & 1.0751 & 4.2180 & -0.6410 & -2.1657 & -1.7483\\
2000 & 0 & 0.3090 & 2.7648 & -1.9186 & 1.3424 & 1.8277 & 6.5636 & -0.1144 & 1.7384 & 0.6675\\
3000 & 0 & 0.0000 & 0.0000 & -0.3859 & 1.0324 & 1.1638 & -1.7335 & 1.0356 & -1.1057 & -0.1592\\
4000 & 0 & 0.6910 & 1.2354 & -4.1556 & 1.7555 & 1.6077 & 2.2617 & 2.2785 & 1.4404 & -1.0144\\
5000 & 0 & 0.3090 & 2.7645 & -2.5008 & 1.3923 & 1.9936 & 6.9655 & 3.5769 & -1.4988 & -1.8295\\
10000 & 0 & 0.6911 & 1.2356 & 0.0000 & 1.8453 & 1.9870 & 3.1116 & 9.4607 & 0.3689 & -0.3355\\ [1ex] % [1ex] adds vertical space
\hline %inserts single line
\end{tabular}
}
\label{tab:OFICAAPOD1} % is used to refer this table in the text
\end{table}

\subsection{Error and Convergence Analyses of Barycentric SG Quadratures}
\label{subsubsub:ECABSGQ1}
Let ${\left\|g\right\|_{{\infty},\F{\Gamma}_k}} = {\left\| {}_k g \right\|_{\infty}} = \sup {\left| {g\left(t^{(k)}\right)} \right|}\,\foralla g \in \MBF \cap \Def{\F{\Gamma}_k}, k \in \MBK_K$. The following two theorems underline the SG quadrature truncation error and its bounds on any partition $\F{\Gamma}_k$. Their proofs can be immediately derived from (\cite[Proofs of Theorems 4.1 and 4.2]{Elgindy20172})
by replacing the notations $m_k, \alpha_i^{(k),*}, {\hat z_{{m_k},i,j}^{(k),\alpha}}$ with $N_k, \alpha, \hat t_{N_k,j}^{(k),\alpha}$, respectively.

\begin{thm}\label{sec:erranalysgp1}
Let $N_k, M_k \in \mathbb{Z}_0^+$, and consider any arbitrary integration nodes set $\{y_{M_k, 0:M_k}\} \subset \F{\Gamma}_k\,\forall k \in \MBK_K$. Suppose also that ${}_k g \in C^{N_k + 1}(\F{\Gamma}_k)$ is approximated by Formula \eqref{eq:Sel1} with the associated discrete interpolation coefficients given by Formula \eqref{sec:ort:eq:sgt}, $\foralls g \in \MBF$. Then $\exists\,\{\zeta_{M_k,0:M_k}\} \subset  \Int{\F{\Gamma}_k}$ such that
\begin{equation}
\C{I}_{{\tau _{k - 1}}, {y_{M_k, i}}}^{\left(t^{(k)}\right)} g = {}_k\F{P}_i\,g_{0:N_k} + E_{N_k}^{\left( {\alpha} \right)}\left( {{y_{M_k, i}},\zeta_{M_k,i}} \right) ,
\end{equation}
where $g_{0:N_k}^t = g\left(\hat t_{N_k,0:N_k}^{(k),\alpha}\right)$, 
\begin{equation}
	E_{N_k}^{\left( {\alpha} \right)}\left( {{y_{M_k, i}},\zeta_{M_k,i}} \right) = \frac{{{g^{({N_k} + 1)}}\left( {\zeta_{M_k,i}} \right)}}{{({N_k} + 1)!{\mkern 1mu} K_{k,{N_k} + 1}^{\left( {\alpha} \right)}}} \C{I}_{\tau _{k - 1}, {y_{M_k, i}}}^{\left(t^{(k)}\right)} {\hat G_{k, {N_k} + 1}^{\left( {\alpha} \right)}},
\end{equation}
is the truncation error of the SG quadrature $\foralle i, k$,
\begin{equation}
K_{k,j}^{(\alpha )} = \frac{{{2^{2j - 1}}}}{{{{|\F{\Gamma}_k|}^j}}}\frac{{\Gamma \left( {2\alpha  + 1} \right)\Gamma \left( {j + \alpha } \right)}}{{\Gamma \left( {\alpha  + 1} \right)\Gamma \left( {j + 2\alpha } \right)}}\quad \forall j \in \mathbb{Z}_0^+,
\end{equation}
is the leading coefficient of the $(j,k)$-SG polynomial, and $\Gamma$ is the usual Gamma function.
\end{thm}

\begin{thm}\label{thm:Jan212022}
Let ${\left\| {{{}_k g^{({N_k} + 1)}}} \right\|_{\infty}} = A_k \in \MBRzerP\, \forall k \in \MBK_K$, where the constant $A_k$ depends on $k$ but is independent of $N_k$. Suppose also that the assumptions of Theorem \ref{sec:erranalysgp1} hold true. Then there exist some constants ${D^{\left( \alpha \right)}} > 0, B_{1,k}^{\left( \alpha \right)} = {A_k}{D^{\left( \alpha \right)}}$, and $B_2^{\left( \alpha \right)} > 1$, which depend on $\alpha$ but are independent of $N_k$, such that the SG quadrature truncation error, $E_{{N_k}}^{\left( \alpha \right)}\left( {y_{M_k, i},\zeta_{M_k,i}} \right)$, is bounded by
\begin{equation}
	\begin{array}{l}
	{\left\| {E_{{N_k}}^{\left( \alpha \right)}\left( {y_{M_k, i},\zeta_{M_k,i}} \right)} \right\|_{{{\infty}},\F{\Gamma}_k}} =  B_{1,k}^{\left( \alpha \right)}\,{2^{ - 2{N_k} - 1}}{{{e}}^{{N_k}}}{N_k}^{\alpha - {N_k} - \frac{3}{2}}\left( {y_{M_k, i} - {\tau _{k - 1}}} \right){{|\F{\Gamma}_k|} ^{{N_k} + 1}} \times \\
	\left( {\left\{ \begin{array}{l}
	1,\quad {N_k} \ge 0 \wedge \alpha \ge 0,\\
	\displaystyle{\frac{{\Gamma \left( {\frac{{{N_k}}}{2} + 1} \right)\Gamma \left( {\alpha + \frac{1}{2}} \right)}}{{\sqrt \pi\,\Gamma \left( {\frac{{{N_k}}}{2} + \alpha + 1} \right)}}},\quad N_k \in \MBZOP \wedge  - \frac{1}{2} < \alpha < 0,\\
	\displaystyle{\frac{{2\Gamma \left( {\frac{{{N_k} + 3}}{2}} \right)\Gamma \left( {\alpha + \frac{1}{2}} \right)}}{{\sqrt \pi  \sqrt {\left( {{N_k} + 1} \right)\left( {{N_k} + 2\alpha + 1} \right)}\,\Gamma \left( {\frac{{{N_k} + 1}}{2} + \alpha} \right)}}},\quad N_k \in \MBZzereP \wedge  - \frac{1}{2} < \alpha < 0,\\
	B_2^{\left( \alpha \right)} {\left( {{N_k} + 1} \right)^{ - \alpha}},\quad {N_k} \to \infty  \wedge  - \frac{1}{2} < \alpha < 0
	\end{array} \right.} \right),\quad \forall i \in \MBJ_{M_k}^+, k \in \MBK_K.
	\end{array}
\end{equation}
\end{thm}
Theorem \ref{thm:Jan212022} shows that the SG quadrature formula converges exponentially fast for piecewise smooth functions whose pieces are defined on $\F{\Gamma}_{1:K}$. Because $\psi$ is a piecewise smooth function on $\F{\Omega}_T$, the truncation errors in approximating the definite integrals over the intervals $\F{\Omega}_{x_{N,0:N-1}}$ using SG quadratures decay with an exponential convergence rate, and the total quadrature error is dominated by the errors committed in constructing $\breve u_{N,M}$ and $\tilde s \foralle N \in \MBJ_N$. On the other hand, regardless of how well the estimates of $\breve u_{N,M}$ and $\tilde s$, the FPSQ error in approximating $\C{I}_{x_{N,n}^{(t)}} \psi \foralle n \in \MBJ_N$, in the best scenario, is $O\left(N^{-1/2}\right)$, as $N \to \infty$, not to mention the size gap between the SGIM employed in the construction of the SG quadrature and the FIM required to achieve the same degree of accuracy-- clearly, the SGIM wins this race hands down. 

\begin{rem}
The error and the convergence analysis of GG quadratures for sufficiently smooth functions on the interval $[-1, 1]$ are outlined by (\cite[Theorems 3.1 \& 3.2]{Elgindy20171}). The error of the SGG quadrature for sufficiently smooth functions on the interval $\F{\Omega}_T$ is expounded in (\cite[Theorem 4.1]{Elgindy20161}).   
\end{rem}

\section{Reconstruction of Piecewise Analytic Functions with High Accuracy up to the Points of Jump Discontinuities}
\label{sec:DTD}
Equation \eqref{eq:FTS1AE1hi126Apr2021} shows that a $T$-periodic Fourier interpolant converges to a piecewise analytic function with jump discontinuities at rate $O(N^{-1/2})$ as $N$ grows large. To recover the piecewise analytic function with high accuracy from PS data, we present a novel, simple, and efficient edge-detection technique in the next section. We demonstrate later in Section \ref{sec:FPCFPI1} how to tune this new technique. 

\subsection{Detecting the Jump Discontinuities and Reconstructing the Piecewise Analytic Function}
\label{sec:DTD2}
Detecting the jump discontinuity points in the OC of Problem $\C{P}$ is a crucial step and an essential prerequisite to obtain accurate approximations to the solutions of the problem. In fact, the discontinuity points of the controller are not known a priori; therefore, we need to estimate their locations from the Fourier interpolant approximation. To this end, we present a novel accurate and efficient edge detection method that can estimate the locations of the desired discontinuities for piecewise constant functions with two jump discontinuities in ${\F{\Omega}_T}$. The proposed method is inspired by the fact that the horizontal extension of the Gibbs phenomenon is reduced as the number of spectral terms grows, but the ultimate graph of Fourier interpolant in the close vicinity of a jump discontinuity point $\xi$ turns into a jagged line on both sides of $\xi$ and passes in almost a vertical direction through a point whose abscissa is $\xi$ and ordinate falls within the open interval $\left(f(\xi^-),f(\xi^+)\right)$, except when $\xi$ exist at an interpolation point where Fourier interpolant matches the function value at $\xi$. Therefore, the location of $\xi$ is gradually squeezed between the locations of the sharp spikes and eventually falls (almost) at the midpoint between the two abscissas whose ordinates are the peak and the bottom out of the two jagged lines enclosing $\xi$, as $N \to \infty$. Hence, the peculiar manner in which Fourier interpolant behaves near a jump discontinuity point provides an excellent means of detecting one. We present a brief description of this method in what follows. 

\textbf{Implementation of the Edge Detection Strategy.} Given a piecewise constant function $f \in \C{H}_T^{0}: f(0) = f(T)$ with two jump discontinuity points in ${\F{\Omega}_T}$, start by constructing its Fourier interpolant $I_Nf$ and determine its extreme values on ${\F{\Omega}_T}$. To this end, we evaluate $I_Nf$ at a set of equally-spaced nodes $\MBS_M = \{y_{M,0:M-1}\}\,\foralls$ relatively large $M \in \MBZ^+$. Next, we find 
$d_{\max}:= \indmax I_Nf(\bm{y}_M)$ and $d_{\min}:= \indmin I_Nf(\bm{y}_M)$. To refine the obtained approximations to the extreme values, we extremize $I_Nf$ on the relatively small uncertainty intervals $[y_{M,d_{\max}-1}, y_{M,d_{\max}+1}]$ and $[y_{M,d_{\min}-1}, y_{M,d_{\min}+1}]$, respectively. For this task, we prefer to apply the recent fast line search method known by the Chebyshev PS line search method (CPSLSM) \cite{Elgindy2018optimization}. Let $I_N^{\max}f$ and $I_N^{\min}f$ be the maximum and minimum values of $I_Nf$ obtained by the CPSLSM, respectively. The next step is to set up the straight line $y = I_N^{\text{ave}}f = \frac{1}{2} \left(I_N^{\max}f + I_N^{\min}f\right)$ whose ordinate is the average value of the calculated extreme values; we call this line the ``separation line,'' as we shall use it later to separate the Fourier interpolant values into two groups of discrete data. For now, we choose a user-defined tolerance $\epsilon$, and form the ``user-defined discontinuity feasible zone''--- basically a narrow strip centered about the separation line with radius $\epsilon$, where any Fourier interpolant value within an $\epsilon$-distance from the separation line is recognized as a possible Fourier interpolant value paired with an estimated discontinuity point $\tilde \xi$ that is sufficiently close to a true discontinuity point $\xi$. Our rationale here is simple: ``while the Fourier interpolant at a discontinuity point $\xi$ is not necessarily equal to the average of the left and right limits at $\xi$, as verified by Table \ref{tab:OFICAAPOD1}, it is less likely that the Fourier interpolant at a continuity point to be exactly equal to the average of the left and right limits at $\xi$.'' We expect to have at most a single approximate discontinuity point near each discontinuity point, since the graph of Fourier interpolant moves in virtually a vertical direction through each point whose abscissa is a discontinuity point, so that the Fourier interpolant values at the interpolation points that are not sufficiently close to a discontinuity point are expected to live outside the user-defined discontinuity feasible zone, for relatively small $\epsilon$. For a certain tolerance $\epsilon$, we refer to this zone by ``the $\epsilon$-discontinuity feasible zone,'' and denote it by $\C{Z}_{\epsilon}^{\text{disc}}$. We prefer to set the $\epsilon$ value to within a relatively small length of ``the local extremeshoot height\footnote{The extremeshoot height refers to the vertical distance between the maximum overshoot and the minimum undershoot of the signal near a discontinuity.}'' such that $\epsilon = \tilde \epsilon\;(I_N^{\max}f - I_N^{\min}f)\; \foralls \tilde \epsilon \in (0, 0.01]$. Now, let $\Xi = \left\{\tilde \xi_{1:L}\right\}$ be the set of approximate discontinuity points collected at this step $\foralls L \in \{1, 2\}$. If $L = 2$, i.e., the method determines two approximate discontinuity points, then we consider the method successful and terminate the procedure at this step. Otherwise, we create the ``discontinuity auxiliary function $I_N^{\text{aux}}f$,'' which is a two-state, piecewise constant function whose two states are the obtained extreme values of $I_Nf$ such that
\begin{equation}\label{eq:INaux1}
I_N^{{\text{aux}}}f({y_{M,l}}) = \left\{ \begin{array}{l}
I_N^{\min }f,\quad {I_N}f({y_{M,l}}) < I_N^{{\text{ave}}}f,\\
I_N^{\max }f,\quad {I_N}f({y_{M,l}}) > I_N^{{\text{ave}}}f,
\end{array} \right.\quad \forall l \in \MBJ_{M}^+.
\end{equation}
Since the function $f$ switches its state abruptly at a discontinuity point $\xi$, we expect $\xi$ to occur sufficiently close to the two consecutive interpolation points whose ordinates lie closely above and below the separation line. To find these points, we find the two-elements index vector $J_{1:2}: I_N^{{\text{aux}}}f(y_{M,J_l}) - I_N^{{\text{aux}}}f(y_{M,J_l+1}) \ne 0\, \forall l = 1, 2$. Now, the desired discontinuity points $\xi_1$ and $\xi_2$ either exist in the closed intervals $[y_{M,J_1}, y_{M,J_1+1}]$ and $[y_{M,J_2}, y_{M,J_2+1}]$ or occur in close proximity of their boundaries. To account for both scenarios, denote the boundary points $y_{M,J_1}, y_{M,J_1+1}, y_{M,J_2}$, and $y_{M,J_2+1}$ by $b_j\,\forall j = 1:4$, respectively. If the set $\Xi$ is empty, we estimate $\xi_1$ and $\xi_2$ by the midpoints of the two intervals such that $\tilde \xi_l = \frac{1}{2} (b_{2 l}-b_{2 l-1})\,\forall l=1,2$. On the other hand, if the set $\Xi$ contains already one of the two discontinuity points, say $\xi_1$, then we drop the interval which either includes $\xi_1$ or whose boundaries are closely adjacent to $\xi_1$. We can then estimate $\xi_2$ by bisecting the remaining interval. That is, we can calculate $\tilde \xi_2$ by the formula
\[{\tilde \xi _2} = \left\{ \begin{array}{l}
\frac{1}{2}({b_2} - {b_1}),\quad {\text{if }}{b_1} - {{\tilde \xi }_1} > \varepsilon  \vee {{\tilde \xi }_1} - {b_2} > \varepsilon,\\
\frac{1}{2}({b_4} - {b_3}),\quad {\text{otherwise,}}
\end{array} \right.\]
%T,\quad \text{if }M\text{ is relatively large}\text{ and }J_2 = M,\\
$\foralls$ relatively small positive number $\varepsilon < T/M$. As a further correction step in practice, an estimated discontinuity point is assigned the value $T$ if its location is within a sufficiently small distance from $x = T$. For a relatively large value of $M$, we may reasonably set the estimated discontinuity point equal to $T$ when $J_2 = M - 1$. A graphical illustration of the method is depicted in Figure \ref{fig:EdDet1}, for $M = 200, \epsilon = 0.5\% (I_N^{\max}f-I_N^{\min}f)$, and $\varepsilon = T/(2 M)$. Table \ref{tab:OFICAAPOD2} shows the observed relative errors in the estimated jump discontinuity points of the functions $f_6, \ldots, f_{12}$ using $M \in \{100,400\}$ and the same $\epsilon$- and $\varepsilon$ values. We notice from the table that, for a certain $N$ value, the relative errors often drop-off when $M$ increases; roughly speaking, the estimated interior discontinuity points approximate the true ones to two-three and three-five significant digits for $M = 100$ and $400$, respectively. On the other hand, for a certain $M$ value, the relative errors may slightly decrease at the beginning for increasing $N$ values, but (almost) cease to fall beyond a certain level as $N$ grows larger, in general, except for $\xi_1 = 0.5$ of $f_6$, where the error drops to zero abruptly at $(N,M) = (200,100)$ and $(500,400)$ and sustains at this level for growing values of $N$. In all cases, the method locates the exact boundary jump discontinuity point perfectly whenever exists.

\begin{figure}[ht]
\hspace{-2cm}\includegraphics[scale=0.4]{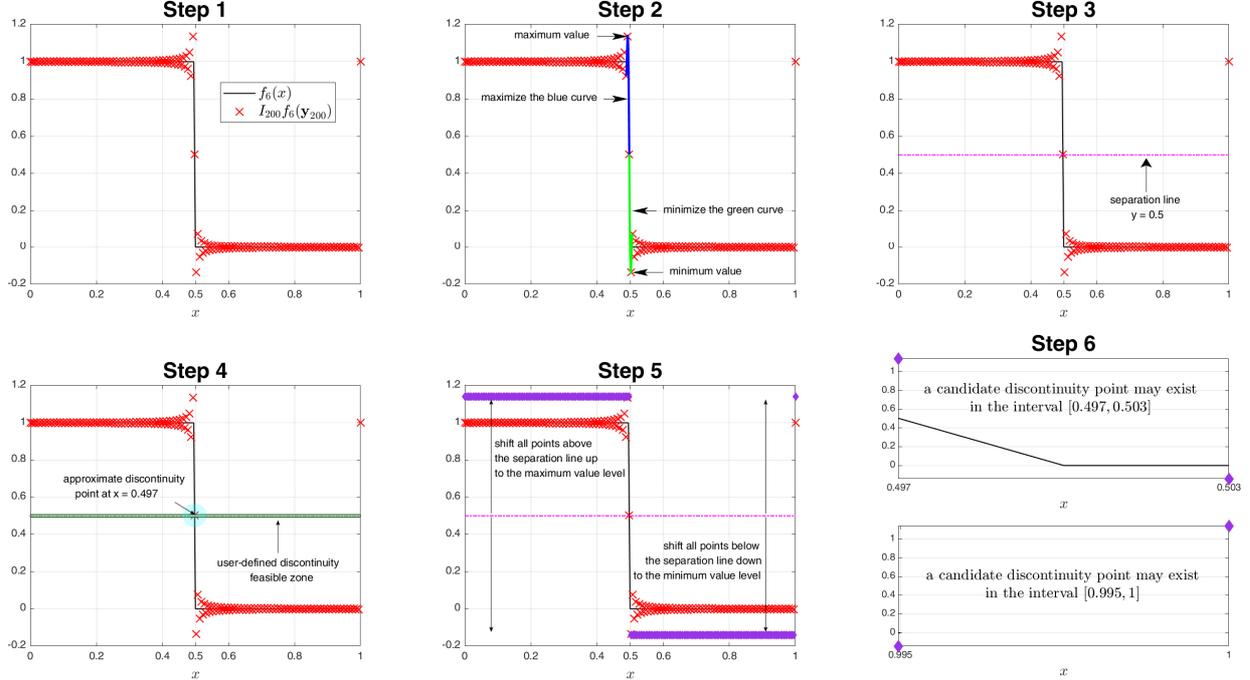}
\caption{Necessary steps to locate the discontinuity points of the wave function $f_6$ through the proposed edge detection method. In the first step the method calculates the Fourier interpolant $I_{200}f_6$ at the equally-spaced nodes $\bm{y}_{200} = \bm{x}_{200}$ and sets $y_{200, 200} = y_{200, 0}$. In Step 2, the indices of the extreme values of $I_{200}f_6(\bm{y}_{200})$ are determined ($99$ and $101$ in this case); these extreme values specify the parts of the graph of $I_{200}f_6$ to be extremized using the CPSLSM (the blue and green curves in the plot whose domains are the narrow intervals $[0.487, 0.497]$ and $[0.497, 0.508]$, respectively). In Step 3, the separation line $y = I_N^{\text{ave}}f_6 = 0.5$ is installed by averaging the determined extreme values from the previous step (the extreme values were found to be $1.14116$ and $-0.14116$). Interpolation points are deemed approximate discontinuity points in Step 4 if their ordinates fall within the boundaries of roughly $\C{Z}_{0.0064}^{\text{disc}}$; the method stops at this step if two estimates are found (only one estimate $\tilde \xi_1 \approx 0.497$ was recorded here with ordinate $0.503$). Otherwise, the discontinuity auxiliary function $I_N^{\text{aux}}f_6$ is derived in Step 5 by shifting all points above and below the separation line to the extreme values levels $y = 1.14116$ and $y = -0.14116$. The range analysis of $I_N^{\text{aux}}f_6$ gives birth to the two intervals $[0.497, 0.503]$ and $[0.995, 1]$ in step 6, which may contain the discontinuity points or their boundaries are sufficiently close to them. The midpoints of these two intervals are considered estimates to the the discontinuity points, except when $J_2 = M-1$, or when an approximate discontinuity point that was discovered in Step 4 either exist in one interval, or occurs in the vicinity of one of its boundaries. In this example, the midpoints of both intervals are abandoned, since $\tilde \xi_1 = 0.497 \in [0.497, 0.503]$ and $J_2 = 199$, so we set $\tilde \xi_2 = 1$.}
%All calculations were rounded to $3$ decimal digits.
\label{fig:EdDet1}
\end{figure}

\begin{table}[ht]
\caption{Observed relative errors in the estimated discontinuity points using $M = 100$ and $M = 400$. All approximations are rounded to six significant digits.} % title of Table
\centering % used for centering table
\begin{threeparttable}
\resizebox{1\columnwidth}{!}{%
\begin{tabular}{*{16}{c}} % centered columns (15 columns)
\toprule
 & & \multicolumn{2}{c}{$f_6$} & \multicolumn{2}{c}{$f_7$} & \multicolumn{2}{c}{$f_8$} & \multicolumn{2}{c}{$f_9$} & \multicolumn{2}{c}{$f_{10}$} & \multicolumn{2}{c}{$f_{11}$} & \multicolumn{2}{c}{$f_{12}$}\\
 & & $\xi_1 = 0.5$ & $\xi_2 = 1$ & $\xi_1 = 1/3$ & $\xi_2 = 1$ & $\xi_1 = 2/3$ & $\xi_2 = 1$ & $\xi_1 = 0.8183$ & $\xi_2 = 1$ & $\xi_1 = \pi/5$ & $\xi_2 = \pi/4$ & $\xi_1 = e/5$ & $\xi_2 = e/3$ & $\xi_1 = \ln 1.5$ & $\xi_2 = \ln 2$\\
\cmidrule{3-16} %inserts double horizontal lines
 & $N$ & $\left|\xi_1-\tilde \xi_1\right|_r$ & $\left|\xi_2-\tilde \xi_2\right|_r$& $\left|\xi_1-\tilde \xi_1\right|_r$ & $\left|\xi_2-\tilde \xi_2\right|_r$& $\left|\xi_1-\tilde \xi_1\right|_r$ & $\left|\xi_2-\tilde \xi_2\right|_r$& $\left|\xi_1-\tilde \xi_1\right|_r$ & $\left|\xi_2-\tilde \xi_2\right|_r$& $\left|\xi_1-\tilde \xi_1\right|_r$ & $\left|\xi_2-\tilde \xi_2\right|_r$& $\left|\xi_1-\tilde \xi_1\right|_r$ & $\left|\xi_2-\tilde \xi_2\right|_r$& $\left|\xi_1-\tilde \xi_1\right|_r$ & $\left|\xi_2-\tilde \xi_2\right|_r$\\ [0.5ex] % inserts table
\midrule % inserts single horizontal line
\multirow{20}{*}{\rotatebox{90}{$M = 100$}} & 100 & 1.01010E-02 & 0 & 1.51515E-02 & 0 & 7.57576E-03 & 0 & 6.31637E-03 & 0 & 1.13102E-02 & 3.27207E-03 & 5.98232E-03 & 2.26636E-03 & 8.94233E-03 & 1.77161E-03\\
& 200 & 0 & 0 & 1.51515E-02 & 0 & 7.57576E-03 & 0 & 6.31637E-03 & 0 & 4.76605E-03 & 3.27207E-03 & 5.98232E-03 & 2.26636E-03 & 8.94233E-03 & 1.77161E-03\\
& 300 & 0 & 0 & 1.51515E-02 & 0 & 7.57576E-03 & 0 & 6.02752E-03 & 0 & 4.76605E-03 & 3.27207E-03 & 5.98232E-03 & 2.26636E-03 & 8.94233E-03 & 1.77161E-03\\
& 400 & 0 & 0 & 1.51515E-02 & 0 & 7.57576E-03 & 0 & 6.02752E-03 & 0 & 4.76605E-03 & 3.27207E-03 & 5.98232E-03 & 2.26636E-03 & 8.94233E-03 & 1.77161E-03\\
& 500 & 0 & 0 & 1.51515E-02 & 0 & 7.57576E-03 & 0 & 6.02752E-03 & 0 & 4.76605E-03 & 3.27207E-03 & 5.98232E-03 & 2.26636E-03 & 8.94233E-03 & 1.77161E-03\\
& 600 & 0 & 0 & 1.51515E-02 & 0 & 7.57576E-03 & 0 & 6.31637E-03 & 0 & 4.76605E-03 & 3.27207E-03 & 5.98232E-03 & 2.26636E-03 & 8.94233E-03 & 1.77161E-03\\
& 700 & 0 & 0 & 1.51515E-02 & 0 & 7.57576E-03 & 0 & 6.31637E-03 & 0 & 4.76605E-03 & 3.27207E-03 & 5.98232E-03 & 2.26636E-03 & 8.94233E-03 & 1.77161E-03\\
& 800 & 0 & 0 & 1.51515E-02 & 0 & 7.57576E-03 & 0 & 6.31637E-03 & 0 & 4.76605E-03 & 3.27207E-03 & 5.98232E-03 & 2.26636E-03 & 8.94233E-03 & 1.77161E-03\\
& 900 & 0 & 0 & 1.51515E-02 & 0 & 7.57576E-03 & 0 & 6.02752E-03 & 0 & 4.76605E-03 & 3.27207E-03 & 5.98232E-03 & 2.26636E-03 & 8.94233E-03 & 1.77161E-03\\
& 1000 & 0 & 0 & 1.51515E-02 & 0 & 7.57576E-03 & 0 & 6.02752E-03 & 0 & 4.76605E-03 & 3.27207E-03 & 5.98232E-03 & 2.26636E-03 & 8.94233E-03 & 1.77161E-03\\
& 1100 & 0 & 0 & 1.51515E-02 & 0 & 7.57576E-03 & 0 & 6.02752E-03 & 0 & 4.76605E-03 & 3.27207E-03 & 5.98232E-03 & 2.26636E-03 & 8.94233E-03 & 1.77161E-03\\
& 1200 & 0 & 0 & 1.51515E-02 & 0 & 7.57576E-03 & 0 & 6.31637E-03 & 0 & 4.76605E-03 & 3.27207E-03 & 5.98232E-03 & 2.26636E-03 & 8.94233E-03 & 1.77161E-03\\
& 1300 & 0 & 0 & 1.51515E-02 & 0 & 7.57576E-03 & 0 & 6.31637E-03 & 0 & 4.76605E-03 & 3.27207E-03 & 5.98232E-03 & 2.26636E-03 & 8.94233E-03 & 1.77161E-03\\
& 1400 & 0 & 0 & 1.51515E-02 & 0 & 7.57576E-03 & 0 & 6.02752E-03 & 0 & 4.76605E-03 & 3.27207E-03 & 5.98232E-03 & 2.26636E-03 & 8.94233E-03 & 1.77161E-03\\
& 1500 & 0 & 0 & 1.51515E-02 & 0 & 7.57576E-03 & 0 & 6.02752E-03 & 0 & 4.76605E-03 & 3.27207E-03 & 5.98232E-03 & 2.26636E-03 & 8.94233E-03 & 1.77161E-03\\
& 2000 & 0 & 0 & 1.51515E-02 & 0 & 7.57576E-03 & 0 & 6.02752E-03 & 0 & 4.76605E-03 & 3.27207E-03 & 5.98232E-03 & 2.26636E-03 & 8.94233E-03 & 1.77161E-03\\
& 3000 & 0 & 0 & 1.51515E-02 & 0 & 7.57576E-03 & 0 & 6.31637E-03 & 0 & 4.76605E-03 & 3.27207E-03 & 5.98232E-03 & 2.26636E-03 & 8.94233E-03 & 1.77161E-03\\
& 4000 & 0 & 0 & 1.51515E-02 & 0 & 7.57576E-03 & 0 & 6.02752E-03 & 0 & 4.76605E-03 & 3.27207E-03 & 5.98232E-03 & 2.26636E-03 & 8.94233E-03 & 1.77161E-03\\
& 5000 & 0 & 0 & 1.51515E-02 & 0 & 7.57576E-03 & 0 & 6.02752E-03 & 0 & 4.76605E-03 & 3.27207E-03 & 5.98232E-03 & 2.26636E-03 & 8.94233E-03 & 1.77161E-03\\
& 10000 & 0 & 0 & 1.51515E-02 & 0 & 7.57576E-03 & 0 & 6.02752E-03 & 0 & 4.76605E-03 & 3.27207E-03 & 5.98232E-03 & 2.26636E-03 & 8.94233E-03 & 1.77161E-03\\
\midrule
\multirow{20}{*}{\rotatebox{90}{$M = 400$}} & 100 & 1.00251E-02 & 0 & 3.75940E-03 & 0 & 1.87970E-03 & 0 & 3.06798E-03 & 0 & 4.78300E-03 & 4.02499E-04 & 2.67893E-03 & 8.70828E-05 & 1.73431E-03 & 3.37813E-03\\
& 200 & 5.01253E-03 & 0 & 3.75940E-03 & 0 & 1.87970E-03 & 0 & 5.20671E-06 & 0 & 7.94154E-04 & 3.59358E-03 & 1.93109E-03 & 2.67893E-03 & 4.44690E-03 & 2.37647E-04\\
& 300 & 5.01253E-03 & 0 & 3.75940E-03 & 0 & 1.87970E-03 & 0 & 5.20671E-06 & 0 & 7.94154E-04 & 4.02499E-04 & 2.67893E-03 & 8.70828E-05 & 1.73431E-03 & 3.85342E-03\\
& 400 & 2.50627E-03 & 0 & 3.75940E-03 & 0 & 1.87970E-03 & 0 & 5.20671E-06 & 0 & 7.94154E-04 & 4.02499E-04 & 1.93109E-03 & 8.70828E-05 & 4.44690E-03 & 2.37647E-04\\
& 500 & 0 & 0 & 3.75940E-03 & 0 & 1.87970E-03 & 0 & 5.20671E-06 & 0 & 7.94154E-04 & 4.02499E-04 & 1.93109E-03 & 8.70828E-05 & 1.73431E-03 & 2.37647E-04\\
& 600 & 0 & 0 & 3.75940E-03 & 0 & 1.87970E-03 & 0 & 5.20671E-06 & 0 & 7.94154E-04 & 4.02499E-04 & 2.67893E-03 & 8.70828E-05 & 1.73431E-03 & 2.37647E-04\\
& 700 & 0 & 0 & 3.75940E-03 & 0 & 1.87970E-03 & 0 & 5.20671E-06 & 0 & 7.94154E-04 & 4.02499E-04 & 1.93109E-03 & 8.70828E-05 & 1.73431E-03 & 2.37647E-04\\
& 800 & 0 & 0 & 3.75940E-03 & 0 & 1.87970E-03 & 0 & 5.20671E-06 & 0 & 7.94154E-04 & 4.02499E-04 & 1.93109E-03 & 8.70828E-05 & 1.73431E-03 & 2.37647E-04\\
& 900 & 0 & 0 & 3.75940E-03 & 0 & 1.87970E-03 & 0 & 5.20671E-06 & 0 & 7.94154E-04 & 4.02499E-04 & 2.67893E-03 & 8.70828E-05 & 1.73431E-03 & 2.37647E-04\\
& 1000 & 0 & 0 & 3.75940E-03 & 0 & 1.87970E-03 & 0 & 5.20671E-06 & 0 & 7.94154E-04 & 4.02499E-04 & 1.93109E-03 & 8.70828E-05 & 1.73431E-03 & 2.37647E-04\\
& 1100 & 0 & 0 & 3.75940E-03 & 0 & 1.87970E-03 & 0 & 5.20671E-06 & 0 & 7.94154E-04 & 4.02499E-04 & 2.67893E-03 & 8.70828E-05 & 1.73431E-03 & 2.37647E-04\\
& 1200 & 0 & 0 & 3.75940E-03 & 0 & 1.87970E-03 & 0 & 5.20671E-06 & 0 & 7.94154E-04 & 4.02499E-04 & 1.93109E-03 & 8.70828E-05 & 1.73431E-03 & 2.37647E-04\\
& 1300 & 0 & 0 & 3.75940E-03 & 0 & 1.87970E-03 & 0 & 5.20671E-06 & 0 & 7.94154E-04 & 4.02499E-04 & 1.93109E-03 & 8.70828E-05 & 1.73431E-03 & 2.37647E-04\\
& 1400 & 0 & 0 & 3.75940E-03 & 0 & 1.87970E-03 & 0 & 5.20671E-06 & 0 & 7.94154E-04 & 4.02499E-04 & 2.67893E-03 & 8.70828E-05 & 1.73431E-03 & 2.37647E-04\\
& 1500 & 0 & 0 & 3.75940E-03 & 0 & 1.87970E-03 & 0 & 5.20671E-06 & 0 & 7.94154E-04 & 4.02499E-04 & 1.93109E-03 & 8.70828E-05 & 1.73431E-03 & 2.37647E-04\\
& 2000 & 0 & 0 & 3.75940E-03 & 0 & 1.87970E-03 & 0 & 5.20671E-06 & 0 & 7.94154E-04 & 4.02499E-04 & 1.93109E-03 & 8.70828E-05 & 1.73431E-03 & 2.37647E-04\\
& 3000 & 0 & 0 & 3.75940E-03 & 0 & 1.87970E-03 & 0 & 5.20671E-06 & 0 & 7.94154E-04 & 4.02499E-04 & 1.93109E-03 & 8.70828E-05 & 1.73431E-03 & 2.37647E-04\\
& 4000 & 0 & 0 & 3.75940E-03 & 0 & 1.87970E-03 & 0 & 5.20671E-06 & 0 & 7.94154E-04 & 4.02499E-04 & 1.93109E-03 & 8.70828E-05 & 1.73431E-03 & 2.37647E-04\\
& 5000 & 0 & 0 & 3.75940E-03 & 0 & 1.87970E-03 & 0 & 5.20671E-06 & 0 & 7.94154E-04 & 4.02499E-04 & 1.93109E-03 & 8.70828E-05 & 1.73431E-03 & 2.37647E-04\\
& 10000 & 0 & 0 & 3.75940E-03 & 0 & 1.87970E-03 & 0 & 5.20671E-06 & 0 & 7.94154E-04 & 4.02499E-04 & 1.93109E-03 & 8.70828E-05 & 1.73431E-03 & 2.37647E-04\\
\bottomrule %inserts single line
\end{tabular}
}\\
\begin{tablenotes}\footnotesize
\item[*] $|\cdot|_r$ gives the relative error, and the letter E stands for power of $10$.
\end{tablenotes}
\end{threeparttable}
\label{tab:OFICAAPOD2} % is used to refer this table in the text
\end{table}

Now that we know how to detect discontinuity points with plausible accuracy, it remains a final question before we end this section: how can we quickly reconstruct the piecewise analytic function from the PS data? A possible answer to this question is simple and does not require much effort after recovering the discontinuity points with satisfactory precision. The key to the answer lies again in the separation line $y = I_N^{\text{ave}}f$. In particular, the Fourier interpolant values are divided into two groups of discrete data by this line: one group contains all Fourier interpolant values above the separation line and occurs almost entirely close to the upper state level of the discontinuous function, say $\MB{FI}^{\text{u}}$, and another group contains all Fourier interpolant values below the separation line and mainly appears close to the lower state level of the discontinuous function, say $\MB{FI}^{\text{d}}$, except near the discontinuity points $\xi_1$ and $\xi_2$, where the Fourier interpolant values in each group are significantly larger or smaller than the other values in the same set. This motivated us to measure the central tendency of each data group using the median, which is less likely to be distorted by outliers than the mean. To this end, let ${\breve f^{\max }_{N,M}}$ and ${\breve f^{\min }_{N,M}}$ be the medians of the two sets of Fourier interpolant values, $\MB{FI}^{\text{u}}$ and $\MB{FI}^{\text{d}}$, and define the approximate discontinuous function $\breve f_{N,M}$ by
\[\breve f_{N,M}(t) = \left\{ \begin{array}{l}
{{\breve f}^{\max }_{N,M}},\quad 0 \le t < {{\tilde \xi }_1} \vee {{\tilde \xi }_2} \le t \le T,\\
{{\breve f}^{\min }_{N,M}},\quad {{\tilde \xi }_1} \le t < {{\tilde \xi }_2},
\end{array} \right.\]
if ${I_N}f(0) \in \MB{FI}^{\text{u}}$, or by
\[\breve f_{N,M}(t) = \left\{ \begin{array}{l}
{{\breve f}^{\min }_{N,M}},\quad 0 \le t < {{\tilde \xi }_1} \vee {{\tilde \xi }_2} \le t \le T,\\
{{\breve f}^{\max }_{N,M}},\quad {{\tilde \xi }_1} \le t < {{\tilde \xi }_2},
\end{array} \right.\]
otherwise. A pseudocode for the construction of $\breve f_{N,M}$ from the PS data is presented in Algorithm \ref{alg:2}. Table \ref{tab:OFICAAPOD3} shows the small absolute errors in the estimated extreme values of the discontinuous functions $f_6, \ldots, f_{12}$ using $M \in \{100, 400\}$, for increasing values of $N$. Figures \ref{fig:brevef6}-\ref{fig:brevef12} also show snapshots of the approximate discontinuous functions $\breve f_{6,N,M}, \ldots, \breve f_{12,N,M}$ over one period using $M = 400$, for increasing values of $N$, where $\breve f_{j,N,M}$ denotes the approximate discontinuous function $(\breve f_j)_{N,M} \forall j$.

\begin{table}[ht]
\caption{Observed absolute errors in the estimated extreme values using $M = 100$ and $M = 400$. All approximations are rounded to 5 significant digits.} % title of Table
\centering % used for centering table
\resizebox{1\columnwidth}{!}{%
\begin{tabular}{*{16}{c}} % centered columns (15 columns)
\toprule
 & $N$ & $\left|f_6^{\max}-{\breve f_{6,N,M}^{\max }}\right|$ & $\left|f_6^{\min}-{\breve f_{6,N,M}^{\min }}\right|$ & $\left|f_7^{\max}-{\breve f_{7,N,M}^{\max }}\right|$ & $\left|f_7^{\min}-{\breve f_{7,N,M}^{\min }}\right|$ & $\left|f_8^{\max}-{\breve f_{8,N,M}^{\max }}\right|$ & $\left|f_8^{\min}-{\breve f_{8,N,M}^{\min }}\right|$ & $\left|f_9^{\max}-{\breve f_{9,N,M}^{\max }}\right|$ & $\left|f_9^{\min}-{\breve f_{9,N,M}^{\min }}\right|$ & $\left|f_{10}^{\max}-{\breve f_{10,N,M}^{\max }}\right|$ & $\left|f_{10}^{\min}-{\breve f_{10,N,M}^{\min }}\right|$ & $\left|f_{11}^{\max}-{\breve f_{11,N,M}^{\max }}\right|$ & $\left|f_{11}^{\min}-{\breve f_{11,N,M}^{\min }}\right|$ & $\left|f_{12}^{\max}-{\breve f_{12,N,M}^{\max }}\right|$ & $\left|f_{12}^{\min}-{\breve f_{12,N,M}^{\min }}\right|$\\ [0.5ex] % inserts table
\midrule % inserts single horizontal line 
\multirow{20}{*}{\rotatebox{90}{$M = 100$}} & 100 & 0 & 5.46088E-03 & 0 & 4.49677E-03 & 0 & 4.50457E-04 & 2.49110E-05 & 0 & 0 & 3.35184E-05 & 0 & 9.11425E-05 & 4.12214E-03 & 0\\
& 200 & 4.99626E-03 & 4.99875E-03 & 0 & 2.51316E-03 & 9.32089E-03 & 3.03750E-02 & 4.89755E-02 & 3.53520E-03 & 4.95266E-04 & 2.28608E-02 & 7.44350E-03 & 1.73682E-03 & 1.45449E-02 & 1.81804E-02\\
& 300 & 0 & 0 & 0 & 0 & 0 & 0 & 3.02093E-02 & 0 & 0 & 0 & 0 & 3.25348E-03 & 3.96128E-03 & 0\\
& 400 & 0 & 7.93199E-05 & 3.28370E-03 & 1.05061E-03 & 1.43522E-03 & 7.84150E-03 & 3.91300E-02 & 1.86406E-03 & 7.36512E-05 & 2.06300E-03 & 8.31750E-03 & 3.33354E-02 & 1.22597E-02 & 1.66248E-03\\
& 500 & 0 & 6.65826E-05 & 1.49456E-05 & 1.68377E-05 & 9.26264E-05 & 1.13569E-03 & 2.86345E-02 & 0 & 0 & 7.84322E-03 & 0 & 1.80370E-03 & 1.90864E-03 & 2.03980E-04\\
& 600 & 1.39572E-03 & 1.60546E-03 & 0 & 0 & 0 & 0 & 5.36424E-03 & 0 & 8.76537E-05 & 1.73440E-03 & 0 & 6.07519E-03 & 6.45157E-03 & 6.11127E-04\\
& 700 & 0 & 5.86112E-05 & 0 & 5.43198E-05 & 0 & 6.00823E-05 & 3.87591E-04 & 0 & 0 & 3.66535E-04 & 1.92613E-04 & 4.75528E-03 & 1.73308E-03 & 0\\
& 800 & 0 & 5.54328E-05 & 2.61831E-04 & 1.14336E-04 & 5.59091E-04 & 5.05955E-03 & 1.03214E-02 & 2.92815E-04 & 1.01803E-04 & 4.36209E-03 & 1.77192E-03 & 9.55560E-03 & 3.90087E-03 & 3.24764E-04\\
& 900 & 0 & 0 & 0 & 0 & 0 & 0 & 2.41966E-05 & 0 & 0 & 0 & 0 & 0 & 0 & 0\\
& 1000 & 3.82151E-04 & 5.23656E-04 & 6.24540E-04 & 1.30314E-04 & 1.57500E-04 & 1.16577E-04 & 1.00520E-02 & 0 & 6.35917E-06 & 2.48792E-03 & 8.63638E-04 & 1.46659E-04 & 8.34752E-03 & 7.09763E-04\\
& 1100 & 0 & 0 & 0 & 0 & 0 & 0 & 0 & 0 & 0 & 0 & 0 & 0 & 0 & 0\\
& 1200 & 0 & 0 & 0 & 0 & 0 & 0 & 4.11854E-03 & 0 & 0 & 1.76179E-05 & 0 & 9.28325E-05 & 1.26346E-04 & 0\\
& 1300 & 0 & 3.65047E-05 & 0 & 7.18113E-05 & 0 & 5.71562E-05 & 8.10304E-05 & 0 & 0 & 7.69238E-05 & 0 & 2.75065E-04 & 6.27262E-04 & 0\\
& 1400 & 1.42066E-04 & 2.31099E-04 & 0 & 3.46255E-05 & 1.88635E-04 & 1.18611E-03 & 6.31568E-03 & 1.56579E-04 & 0 & 3.65358E-04 & 1.22283E-04 & 1.18564E-03 & 1.61653E-05 & 1.02878E-04\\
& 1500 & 0 & 8.64334E-05 & 0 & 0 & 0 & 0 & 1.35611E-03 & 0 & 0 & 5.89476E-04 & 0 & 7.21775E-04 & 5.77529E-05 & 0\\
& 2000 & 0 & 5.10837E-05 & 1.52090E-05 & 1.07136E-05 & 8.96558E-05 & 3.63116E-04 & 6.37900E-05 & 0 & 0 & 3.81931E-04 & 3.63493E-05 & 1.51160E-04 & 2.17857E-06 & 0\\
& 3000 & 0 & 3.18855E-05 & 0 & 0 & 0 & 0 & 5.74120E-04 & 0 & 0 & 3.16420E-04 & 1.52756E-04 & 2.07019E-03 & 2.15656E-05 & 0\\
& 4000 & 5.33208E-05 & 5.53811E-05 & 8.26423E-05 & 3.18614E-05 & 3.46062E-05 & 3.65272E-05 & 1.47634E-03 & 3.49097E-05 & 4.88532E-06 & 3.00460E-04 & 3.29526E-04 & 1.46024E-03 & 5.94479E-05 & 5.52881E-05\\
& 5000 & 0 & 5.06295E-05 & 0 & 2.09657E-06 & 0 & 4.73278E-06 & 3.15796E-04 & 0 & 0 & 2.55724E-06 & 0 & 4.07073E-06 & 2.82026E-06 & 0\\
& 10000 & 0 & 5.02725E-05 & 0 & 4.33234E-05 & 0 & 1.24372E-05 & 5.07503E-08 & 0 & 0 & 3.13849E-04 & 0 & 1.01557E-06 & 4.19042E-05 & 0\\
\midrule
\multirow{20}{*}{\rotatebox{90}{$M = 400$}} & 100 & 0 & 7.45429E-07 & 0 & 1.08003E-07 & 8.80460E-05 & 6.31829E-04 & 8.33915E-03 & 0 & 0 & 2.13359E-05 & 0 & 1.37385E-05 & 1.46014E-03 & 3.08440E-05\\
& 200 & 0 & 1.01475E-07 & 0 & 8.87797E-06 & 0 & 5.73833E-03 & 7.74449E-03 & 0 & 0 & 6.42337E-06 & 2.61235E-05 & 6.21617E-04 & 3.73607E-04 & 4.76268E-07\\
& 300 & 0 & 4.37162E-05 & 0 & 0 & 0 & 0 & 5.36993E-05 & 0 & 0 & 0 & 0 & 8.08618E-04 & 9.86152E-04 & 0\\
& 400 & 0 & 1.29624E-03 & 0 & 1.09308E-03 & 0 & 2.66733E-05 & 4.08258E-06 & 0 & 0 & 7.64073E-05 & 0 & 1.70227E-04 & 2.55136E-04 & 4.76633E-07\\
& 500 & 0 & 5.13750E-05 & 0 & 4.34465E-06 & 0 & 2.77615E-04 & 3.35999E-05 & 0 & 0 & 2.27026E-03 & 0 & 3.65232E-04 & 3.83288E-04 & 0\\
& 600 & 0 & 4.61691E-05 & 0 & 0 & 0 & 0 & 3.20427E-04 & 0 & 0 & 0 & 0 & 5.51513E-05 & 1.58906E-03 & 0\\
& 700 & 0 & 8.19535E-05 & 0 & 0 & 0 & 0 & 8.12122E-04 & 0 & 0 & 3.29693E-04 & 0 & 0 & 0 & 0\\
& 800 & 1.24994E-03 & 1.24998E-03 & 0 & 6.30566E-04 & 2.45823E-03 & 7.52286E-03 & 3.30175E-05 & 4.49499E-03 & 1.28307E-04 & 6.18159E-03 & 7.77397E-03 & 2.64532E-02 & 3.89492E-03 & 8.93570E-04\\
& 900 & 4.17024E-05 & 7.19074E-05 & 0 & 0 & 0 & 0 & 1.31607E-04 & 0 & 1.65860E-06 & 1.64881E-06 & 0 & 8.01157E-04 & 1.56246E-04 & 0\\
& 1000 & 0 & 1.34536E-05 & 0 & 4.84323E-06 & 0 & 1.15471E-05 & 7.87413E-05 & 5.36932E-07 & 2.14505E-07 & 8.21346E-06 & 0 & 1.27814E-04 & 9.26297E-05 & 0\\
& 1100 & 2.59035E-05 & 2.79315E-05 & 7.65316E-07 & 7.69717E-06 & 1.00792E-05 & 1.63958E-04 & 4.62507E-05 & 1.80063E-05 & 0 & 1.59331E-04 & 1.41455E-04 & 7.03096E-04 & 1.37554E-04 & 0\\
& 1200 & 0 & 0 & 0 & 0 & 0 & 0 & 2.19871E-06 & 0 & 0 & 0 & 0 & 0 & 0 & 0\\
& 1300 & 0 & 1.29724E-05 & 0 & 1.13542E-05 & 0 & 3.09164E-05 & 1.56925E-04 & 0 & 1.82411E-06 & 2.20495E-05 & 0 & 9.22740E-05 & 3.79064E-04 & 0\\
& 1400 & 0 & 0 & 0 & 0 & 0 & 1.38501E-04 & 5.58243E-04 & 0 & 0 & 1.56315E-04 & 0 & 3.51820E-06 & 0 & 0\\
& 1500 & 0 & 2.09907E-05 & 0 & 0 & 0 & 0 & 7.19364E-04 & 1.19776E-05 & 0 & 1.35936E-04 & 1.22263E-04 & 5.90207E-04 & 2.09164E-05 & 0\\
& 2000 & 0 & 4.13727E-06 & 8.73579E-07 & 9.97008E-07 & 5.81071E-06 & 7.10885E-05 & 1.22050E-04 & 2.72419E-07 & 0 & 1.94218E-03 & 0 & 4.49886E-05 & 9.27621E-05 & 0\\
& 3000 & 0 & 0 & 0 & 0 & 0 & 0 & 5.24378E-05 & 0 & 0 & 3.56397E-05 & 0 & 0 & 0 & 0\\
& 4000 & 1.20035E-04 & 1.21493E-04 & 1.88141E-04 & 3.12670E-05 & 5.17704E-05 & 6.94291E-06 & 1.12754E-03 & 1.36414E-05 & 1.76000E-06 & 2.09533E-04 & 2.74553E-04 & 1.11656E-03 & 3.63409E-05 & 0\\
& 5000 & 0 & 3.26496E-06 & 0 & 1.90552E-07 & 0 & 3.26261E-06 & 5.22698E-05 & 0 & 2.52510E-07 & 1.04305E-05 & 2.56480E-05 & 1.27053E-04 & 1.27472E-04 & 3.86535E-06\\
& 10000 & 0 & 3.16213E-06 & 0 & 2.80268E-06 & 0 & 9.98510E-07 & 4.02506E-06 & 3.06259E-07 & 0 & 2.24589E-05 & 0 & 9.06546E-05 & 2.70250E-06 & 0\\
\bottomrule %inserts single line
\end{tabular}
}
\label{tab:OFICAAPOD3} % is used to refer this table in the text
\end{table}

\begin{figure}[H]
\hspace{-2cm}\includegraphics[scale=0.4]{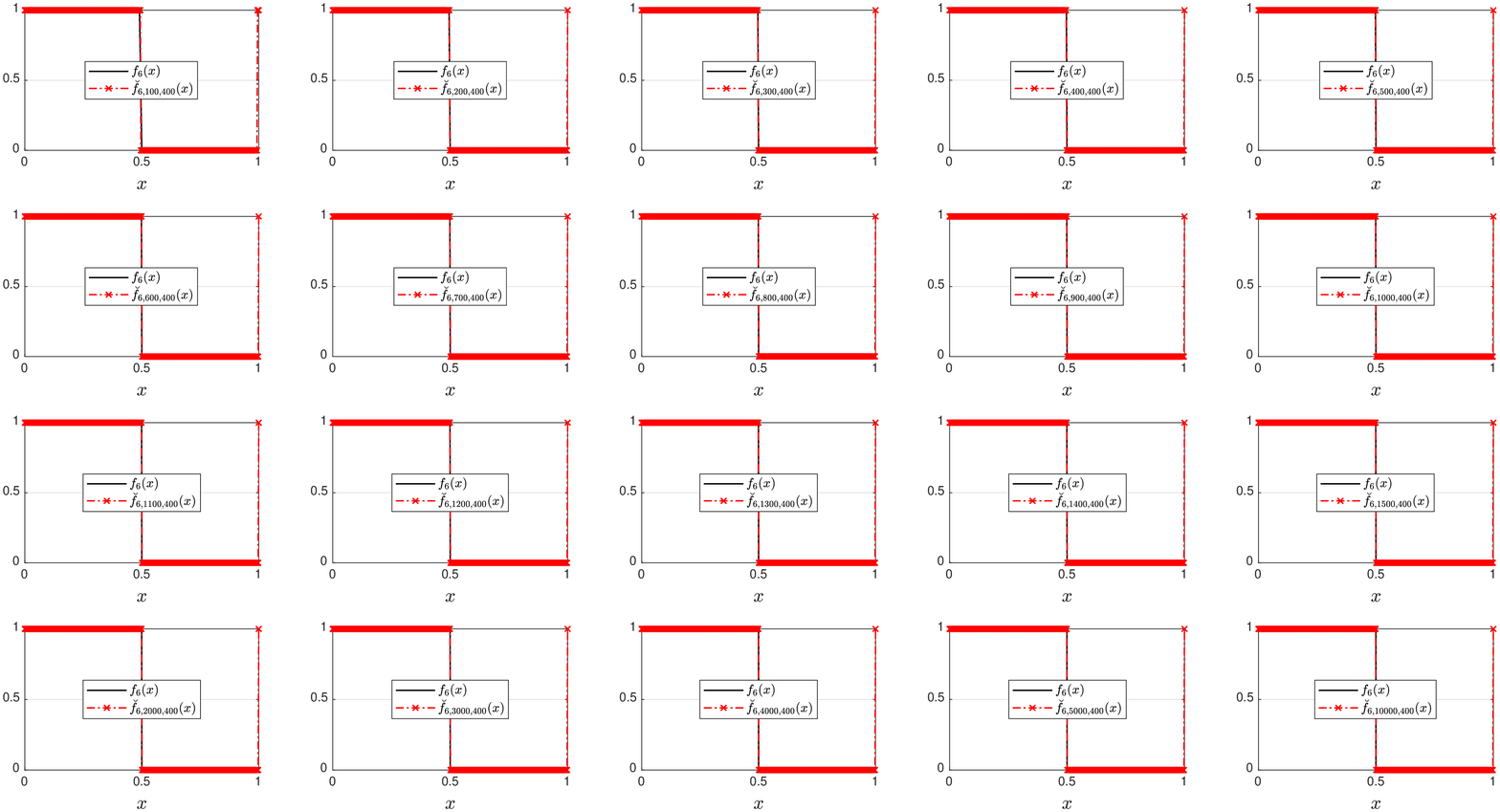}
\vspace{-1cm}\caption{Snapshots of the approximate discontinuous function $\breve f_6$ over one period using $M = 400$, for increasing values of $N$.}
\label{fig:brevef6}
\end{figure}

\begin{figure}[H]
\hspace{-2cm}\includegraphics[scale=0.4]{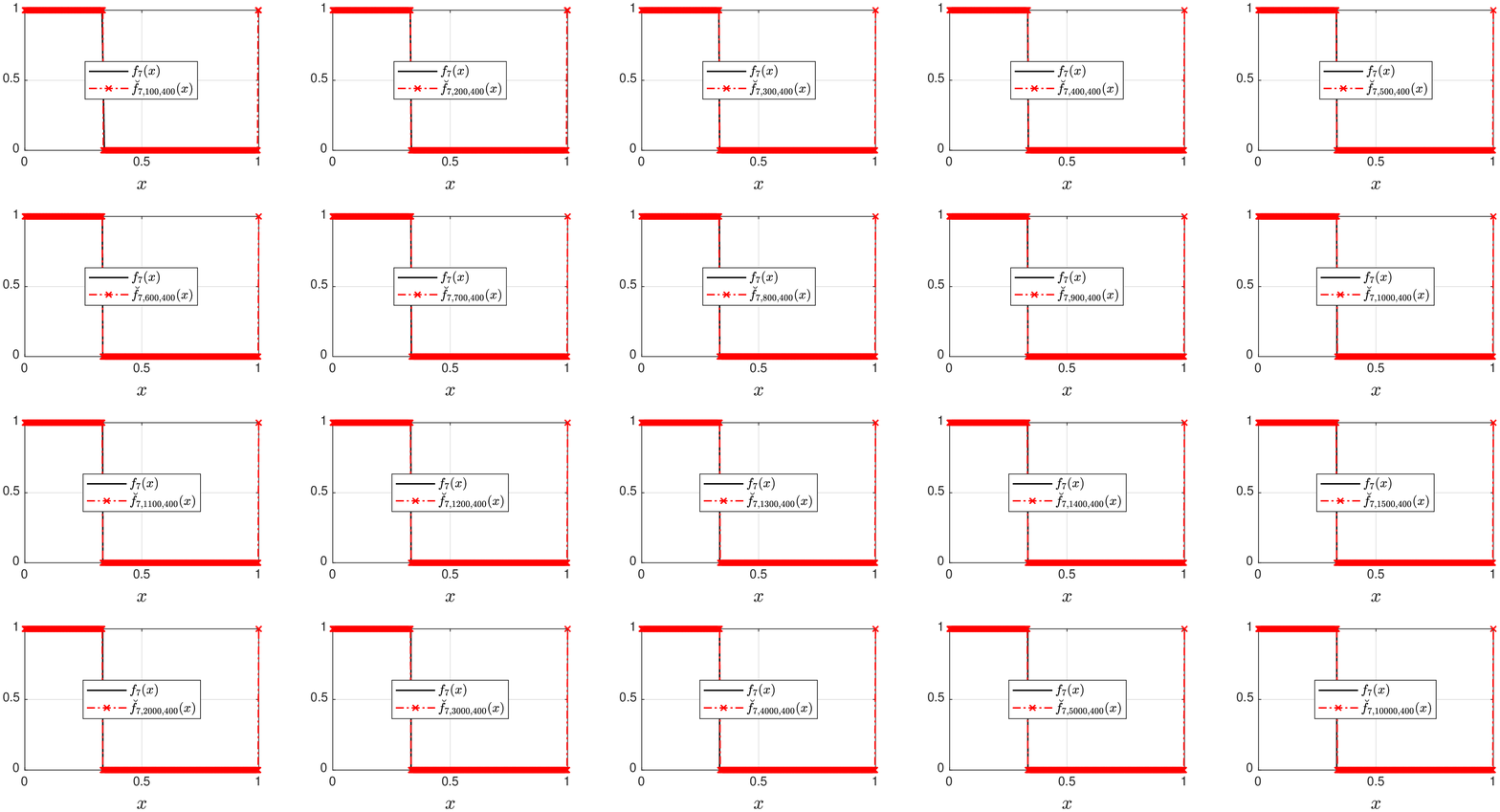}
\vspace{-1cm}\caption{Snapshots of the approximate discontinuous function $\breve f_7$ over one period using $M = 400$, for increasing values of $N$.}
\label{fig:brevef7}
\end{figure}

\begin{figure}[H]
\hspace{-2cm}\includegraphics[scale=0.4]{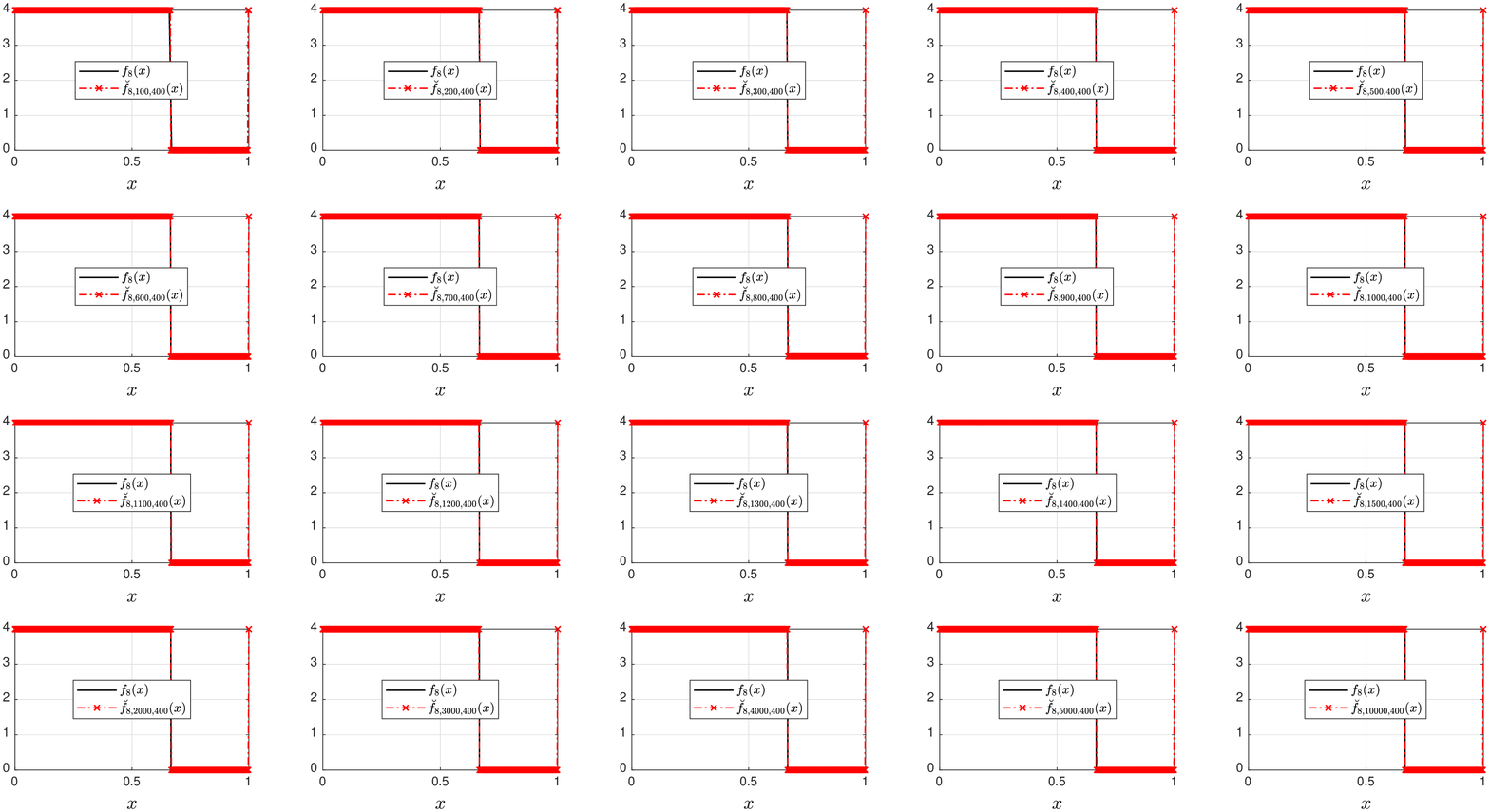}
\vspace{-1cm}\caption{Snapshots of the approximate discontinuous function $\breve f_8$ over one period using $M = 400$, for increasing values of $N$.}
\label{fig:brevef8}
\end{figure}

\begin{figure}[H]
\hspace{-2cm}\includegraphics[scale=0.4]{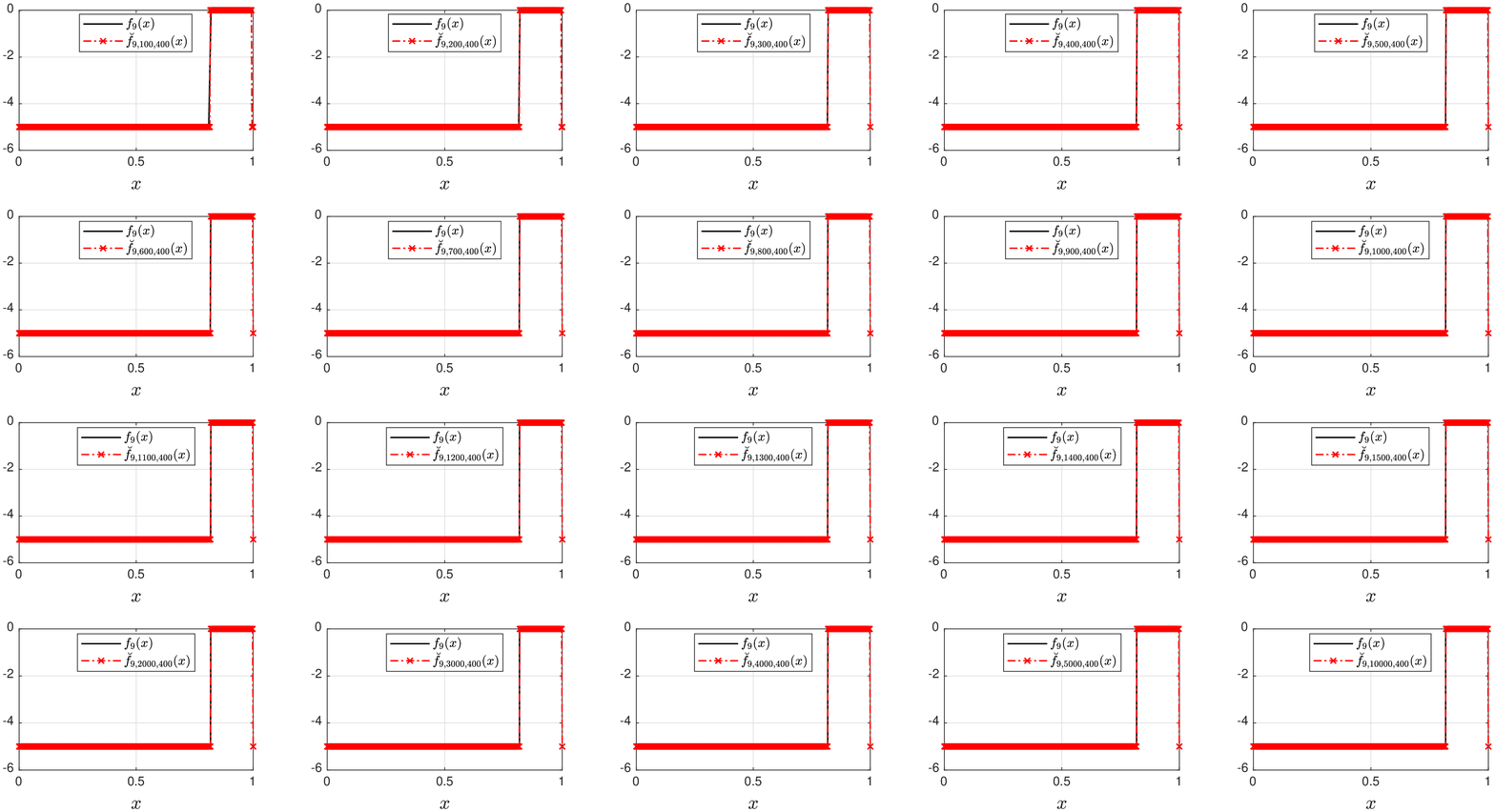}
\vspace{-1cm}\caption{Snapshots of the approximate discontinuous function $\breve f_9$ over one period using $M = 400$, for increasing values of $N$.}
\label{fig:brevef9}
\end{figure}

\begin{figure}[H]
\hspace{-2cm}\includegraphics[scale=0.4]{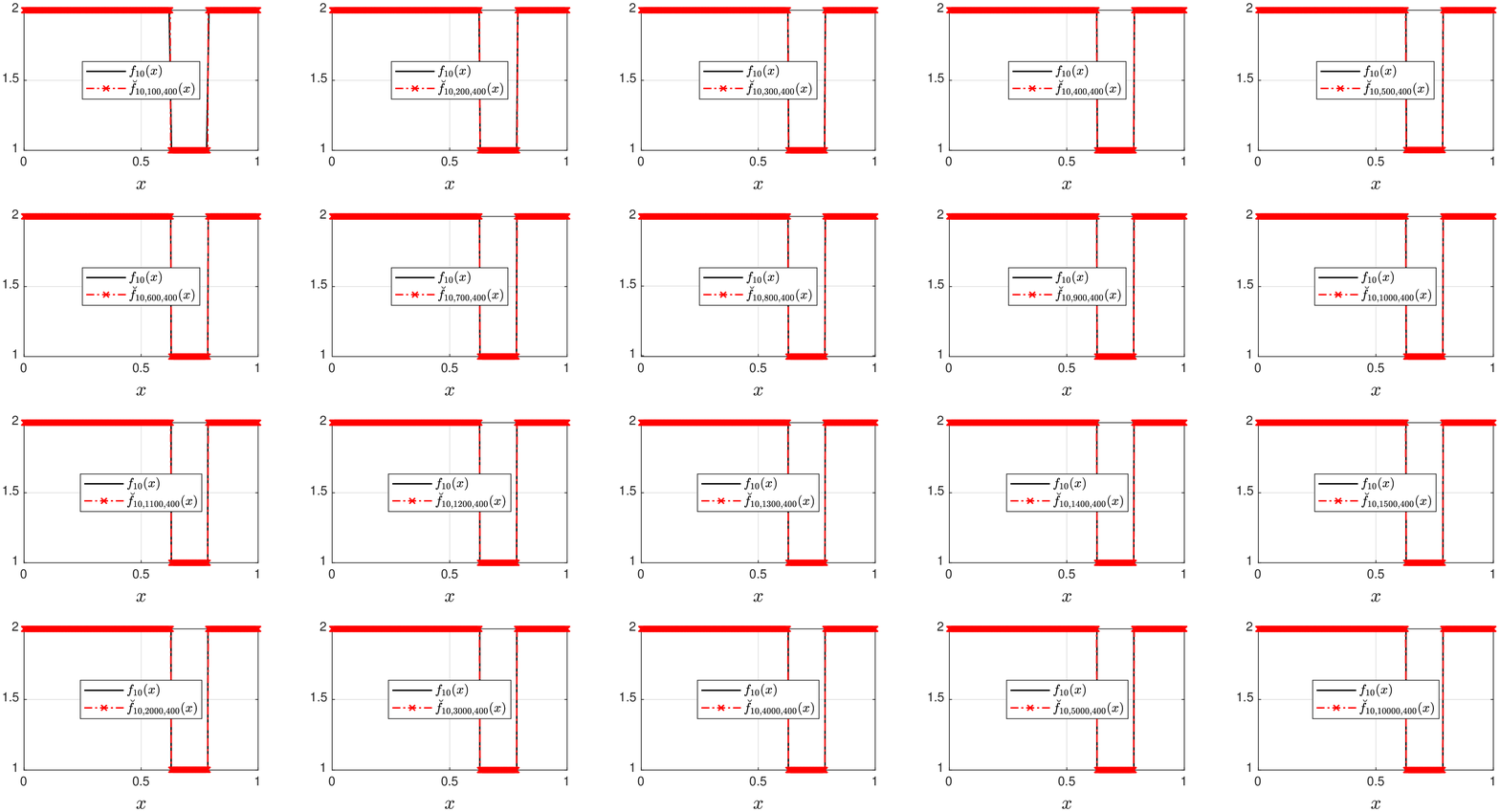}
\vspace{-1cm}\caption{Snapshots of the approximate discontinuous function $\breve f_{10}$ over one period using $M = 400$, for increasing values of $N$.}
\label{fig:brevef10}
\end{figure}

\begin{figure}[H]
\hspace{-2cm}\includegraphics[scale=0.4]{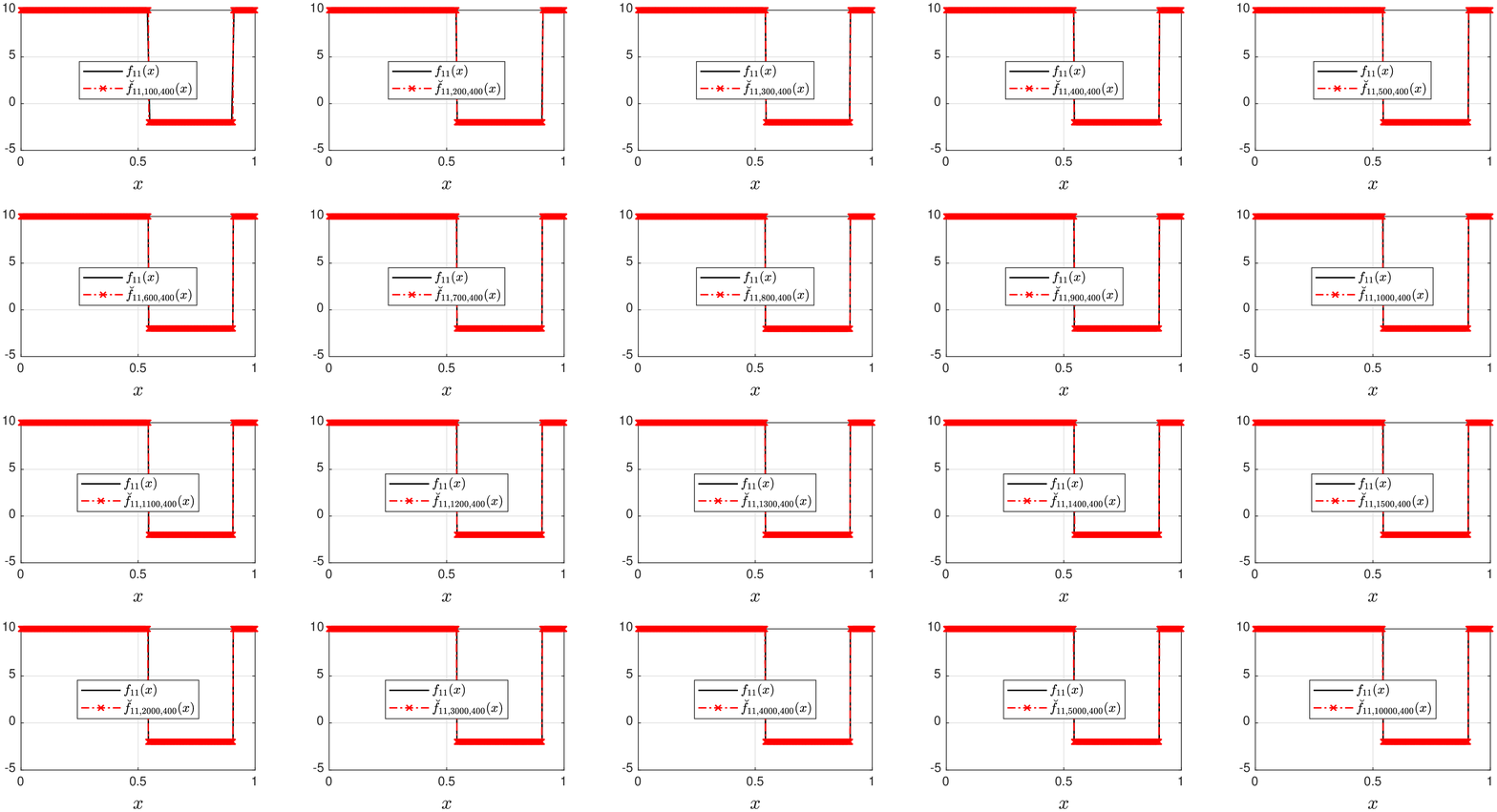}
\vspace{-1cm}\caption{Snapshots of the approximate discontinuous function $\breve f_{11}$ over one period using $M = 400$, for increasing values of $N$.}
\label{fig:brevef11}
\end{figure}

\begin{figure}[H]
\hspace{-2cm}\includegraphics[scale=0.4]{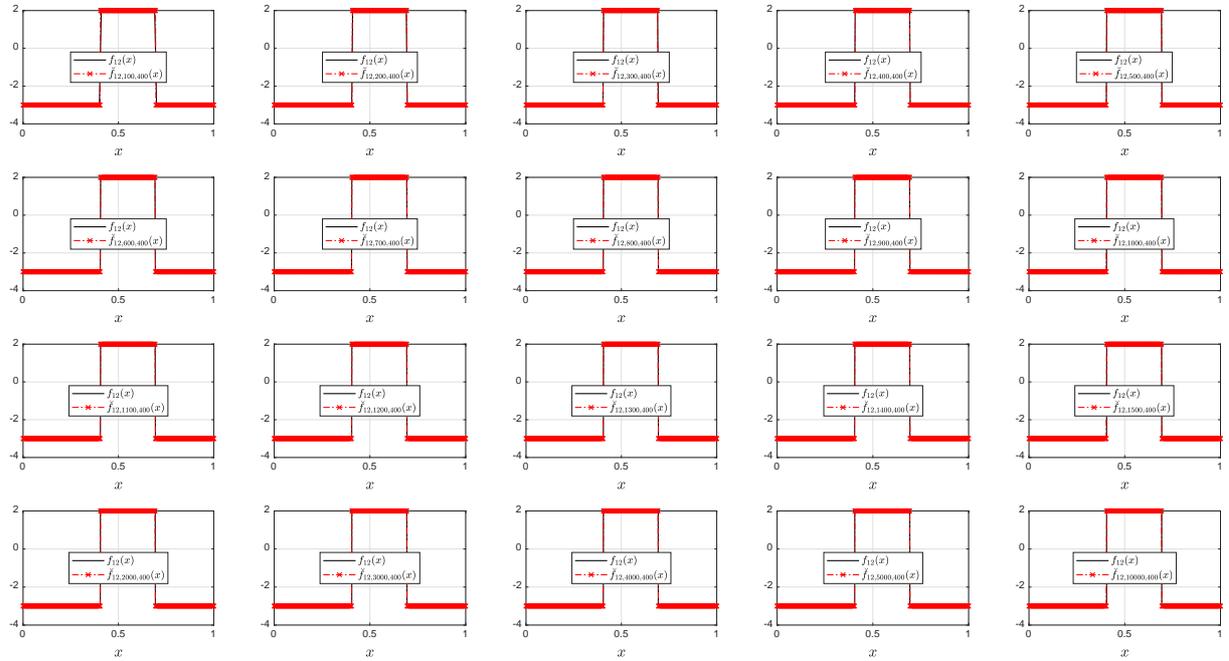}
\vspace{-1cm}\caption{Snapshots of the approximate discontinuous function $\breve f_{12}$ over one period using $M = 400$, for increasing values of $N$.}
\label{fig:brevef12}
\end{figure}

\section{Computational Algorithms}
\label{app:Alg1}
In this section we provide two computational algorithms for the fast, accurate, and economic construction of FPSI matrices and reconstructing an approximate piecewise analytic function from the FPS data.

\begin{algorithm}[ht]\footnotesize
\renewcommand{\thealgorithm}{1}
\caption{Faster, More Accurate, and More Economic Construction of $\Fthe$}
\label{alg:1}
\begin{algorithmic}[H]
\Require A positive real number $T$; a column vector $\bm{x}$ of even $N$-equally-spaced points.
\Ensure The elements $\theta _{l,j}, 1 \le l,j \le N$, of $\Fthe$.
\State Set $\C{N}_{\bm{x}} \leftarrow N; N - 1 \leftarrow N_{-1}; N + 2 \leftarrow  N_{+2};\; N + 3 \leftarrow N_{+3};\; N/2 \leftarrow N_{\div 2};$\quad \Comment{$\C{N}_{\bm{x}}$ is the number of elements of $\bm{x}$.}\\
$\F{O}_N \leftarrow \Fthe; 2 \pi/T \leftarrow c_1;$\quad \Comment{$\F{O}_N$ is the zeros matrix of size $N$.}\\
$[-N_{\div 2}:N_{\div 2}-1]_{\neq 0} \leftarrow \bm{K}; i/c_1 \leftarrow c_2;$\quad \Comment{The notation $[\cdot]_{\neq 0}$ means excluding $0$.}\\
$\exp((c_1 i \bm{K}) \otimes \bm{x}) \leftarrow \F{A};\; 1 - \F{A}_{2:N,:} \leftarrow \F{B};$\quad \Comment{$\F{Z}_{n,:}$ is a row vector whose elements form the $n$th row of a matrix $\F{Z}$.}
\For{$l = 2:N$}
\State Set $l - 1 \leftarrow l_{-1};\;1:\left\lfloor {\left( {l + 3} \right)/2} \right\rfloor \leftarrow \bm{C};\;2:\left\lfloor {\left( {N_{+3} - l} \right)/2} \right\rfloor \leftarrow \bm{D};$ \Comment{$\left\lfloor  \cdot  \right\rfloor$ is the floor function.}
\For{$j = \bm{C}$}
\State $\sum {\left[\F{B}_{l_{-1},:} \oslash \left(\bm{K} \odot \F{A}_{j,:}\right)\right]} \leftarrow \theta_{l,j};$\quad \Comment{$\sum{}$ is the sum of array elements.} 
\EndFor
\State Set $\theta_{l,\bm{C}} \leftarrow \theta_{l,l+1-\bm{C}};$
\If {$l \le N_{-1}$} 
\For{$j = \bm{D}$}
\State $l_{-1} + j \leftarrow k;\; \sum {\left[ {\F{B}_{{l_{ - 1}},:} \oslash (\bm{K} \odot \F{A}_{{k},:)}} \right]} \leftarrow \theta_{l,k};$
\EndFor
\State $\theta_{l,l_{-1}+\bm{D}} \leftarrow \theta_{l,N_{+2}-\bm{D}};$
\EndIf
\State $\bm{x}_l + \text{Re}(c_2 \theta_{l,:}) \leftarrow \theta_{l,:};$\quad \Comment{\text{Re} is the real part of a complex number.}
\EndFor
\State $\Fthe/N \leftarrow \Fthe;$
\end{algorithmic}
\end{algorithm}

%The following arrays are needed to apply the CPSLSM:
%N \in \MBZeP; a row vector $\bm{K} = -N_{\div 2}:N_{\div 2}-1$ for $N \in \MBZeP$; Fourier coefficients $\tilde f_{\bm{K}}$;
%
%$M \in \MBZ^+$ in the algorithm is a relatively large number.
\begin{algorithm}[H]\footnotesize
\renewcommand{\thealgorithm}{2}
\caption{Reconstruction of the Approximate piecewise analytic Function $\breve f_{N,M}$ from the FPS Data}
\label{alg:2}
\begin{algorithmic}[1]
\Require Numbers $T \in \MBR^+, M \in \MBZ^+$; equally-spaced column vector $\bm{y}_M^+$; Fourier interpolant values column vector $\bm{INf} = I_Nf(\bm{y}_M^+)$; user-defined tolerances $\tilde \epsilon, \varepsilon$.
\Ensure The approximate piecewise analytic function $\breve f$.
\State Set $\emptyset \leftarrow \Xi; M^{-} \leftarrow M - 1; \emptyset \leftarrow \Lambda; \emptyset \leftarrow \rho_1; \bmzer \leftarrow [\rho_2; \rho_3]$;\quad \Comment{$\emptyset$ denotes the empty vector.}
\State $\indmin \bm{INf} \leftarrow \nu_{1}; \indmax \bm{INf} \leftarrow \nu_{2}$;
\State Calculate $I_N^{\min}f$ and $I_N^{\max}f$ using the CPSLSM;\quad \Comment{See \ref{sec:DTD2}.}
\State $\frac{1}{2} \left(I_N^{\max}f + I_N^{\min}f\right) \leftarrow \mu; \left| {\bm{INf} - \mu \bmone_{M+1}} \right| \leftarrow d_1; \min d_1 \leftarrow d_2; \epsilon = \tilde \epsilon\,(I_N^{\max}f - I_N^{\min}f)$;
\If {$d_2 \le \epsilon$} 
\State $\ind (d_1 \le \epsilon) \leftarrow \Lambda; \rho_2 = (\Lambda==M^{-}); \Lambda(\rho_2) = M; y_{M,\Lambda} \leftarrow \Xi; \rho_1 = \Lambda;$
\EndIf
\If {$\left|\Lambda\right| < 2$}\quad \Comment{$\left|\Lambda\right|$ is the length of $\Lambda$.} 
\State Calculate $I_N^{{\text{aux}}}f({\bm{y}_M^+})$ using Eq. \eqref{eq:INaux1};
\State Find $J_{1:2}: I_N^{{\text{aux}}}f(y_{M,J_l}) - I_N^{{\text{aux}}}f(y_{M,J_l+1}) \ne 0,\; l = 1, 2$;
   \If {$J_2 = M^{-}$ and $\text{any}(\rho_2) = 0$}\quad \Comment{$\begin{array}{l}
   \text{any}(\F{A})\text{ returns }1\text{ if any of the elements of }\F{A}\\
    \text{is a nonzero number, and returns }0\text{ otherwise}.
   \end{array}$}
      \State $\Xi = [\Xi, T]; \rho_3 = 1;$
   \EndIf
   \If {$\left|\Xi\right| < 2$}
      \If {$\rho_1 = \emptyset$ and $\rho_3 = 0$}
         \State $\frac{1}{2} \left(y_{M,J_{1:2}} + y_{M,J_{1:2}+1}\right) \leftarrow \Xi$;
      \Else
         \For{$l = 1:2$}
            \If {$y_{M,J_l} - \Xi > \varepsilon$\; or\; $\Xi - y_{M,J_l+1}  > \varepsilon$} 
               \State $\text{sort}\left(\left[\Xi; \frac{1}{2} \left(y_{M,J_{l}} + y_{M,J_{l}+1}\right)\right]\right) \leftarrow \Xi$;\quad \Comment{sort($\F{A}$) sorts the elements of $\F{A}$ in ascending order.}
  	        \EndIf
         \EndFor
      \EndIf
   \EndIf
\EndIf
\State {$\ind (\bm{INf} > \mu) \leftarrow \bar J_1; \ind (\bm{INf} < \mu) \leftarrow \bar J_2;$\\ 
$\text{median}\left(\bm{INf}_{\bar J_1}\right) \leftarrow \breve f^{\max}; \text{median}\left(\bm{INf}_{\bar J_2}\right) \leftarrow \breve f^{\min}$;\quad \Comment{median gives the median value of an array.}}
\If {$\bm{INf}_0 > \mu$} 
\State {$\left\{ \begin{array}{l}
{{\breve f}^{\max }},\quad 0 \le t < {{\Xi }_1} \vee {{\Xi }_2} \le t \le T,\\
{{\breve f}^{\min }},\quad {{\Xi }_1} \le t < {{\Xi }_2}
\end{array} \right. \leftarrow \breve f$;}
\Else
\State {$\left\{ \begin{array}{l}
{{\breve f}^{\min }},\quad 0 \le t < {{\Xi }_1} \vee {{\Xi }_2} \le t \le T,\\
{{\breve f}^{\max }},\quad {{\Xi }_1} \le t < {{\Xi }_2}
\end{array} \right. \leftarrow \breve f$;}
\EndIf
\end{algorithmic}
\end{algorithm}

\section{Computational Complexity and Speed of FPSI Matrices}
\label{subsec:CCAA1}
Most parts of Algorithm \ref{alg:1} are optimized and arranged to work on chunks of vectors and matrices; thus, their efficiency increases by allowing vectorized operations. To analyze the computational cost of the algorithm, note that the first $3$ lines require $8$ arithmetic operations. The Kronecker product in Line $4$ requires $N (N-1)$ multiplications; therefore, the line requires $1+3N (N-1)$ arithmetic operations. Line $6$ required $5$ arithmetic operations. Line $8$ requires $3 N-4$ operations. Line $10$ requires $1 + \left\lfloor {\frac{{l + 3}}{2}} \right\rfloor $ additions and subtractions. Line $13$ requires $3 (N-1)$ arithmetic operations. Line $15$ requires $\left\lfloor {\frac{{N + 3 - l}}{2}} \right\rfloor  - 1$ additions and subtractions and Line $17$ requires $2 N$ additions and multiplications. Therefore, the for loop in Lines $5-18$ require
\begin{align*}
&\mathop \sum \limits_{l = 2}^{N - 1} \left( {\mathop \sum \limits_{j = 2}^{\left\lfloor {\frac{1}{2}\left( { - l + N + 3} \right)} \right\rfloor } 3\left( {N - 1} \right) + \left\lfloor {\frac{1}{2}\left( { - l + N + 3} \right)} \right\rfloor  + \left( {3N - 4} \right)\left\lfloor {\frac{{l + 3}}{2}} \right\rfloor + \left\lfloor {\frac{{l + 3}}{2}} \right\rfloor  + 2N + 5} \right) + \left( {3N - 4} \right)\left\lfloor {\frac{{N + 3}}{2}} \right\rfloor\\
&+ \left\lfloor {\frac{{N + 3}}{2}} \right\rfloor  + 2N + 6 = 3\left( {N - 1} \right)\left\lfloor {\frac{{N + 3}}{2}} \right\rfloor  + 2N + 6 + \left\{ \begin{array}{l}
0,\quad N = 2,\\
\frac{3}{4}\left( {N - 2} \right)\left( {N\left( {2N + 7} \right) + 2} \right),\quad N \ge 4,
\end{array} \right.
\end{align*}
arithmetic operations. Since Line $19$ requires $N^2$ divisions, the exact total cost of the algorithm, TC$_{\text{new}}$, is 
\[\text{TC}_{\text{new}} = 3\left( {N - 1} \right)\left\lfloor {\frac{{N + 3}}{2}} \right\rfloor  + N\left( {4N - 1} \right) + 15 + \left\{ \begin{array}{l}
0,\quad N = 2,\\
\frac{3}{4}\left( {N - 2} \right)\left( {N\left( {2N + 7} \right) + 2} \right),\quad N \ge 4
\end{array} \right. = O\left( {\frac{3}{2}{N^3}} \right),\quad {\text{as }}N \to \infty .\]
The operational count of (\cite[Algorithm 3.1]{elgindy2019high}) was roughly estimated to be of $O\left( {\frac{1}{2}{N^3}} \right)$, for large values of $N$, however, if we attempt to calculate its exact total cost, TC$_{\text{old}}$, we count $6$ arithmetic operations in the precomputation of the constants $\displaystyle{\frac{Ti}{2 \pi}}$ and $\displaystyle{\frac{2 \pi i}{T}}$. Line $3$ of that algorithm would now require $\displaystyle{\left( {11N + 3} \right)\left( {\left\lfloor {\frac{{l - 1}}{2}} \right\rfloor  + 2} \right)}$ arithmetic operations. Line $4$ requires $\displaystyle{\left\lfloor {\frac{{l - 1}}{2}} \right\rfloor  + 1}$ subtractions. Line $6$ requires $\displaystyle{\left( {11N + 4} \right)\left( {\left\lfloor {\frac{1}{2}\left( { - l + N - 1} \right)} \right\rfloor  + 1} \right)}$ arithmetic operations, and Line $7$ requires $\displaystyle{2\left\lfloor {\frac{1}{2}\left( { - l + n - 1} \right)} \right\rfloor}$ additions and subtractions. Therefore, a precise estimate of TC$_{\text{old}}$ would be 
\[\text{TC}_{\text{old}} = \left( {11N + 4} \right)\left\lfloor {\frac{N}{2}} \right\rfloor  + 11N + 9 + \left\{ \begin{array}{l}
0,\quad N = 2,\\
\frac{1}{2}\left( {N - 2} \right)\left( {\left( {11N + 38} \right)N + 8} \right),\quad N \ge 4
\end{array} \right. = O\left( {\frac{{11}}{2}{N^3}} \right),\quad {\text{as }}N \to \infty.\]
Hence, ${\text{TC}}_{{\text{new}}}/{{{\text{TC}}_{{\text{old}}}}} \sim 3/11 = 0.\overline{27}$,
i.e., the current algorithm requires approximately one-quarter of the total cost of the previous algorithm. Figure \ref{fig:Fig3_TimeComp1} shows a comparison between the elapsed time (ET) to perform (\cite[Direct Formulas (3.3) and Algorithm 3.1]{elgindy2019high}) and the current algorithm. The calculated execution times were measured multiple times, and the data are shown as the median of the time measurements in seconds (s). For $N = 200$, the ET of (\cite[Algorithm 3.1]{elgindy2019high}), ET$_{\text{old}}$, was approximately $0.359$ s, whereas the ET recorded for the current algorithm, ET$_{\text{new}}$, was approximately $0.084$ s. Thus, $\text{ET}_{\text{new}}/{\text{ET}_{\text{old}}} \approx 0.234$, as nearly expected. The results clearly demonstrate that the proposed algorithm is superior to (\cite[Direct Formulas (3.3) and Algorithm 3.1]{elgindy2019high}) in terms of speed and computational cost.

\begin{figure}[ht]
\centering
\includegraphics[scale=0.25]{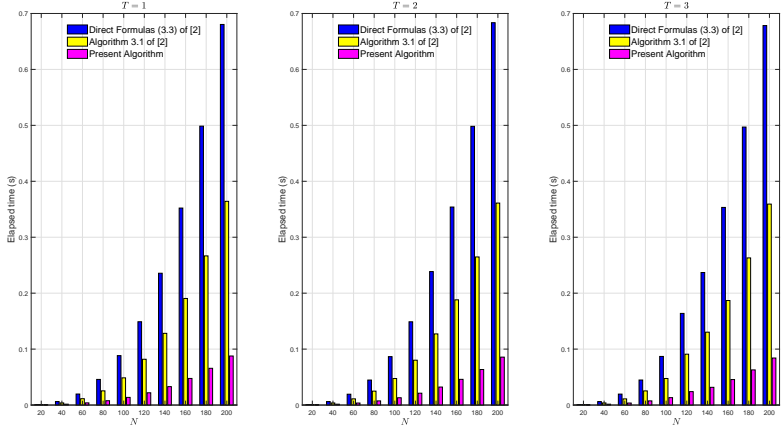}
\caption{ET of constructing $\Fthe$ using (\cite[Direct Formulas (3.3) and Algorithm 3.1]{elgindy2019high}) and the present algorithm, for $T = 1:3$ and $N = 20(20)200$.}
\label{fig:Fig3_TimeComp1}
\end{figure}

\section{Efficient and Stable Computation of SG Matrices}
\label{sss:ECSGM1}
It is noteworthy to mention that ${}_k\F{P}$ and ${}_k\hat{\F{P}}$ are directly related to the usual first-order Gegenbauer integration matrices (GIMs) in barycentric form $\F{P}$ and $\hat{ \F{P}}$ derived by \citet{Elgindy20171} by the useful identity
\begin{equation}\label{eq:qaccess1}
\left[{}_k\F{P}, {}_k\hat{\F{P}}\right] = (\tau_k^-) \left[\F{P},\hat{\F{P}}\right],\quad \forall k;
\end{equation}
see \citet[Algorithms 1 and 8]{Elgindy20171} for how to efficiently construct the pair of matrices $\F{P}$ and $\hat{\F{P}}$. Eq. \eqref{eq:qaccess1} allows us to calculate definite integrals over any partition $\F{\Gamma}_k$ in the physical space by premultiplying either $\F{P}$ or $\hat{\F{P}}$ by the constant factor $\tau_k^-$; thus, reducing the overall amount of computations required significantly.

To evaluate the necessary definite integrals of $\psi$ over the intervals $\F{\Omega}_{x_{N,0:N-1}}$, one still needs to evaluate $\C{I}_{\tau_{k-1}, \tau_{k}}^{(t^{(k)})} \psi \foralle k \in \MBK_K\backslash\{K\}$, assuming that $x_{N,N-1} \in \F{\Gamma}_K$. For instance, to piecewise integrate $\psi$ on the intervals $\F{\Omega}_{x_{N,p:n}}\,\foralls p, n \in \MBJ_N: \{x_{N,p:n}\} \subset \F{\Gamma}_2$, one needs first to calculate $\C{I}_{\F{\Gamma}_1}^{(t^{(k)})} \psi$ using the row vector ${{}_1{\F{P}_{N_k+1}}}$ before using ${}_2\F{P}$ to estimate the required integrals on $[\tau_1, x_{N,i}]_{i = p:n}$. One can similarly calculate ${{}_k{\F{P}_{N_k+1}}}\,\forall k \in \MBK_K$ using the useful formula
\begin{equation}
{{}_k{\F{P}_{N_k+1}}} = (\tau_k^-) \F{P}_{N_k+1},
\end{equation}
where $\F{P}_{N_k+1}$ is a row vector whose elements form the $(N_k+1)$st-row of the barycentric GIM constructed using \cite[Algorithm 6 or 7]{Elgindy20171}.

Figure \ref{fig:FAccuracy4_fn6} shows the SG quadrature error infinity- and -Euclidean norms on a log-lin scale of the definite integrals of the square wave function $f_6$ when successively integrated over the intervals $\F{\Omega}_{x_{N,0: N-1}}$ using only two GG points and the Gegenbauer parameter (index) value $\alpha = -0.1$. Note that the obtained integral approximations are accurate to almost full precision in double-precision floating-point arithmetic because $f_6$ is a linear piecewise function and the $2$-point SG quadrature is exact for polynomials of degree at most three. For piecewise constant functions, the SG quadrature truncation error collapses for $N_k \ge 0$, as indicated by Theorem \ref{sec:erranalysgp1}, and the computational error is dominated by the maximum error in the approximate jump discontinuity points and the extreme values of the discontinuous function. This outcome is consistent with Figure \ref{fig:FAccuracy4SGG_fn6}, which shows plots of the $2$-point SG quadrature error infinity- and -Euclidean norms on the log-lin scale of the reconstructed square wave function $\breve f_{6,N,400}$ when successively integrated over the intervals $\F{\Omega}_{x_{N,0: N-1}}$ for $\alpha = -0.1$ and several increasing values of $N$. Clearly, the quadrature error infinity-norm has the same maximum error order of the approximate jump discontinuity points and extreme values recorded in Tables \ref{tab:OFICAAPOD2} and \ref{tab:OFICAAPOD3} whenever an approximation error exists. In the absence of approximation errors, the quadrature error infinity-norm approaches the machine epsilon. The figure also shows the corresponding FPSQ error infinity-norms using FPSI matrices of the same size as the mesh grid, as recorded in Table \ref{tab:GGQf6to12}, where we observe the slow decay of the errors as the mesh grid size increases, regardless of how well the reconstructed square wave function is set up. In fact, Table \ref{tab:GGQf6to12} asserts this finding as it shows that the calculated absolute errors in the definite integrals of $f_6, \ldots, f_{12}$ decay, at best, like $O\left(N^{-1/2}\right)$, when they are approximated by FPSQs using the reconstructed $\breve f_{6,N,M}, \ldots, \breve f_{12,N,M}$ obtained by Algorithm \ref{alg:2} with $M = 100$ and $M = 400$.

To support our analysis further, consider the problem of evaluating $\C{I}_{\F{\Omega}_{x_{N,0: N-1}}}^{(t)} \psi$ when
\[s(t) = \left\{ \begin{array}{l}
\frac{2}{3}(t + 1),\quad0 \le t < {\xi _1},\\
 - \frac{1}{{10}}{(t - {\xi _1})^2} + \frac{2}{3}({\xi _1} + 1),\quad{\xi _1} \le t < {\xi _2},\\
\sin \left( {\frac{1}{4}(t - {\xi _2})} \right) - \frac{1}{{10}}{({\xi _2} - {\xi _1})^2} + \frac{2}{3}({\xi _1} + 1),\quad{\xi _2} \le t < T,
\end{array} \right.\quad \text{and}\quad u(t) = \left\{ \begin{array}{l}
0,\quad0 \le t < {\xi _1},\\
2,\quad{\xi _1} \le t < {\xi _2},\\
0,\quad{\xi _2} \le t < T.
\end{array} \right.\]
Here, we can calculate the exact required integrals of $\psi$ because both $s$ and $u$ are available in closed form and their antiderivatives can be written in terms of elementary functions. The plots of $s, u, \psi$ and the error plot for approximating the required integrals using the SG quadratures are shown in Figure \ref{fig:PsiDemo1} for some parameter values. The maximum absolute error recorded was approximately $2.66$E–$15$, which demonstrates the high accuracy of the SG quadratures for exact input data and their exponential convergence using a relatively small number of quadrature nodes. To test the stability of the SG quadratures, we perturb the jump discontinuity points such that $\xi_{1:2} = \tilde \xi_{1:2} + \delta_{1:2}\,\foralls \delta_{1:2}^t \in \canczer{\MBR}^2$, and denote the perturbed state control and state derivative variables by $\tilde s, \tilde u$, and $\tilde \psi$, respectively. Figure \ref{fig:PsiDemo2} shows the exact and perturbed functions for $\delta_{1:2} = [$-$1$E-$5, 1$E-$6]$ in addition to the error plots generated using the same parameter values. The maximum absolute error recorded was approximately $1.08$E-$05$, which verifies the numerical stability of the numerical scheme, as the quadrature errors are nearly equal to the maximum perturbations incorporated in the mathematical model. It is interesting to mention here that better estimates were obtained when we evaluated the definite integrals of $\psi$ over partitions determined by SGG points instead of equispaced nodes. In particular, the maximum absolute error in evaluating $\C{I}_{\F{\Omega}_{\hat t_{N_k,0:N_k}^{(k),\alpha}}}^{(t)} \psi\,\forall k \in \MBK_3$ using the SG quadratures and the same parameter values were approximately $2.22$E-$15$ assuming that $\tau_1 = \xi_1, \tau_2 = \xi_2$, and $K = 3$; see Figure \ref{fig:PsiDemo3}.

\begin{figure}[ht]
\centering
\includegraphics[scale=0.2]{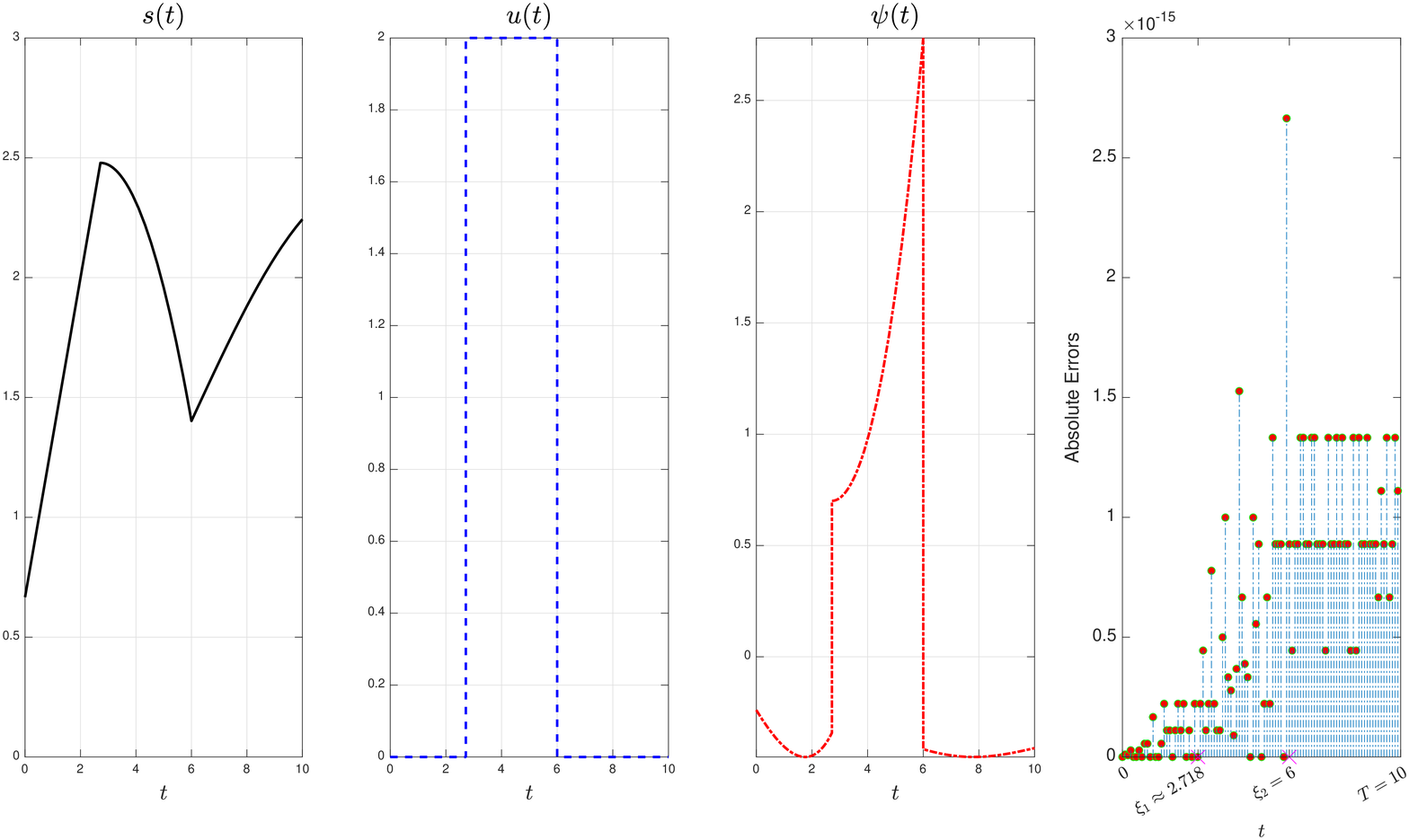}
\caption{Plots of $s, u, \psi$ and the errors plot of approximating $\C{I}_{\F{\Omega}_{x_{100,0: 99}}}^{(t)} \psi$ using the SG quadratures for the parameter values $T = 10, \Sin = 3, \mu_{\max} = 1, k_s = 2.5, \alpha = 1/2, \xi_{1:2} = [e, 6]$, and $N_{1:3} = [18]_3$.}
\label{fig:PsiDemo1}
\end{figure}

\begin{figure}[ht]
\centering
\includegraphics[scale=0.2]{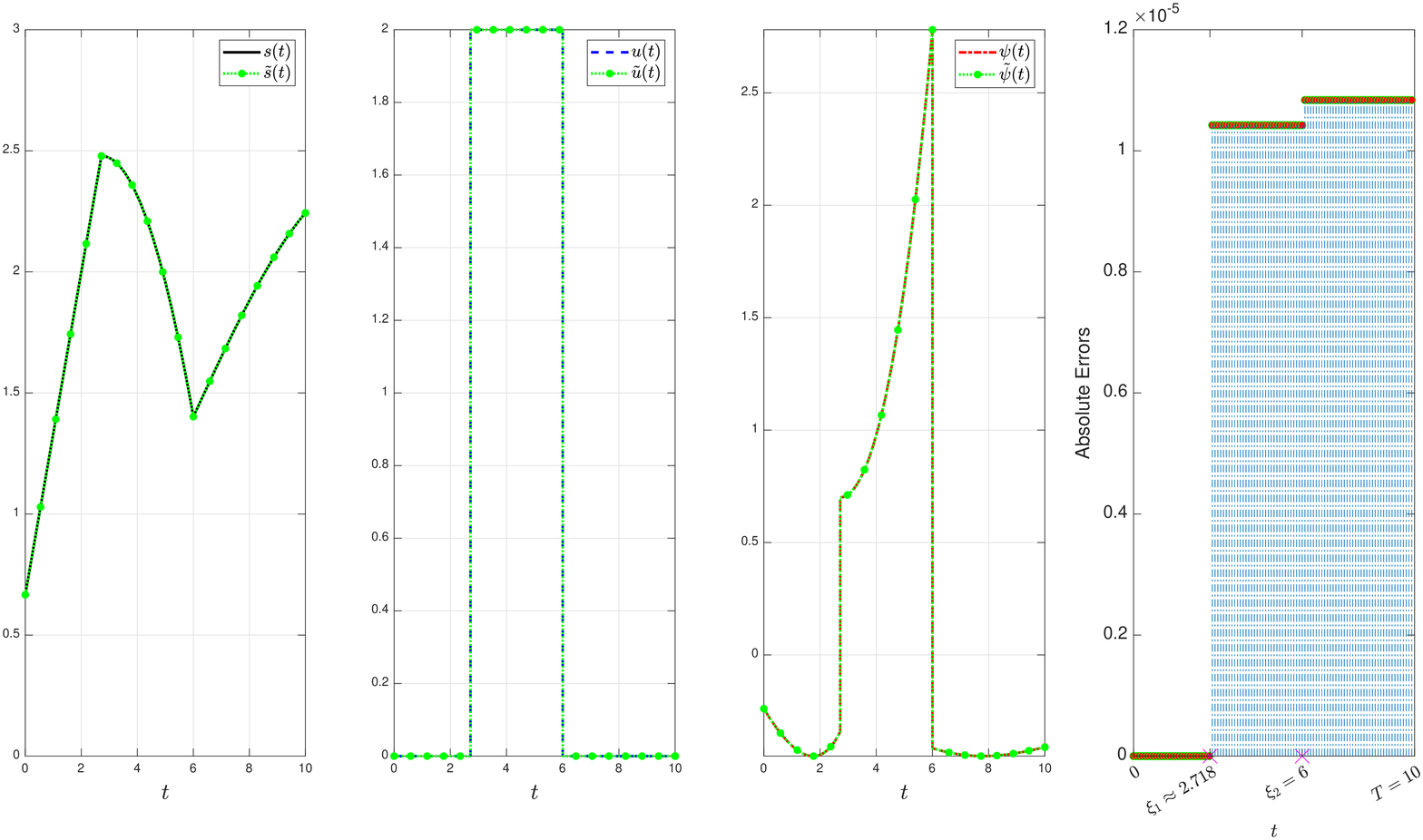}
\caption{Plots of $s, u, \psi$ and their perturbations (in green color) in addition to the errors plot of approximating $\C{I}_{\F{\Omega}_{x_{100,0: 99}}}^{(t)} \tilde \psi$ using the SG quadratures for the parameter values $T = 10, \Sin = 3, \mu_{\max} = 1, k_s = 2.5, \alpha = 1/2, \xi_{1:2} = [e, 6], \delta_{1:2} = [$-$1$E-$5, 1$E-$6]$, and $N_{1:3} = [18]_3$.}
\label{fig:PsiDemo2}
\end{figure}

\begin{figure}[ht]
\centering
\includegraphics[scale=0.3]{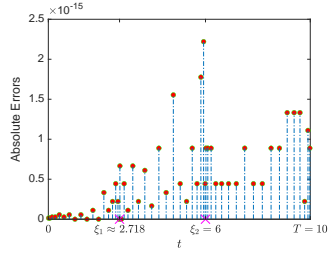}
\caption{Error plot of approximating $\C{I}_{\F{\Omega}_{\hat t_{N_k,0:N_k}^{(k),\alpha}}}^{(t)} \psi$ using the SG quadratures for the parameter values $T = 10, \Sin = 3, \mu_{\max} = 1, k_s = 2.5, \alpha = 1/2, \xi_{1:2} = [e, 6]$, and $N_{1:3} = [18]_3$.}
\label{fig:PsiDemo3}
\end{figure}

\begin{figure}[ht]
\centering
\includegraphics[scale=0.35]{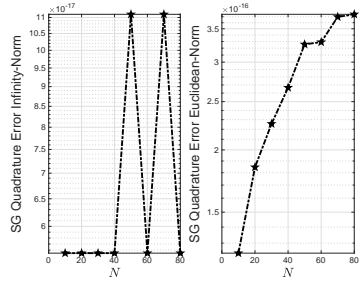}
\caption{Error infinity- and -Euclidean norms of the $2$-point SG quadrature with $\alpha = -0.1$ in log-lin scale of the square wave function $f_6$ when successively integrated over the intervals $\F{\Omega}_{x_{N,0: N-1}}\,\forall\,N = 10, 20, 40$, and $80$.}
\label{fig:FAccuracy4_fn6}
\end{figure}

\begin{figure}[ht]
\centering
\includegraphics[scale=0.2]{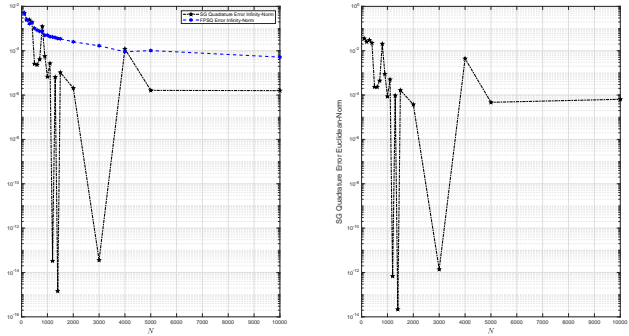}
\caption{Error infinity- and -Euclidean norms of the $2$-point SG quadrature with $\alpha = -0.1$ in log-lin scale of the reconstructed square wave function $\breve f_{6,N,400}$ when successively integrated over the intervals $\F{\Omega}_{x_{N,0: N-1}}$, for several increasing values of $N$. The left plot also shows the corresponding FPSQ error infinity-norms using FPSI matrices of size $N$.}
\label{fig:FAccuracy4SGG_fn6}
\end{figure}

\begin{table}[ht]
\caption{Observed absolute errors in the definite integrals of functions $f_6, \ldots, f_{12}$ when they are approximated by FPSQs using the reconstructed $\breve f_{6,N,M}, \ldots, \breve f_{12,N,M}$ obtained by Algorithm \ref{alg:2} with $M = 100$ and $M = 400$. All approximations are rounded to 5 significant digits.} % title of Table
\centering % used for centering table
\resizebox{1\columnwidth}{!}{%
\begin{tabular}{*{9}{c}} % centered columns (15 columns)
\toprule
 & $N$ & $\left\|\C{I}_{\bm{x}_N} f_6-{\C{I}_{\bm{x}_N} \breve f_{6,N,M}}\right\|_{\infty}$ & $\left\|\C{I}_{\bm{x}_N} f_7-{\C{I}_{\bm{x}_N} \breve f_{7,N,M}}\right\|_{\infty}$ & $\left\|\C{I}_{\bm{x}_N} f_8-{\C{I}_{\bm{x}_N} \breve f_{8,N,M}}\right\|_{\infty}$ & $\left\|\C{I}_{\bm{x}_N} f_9-{\C{I}_{\bm{x}_N} \breve f_{9,N,M}}\right\|_{\infty}$ & $\left\|\C{I}_{\bm{x}_N} f_{10}-{\C{I}_{\bm{x}_N} \breve f_{10,N,M}}\right\|_{\infty}$ & $\left\|\C{I}_{\bm{x}_N} f_{11}-{\C{I}_{\bm{x}_N} \breve f_{11,N,M}}\right\|_{\infty}$ & $\left\|\C{I}_{\bm{x}_N} f_{12}-{\C{I}_{\bm{x}_N} \breve f_{12,N,M}}\right\|_{\infty}$\\ [0.5ex] % inserts table
\midrule % inserts single horizontal line 
\multirow{20}{*}{\rotatebox{90}{$M = 100$}} & 100 & 7.7031E-03 & 2.8499E-03 & 6.4450E-03 & 1.6746E-02 & 3.9564E-03 & 3.2151E-02 & 1.5815E-02\\
& 200 & 3.0731E-03 & 5.6578E-03 & 4.1025E-02 & 2.4944E-02 & 1.0960E-02 & 2.6135E-02 & 2.2157E-02\\
& 300 & 1.6667E-03 & 5.0021E-03 & 1.9992E-02 & 1.8794E-02 & 7.1784E-03 & 2.5017E-02 & 2.1124E-02\\
& 400 & 1.2896E-03 & 7.4427E-03 & 2.0185E-02 & 3.6597E-02 & 4.2957E-03 & 4.7179E-02 & 2.2133E-02\\
& 500 & 1.0332E-03 & 4.3404E-03 & 1.8230E-02 & 2.4678E-02 & 6.3296E-03 & 3.2689E-02 & 2.8087E-02\\
& 600 & 8.9546E-04 & 5.8339E-03 & 2.3331E-02 & 2.8593E-02 & 5.2307E-03 & 3.7298E-02 & 2.4069E-02\\
& 700 & 6.8502E-04 & 4.6952E-03 & 1.8076E-02 & 2.3719E-02 & 5.6152E-03 & 3.8174E-02 & 2.3842E-02\\
& 800 & 5.9732E-04 & 5.0583E-03 & 2.3303E-02 & 2.5245E-02 & 5.1227E-03 & 3.5575E-02 & 2.4860E-02\\
& 900 & 5.5556E-04 & 4.9998E-03 & 2.0001E-02 & 2.3055E-02 & 5.9836E-03 & 3.8001E-02 & 2.7461E-02\\
& 1000 & 3.1066E-04 & 5.5535E-03 & 2.0603E-02 & 2.8201E-02 & 6.4653E-03 & 3.8083E-02 & 2.5833E-02\\
& 1100 & 4.5455E-04 & 4.6968E-03 & 1.9224E-02 & 2.4954E-02 & 5.2746E-03 & 3.9082E-02 & 2.0363E-02\\
& 1200 & 4.1667E-04 & 5.4165E-03 & 2.1667E-02 & 2.6671E-02 & 5.4272E-03 & 3.9515E-02 & 2.1901E-02\\
& 1300 & 3.6638E-04 & 4.8359E-03 & 1.8955E-02 & 2.4179E-02 & 5.5601E-03 & 3.9710E-02 & 2.2902E-02\\
& 1400 & 2.6414E-04 & 5.0957E-03 & 2.1246E-02 & 2.6127E-02 & 5.6021E-03 & 3.9737E-02 & 2.4127E-02\\
& 1500 & 3.1292E-04 & 5.0001E-03 & 2.0000E-02 & 2.4036E-02 & 5.1751E-03 & 4.0622E-02 & 2.5140E-02\\
& 2000 & 2.7553E-04 & 5.0712E-03 & 2.0691E-02 & 2.5050E-02 & 5.5265E-03 & 4.1313E-02 & 2.5945E-02\\
& 3000 & 1.5649E-04 & 5.1667E-03 & 2.0667E-02 & 2.5569E-02 & 5.3667E-03 & 3.7809E-02 & 2.5086E-02\\
& 4000 & 9.0703E-05 & 5.1131E-03 & 2.0131E-02 & 2.4526E-02 & 5.6293E-03 & 4.0266E-02 & 2.4683E-02\\
& 5000 & 9.3887E-05 & 5.0319E-03 & 2.0229E-02 & 2.5145E-02 & 5.4800E-03 & 3.8022E-02 & 2.4418E-02\\
& 10000 & 7.5134E-05 & 5.0287E-03 & 2.0062E-02 & 2.4810E-02 & 5.6274E-03 & 3.8548E-02 & 2.3900E-02\\
\midrule
\multirow{20}{*}{\rotatebox{90}{$M = 400$}} & 100 & 4.9996E-03 & 2.8670E-03 & 6.3275E-03 & 1.8210E-02 & 3.9583E-03 & 3.2113E-02 & 1.4180E-02\\
& 200 & 2.5001E-03 & 8.3199E-04 & 5.7223E-03 & 5.4374E-03 & 3.2109E-03 & 3.3962E-02 & 1.4262E-02\\
& 300 & 1.6449E-03 & 1.6646E-03 & 6.6748E-03 & 2.1666E-03 & 5.1743E-04 & 3.0778E-02 & 5.2880E-03\\
& 400 & 1.8965E-03 & 7.1559E-04 & 1.6532E-03 & 3.7499E-03 & 5.8589E-04 & 2.8770E-03 & 4.6719E-03\\
& 500 & 1.0256E-03 & 3.3118E-04 & 2.2930E-03 & 4.6999E-03 & 1.3704E-03 & 2.0028E-02 & 2.8211E-03\\
& 600 & 8.5638E-04 & 8.3282E-04 & 3.3354E-03 & 4.0511E-03 & 1.3538E-03 & 2.9835E-02 & 6.7314E-03\\
& 700 & 7.5520E-04 & 1.8381E-03 & 6.6682E-03 & 2.3560E-03 & 5.0315E-04 & 2.3221E-02 & 4.8724E-03\\
& 800 & 7.7339E-04 & 1.4188E-03 & 1.0400E-02 & 4.4291E-03 & 1.5498E-03 & 2.1083E-02 & 7.3088E-03\\
& 900 & 5.0073E-04 & 1.6669E-03 & 6.6658E-03 & 8.3327E-04 & 7.6506E-04 & 1.6931E-02 & 2.7026E-03\\
& 1000 & 5.0672E-04 & 1.2866E-03 & 4.6635E-03 & 1.5994E-03 & 9.7877E-04 & 1.9428E-02 & 5.1536E-03\\
& 1100 & 4.2765E-04 & 1.0556E-03 & 4.6855E-03 & 2.2419E-03 & 1.1314E-03 & 1.8328E-02 & 2.7325E-03\\
& 1200 & 4.1667E-04 & 1.2501E-03 & 4.9995E-03 & 1.9188E-03 & 4.7075E-04 & 1.1344E-02 & 4.6587E-03\\
& 1300 & 3.9110E-04 & 9.8973E-04 & 3.6005E-03 & 1.0884E-03 & 6.6346E-04 & 1.2153E-02 & 2.4482E-03\\
& 1400 & 3.5714E-04 & 1.5477E-03 & 6.5332E-03 & 5.3092E-04 & 8.8572E-04 & 2.0928E-02 & 4.6611E-03\\
& 1500 & 3.4382E-04 & 1.0001E-03 & 3.9997E-03 & 5.5662E-04 & 9.8209E-04 & 1.8945E-02 & 2.5315E-03\\
& 2000 & 2.5207E-04 & 1.0838E-03 & 4.5694E-03 & 3.0036E-04 & 6.4333E-04 & 1.3095E-02 & 3.7522E-03\\
& 3000 & 1.6667E-04 & 1.1667E-03 & 4.6666E-03 & 6.7591E-04 & 9.4595E-04 & 1.4990E-02 & 3.2590E-03\\
& 4000 & 8.9983E-05 & 1.3978E-03 & 5.1346E-03 & 7.0223E-04 & 6.5124E-04 & 1.2374E-02 & 3.0360E-03\\
& 5000 & 1.0163E-04 & 1.2332E-03 & 5.0293E-03 & 1.1999E-04 & 7.3058E-04 & 1.1753E-02 & 3.9238E-03\\
& 10000 & 5.1581E-05 & 1.2287E-03 & 4.8663E-03 & 2.5096E-04 & 8.2991E-04 & 1.2116E-02 & 3.6048E-03\\
\bottomrule %inserts single line
\end{tabular}
}
\label{tab:GGQf6to12} % is used to refer this table in the text
\end{table}

%\section*{Acknowledgements}
%
%\section*{Funding}
%%
%% ---------------------------------------------
%% References
%%
\bibliographystyle{model1-num-names}
\bibliography{Bib}
%% ---------------------------------------------
%%
\end{document}